\documentclass[12pt, reqno]{amsart}
\usepackage{amsmath, amsthm, amscd, amsfonts, amssymb, graphicx, color}
\usepackage{mathtools}
\usepackage[bookmarksnumbered, colorlinks, plainpages]{hyperref}
\usepackage[vcentermath]{youngtab}
\usepackage{amsmath, amssymb,ytableau,young,bm,stackrel,textcomp,amsxtra, graphicx}
\usepackage{amsmath, amssymb,amsxtra, graphicx, float,tikz,ytableau,young,bm,stackrel,textcomp,}
\usepackage{bold-extra}
\usepackage{float}
\usepackage{courier}
\usepackage{outlines}
\usepackage{tikz}
\usetikzlibrary{matrix,positioning,calc,fit,calc,shapes.geometric}
\usepackage{xcolor}
\usepackage{color}
\usepackage{tabularx}
\usepackage{tikz,pgf}
 \usetikzlibrary{calc} 
\usepackage{amsthm}
\newtheorem{definition}{Definition}[section] 
\newtheorem{theorem}{Theorem}[section]
\newtheorem{lemma}[theorem]{Lemma}
\newtheorem{proposition}[theorem]{Proposition}
\newtheorem{corollary}[theorem]{Corollary}
\newtheorem{example}[theorem]{Example}

\theoremstyle{remark}
\newtheorem{remark}[theorem]{Remark}

\numberwithin{equation}{section}

\newcommand{\PTab}{\text{PTab}}

\newcommand {\PCW}{\mathsf{PCW}}

\newcommand{\BW}{\mathsf{BW}}
\newcommand{\std}{\mathsf{std}}
\newcommand{\comm}[1]{}
\newcommand {\mmin}{\text{min}}
\newcommand {\mmax}{\text{max}}

\newcommand {\NULL}{\text{NULL}}

\newcommand{\wt}{\mathsf{wt}}
\newcommand{\DE}{\text{DE}}

\title{Combinatorial models for Quasicrystals and Schur-Weyl Duality}

\author{Glenn Appleby}
\address{Department of Mathematics \& Computer Science \\ Santa Clara University \\ Santa Clara, CA 95053, USA}
\email{gappleby@scu.edu}

\author{Mary Claire Simone}
\address{Department of Mathematics \\ Skagit Valley College \\ Mount Vernon, WA 98273}
\email{maryclaire.simone@skagit.edu}

\author{Tamsen Whitehead McGinley}
\address{Department of Mathematics \& Computer Science \\ Santa Clara University \\ Santa Clara, CA 95053, USA}
\email{tmcginley@scu.edu}

\begin{document}
\raggedbottom

\maketitle

\medskip

\begin{abstract}
    Quasicrystals were defined by Cain and Malheiro and Maas-Gariépy \cite{CM17, FMG23} to model the decomposition of Schur polynomials into fundamental quasisymmetric functions.  We define a new combinatorial model for crystal graphs of type $A_{n-1}$, called \emph{parsed content words}, for which the decomposition into quasicrystals is particularly evident.  We use the model to describe a combinatorial interpretation of Schur-Weyl duality and recover the classical combinatorial proof of the well-known expansion of homogeneous symmetric functions into Schur functions within our framework. 
\end{abstract}

\section{Introduction}

Representations of Lie groups and Lie algebras have been analyzed by crystal graphs (\cite{hong2002introduction}, \cite{BS}, \cite{Kashiwara91}), which have been modeled by various combinatorial objects. We study combinatorial objects for both type $A_{n-1}$ crystal graphs and representations of symmetric groups. Irreducible representations for both are indexed by partitions $\lambda$. Naturally, efforts to identify good combinatorial models for irreducible representations focused on those of highest weight $\lambda$ (for $\mathfrak {gl}_n$) and shape $\lambda$ (for symmetric groups).

Many problems in combinatorial representation theory are challenging due to the underlying structures being studied. Sometimes, however, difficulties arise from striving to overcome complexities imposed by the choice of model itself, and not the theory for which it serves as a combinatorial proxy (we use the term \emph{proxy} to avoid referring to a model as a \emph{representation}.) If, given $\lambda$, a model \emph{only} provides  \emph{one} model for each isomorphic irreducible representations highest weight/shape $\lambda$, it might not be well-suited to problems connecting multiple isomorphic but otherwise distinct representations simultaneously (such as in Schur-Weyl duality), or might not be suited to problems devoted to the internal decomposition of representations (such as the theory of quasicrystals) for families of isomorphic representations. We seek to address this here. 

Fix integers $n$ and $k$ with $n \geq k$. Let $V$ be a fixed vector space of dimension $n$, and its $k$-fold tensor product $V^{\otimes k}$. Representations on $V^{\otimes k}$ can be modeled by graphs whose nodes are words of length $k$ in the alphabet $[n]$, denoted $[n]^{\otimes k}$ (called \emph{crystals on words} (see~\ref{subsection_word_crystals})).

Semistandard Young tableaux (SSYT) of shape $\lambda$ have often been used as a combinatorial model for the nodes of a crystal graph, serving as a proxy for the collection of isomorphic (yet distinct) irreducible crystals of highest weight $\lambda$. This model suffers from the difficulties described above. For example, it does not allow the combinatorial analysis of tensor products (Littlewood-Richardson Rule). 
In particular, in the crystal on words in $[n]^{\otimes k}$, only the \emph{distinguished crystal} of highest weight $\lambda$ (see Definition \ref{Def:distinguished crystal}) contains words that are the reading words of SSYT.  Other distinct irreducible crystals of highest weight $\lambda$ are only identified up to isomorphism when we restrict our attention to the SSYT model. So while the distinguished crystal, and its SSYT analogue, are important, they are not sufficient for all purposes.

 The theory of \emph{perforated tableaux} (or \emph{ptableaux}) was developed (see~\cite{AW2022,AW2023}) to analyze multiple $\mathfrak{gl}_n$ representations and rectify these problems.  The decomposition of $V^{\otimes k}$ into irreducible $\mathfrak{gl}_n$ representations is naturally modeled by ptableaux, and distinct crystals of words are identified with distinct but demonstrably isomorphic irreducible ptableaux crystals. A natural tensor product of ptableaux provided a generalization of the Littlewood-Richardson rule within this model, generalizing the classical theory of Littlewood-Richardson fillings. A version of the RSK algorithm was developed for ptableaux, and its role in classifying isomorphic irreducible crystals was explicated \cite{AW2023}.

In this paper, we use a crystal isomorphism between ptableaux and \emph{biwords}, finding the biword model more convenient for our purposes (though relying on previous perforated tableaux results for key steps in our proofs). We provide background and definitions and prove a simple, but key result relating Knuth equivalence to dual (Knuth) equivalence of words appearing in biwords (Theorem~\ref{theorem biword dual Knuth}). We use biwords to define a crystal structure on \emph{parsed content words} (Definition \ref{Def:parsed word}), which allows us to identify quasicrystals and their associated quasisymmetric functions more immediately than in either the SSYT model or the words model by circumventing the standardization algorithm.  Quasicrystals were originally studied by Cain--Malheiro \cite{CM17} and Maas-Gariépy \cite{FMG23}, and later analyzed by \cite{CMRR25}, \cite{CMRR26}, \cite{BCDS}, and \cite{BDMS} as tools to find the Schur function expansion of a symmetric function whose fundamental quasisymmetric function expansion is known, which includes LLT polynomials, modified Macdonald polynomials \cite{HHL2005}, characters of higher Lie modules \cite{GR1993}
and the plethysm of two Schur functions \cite{LW12}. 

We then apply these results from biwords and parsed content words models to give a combinatorial (graph-theoretic) analysis of Schur-Weyl duality in type $A_{n-1}$. In this setting, Theorem \ref{Thm:S-W Duality} shows that Schur-Weyl duality of $V^{\otimes k}$ can be combinatorially identified as the decomposition of $[n]^{\otimes k}$ into families of graphs indexed by $\lambda \vdash k$ isomorphic to a Cartesian product of an irreducible crystal graph of highest weight $\lambda$ and a dual equivalence graph of shape $\lambda$.  Finally, we include a proof of the Schur-expansion of the complete homogeneous symmetric functions via biwords. 

\section{Biwords}\label{Sec_Biwords}
 \subsection{Definitions}

This paper uses \emph{biwords} to compute various invariants and properties of representations, including crystals (in the representation of $\mathfrak{gl}_n$ or its Lie algebras) and representations of symmetric groups $S_k$. 

 \begin{definition} \label{Def:biword} A \emph{biword} is a two-rowed array of positive integers
\[ \left( \begin{matrix} \kappa \\ \rho \end{matrix} \right) = \left(\begin{matrix} \kappa_1&  \kappa_2 & \cdots &\kappa_k \\
\rho_1&  \rho_2 & \cdots &\rho_k \end{matrix}\right) \]
with the property that entries in the top row $\kappa$ are weakly increasing ($\kappa_i\le  \kappa_{i+1}$ for $i=1,2,\dots k-1$), and beneath every constant factor of $\kappa$ the entries in $\rho$ are \emph{weakly decreasing}. 
\end{definition}

Throughout the paper, we fix $[n]^{\otimes k}$ (consisting of words of length $k$ in the alphabet $[n]=\{1,2,\dots, n\}$). Let $\BW_{n,k}$ denote the subset of biwords such that the bottom row is a word in $[n]^{\otimes k}$. 

For example, the two-rowed array:
\[ \left( \begin{matrix} \kappa \\ \rho \end{matrix} \right) = \left( \begin{matrix} 1111222333334 \\ \, 
4431532664112 \end{matrix} \right)\]
is a biword in $\BW_{6,13}$.

Many works (\cite{fulton1997young},\cite{BS}) require that the bottom row $\rho$ of a biword be weakly \emph{increasing} under constant factors of $\kappa$; we have reversed that ordering to be compatible with our constructions of perforated tableaux \cite{AW2022} (described below) and our version of the RSK algorithm defined on them \cite{AW2023}.

We associate biwords to \emph{perforated tableaux} (see~\cite{AW2022,AW2023}), defined in Section~\ref{crystals}. The familiar semistandard Young tableau (SSYT) will be a special case of this construction.  For now, we record our intended use of the two rows in a biword: we let the \emph{top} row $\kappa$ of a biword $\kappa \choose \rho$ specify the \emph{content} of a tableau or perforated tableau and the \emph{bottom} row $\rho$ indicate the \emph{row} in the (perforated) tableau in which the content appears. That is, we have the following interpretation:

\[ \left( \begin{matrix} \kappa \\ \rho \end{matrix} \right) = \left(\begin{matrix} \kappa_1&  \kappa_2 & \cdots &\kappa_k \\
\rho_1&  \rho_2 & \cdots &\rho_k \end{matrix}\right)=\left( \begin{matrix}
    \hbox{Content} \\ \hbox{Row}
\end{matrix} \right). \]

\begin{definition}
    Given some biword, the top row (typically denoted $\kappa$)  is the the \emph{content word} of the biword, and the bottom row (typically denoted $\rho$) is the \emph{row word} of the biword. We define the \emph{content} of a biword $\kappa \choose \rho$ as the multiset of numbers appearing in its content word $\kappa$.
\end{definition}
In the depiction of a biword as a tableau (described in Section \ref{crystals} below) we will see that the content of the biword and the content of its associated tableau coincide.

\subsubsection{From Words to Biwords.} \label{Subsection: WordstoBi}

To generate biwords $\kappa \choose \rho$ from a given word $\rho \in [n]^{\otimes k}$, there are several choices of content words $\kappa$.

Given a word
$\omega = \omega_{1}\omega_{2} \cdots \omega_{k-1} \omega_k \in [n]^{\otimes k},$ let ${\mathcal P}_\mmin=\{i : \omega_i<\omega_{i+1}.\}$,  often called the \emph{ascent set} of $\omega.$ Let ${\mathcal P}$ be a multiset of $\{0,1,\ldots, k\}$ such that ${\mathcal P}_\mmin\subseteq {\mathcal P}.$ Let $p_i$ be the number of $i$'s in $\mathcal P$.  With this, we have the following.
\begin{definition}\label{Def:parsed word}
A \emph{parsed word of $\omega$}, $\omega_{\mathcal P}$, is the word $\omega$ with $p_i$ \emph{parsing bars} ``$|$'' inserted after $\omega_i$ (and $p_0$ parsing bars before $\omega_1$.)  We say that $\mathcal P$ is a \emph{valid parsing} of $\omega.$  We call ${\mathcal P}_\mmin$ the \emph{minimal parsing} of $\omega$, with corresponding parsed word $\omega_{{\mathcal P}_\mmin}$. Any valid parsing of $\omega$ is a \emph{refinement} of ${\mathcal P}_\mmin$.
\end{definition}
Note that a valid parsing $\mathcal P$ separates $\omega$ into a sequence of (possibly empty) weakly decreasing factors. 

\begin{example} If
\[ \omega=4433442233344411122133444 \in [n]^{\otimes 25}, \]
then ${\mathcal P}_\mmin=\{4,8,11,17,20,22\}$ so 
\[  \omega_{{\mathcal P}_{\mmin}}=4433|4422|333|444111|221|33|444. \]
If ${\mathcal P}= \{\!\{4,8,11,13,17,20,22,22\}\!\},$ then
\[ \omega_{\mathcal P} =4433|4422|333|44|4111|221|33||444 \]
(the ``$\, ||$" above denotes an empty factor).  
\end{example}

\begin{definition}\label{def:biword of a parsed word}
    The \emph{biword associated to a parsed word $\omega_{\mathcal P}$} is the biword $\kappa \choose \omega$ with $\kappa$ the content word containing only $i$'s over the $i$th factor of $\omega_{\mathcal P}$.
\end{definition}

\begin{example}
    In the previous example $\omega_{\mathcal P} =4433|4422|333|44|4111|221|33||444$, the biword associated to $\omega_{\mathcal P}$ would be:
    \[ \left( \begin{matrix} \kappa \\ \omega \end{matrix} \right) = \left( \begin{matrix} 1111222233344555566677999 \\ \, 
4433442233344411122133444 \end{matrix} \right).\]
\end{example}

In this paper we are primarily interested in biwords associated to two sorts of parsings.

\begin{definition}\label{defn:minimal_biword}
        Given a word $\rho \in [n]^{\otimes k}$, the \emph{minimal biword associated to $\rho$} is the biword $\kappa \choose \rho$ associated with the minimally parsed row word $\rho_{{\mathcal P}_{\mmin}}$ consisting of $1$'s over the first maximal weakly decreasing factor of $\rho$, $2$'s over the second maximal weakly decreasing factor of $\rho$, etc.
    \end{definition}

\begin{example}
    Let $\rho= 4431532664112$ with minimal parsing $\omega_{{\mathcal P}_{\mmin}}= 4431|532|66411|2$. The (unique) minimal biword associated to $\rho$ is:
\[ \left( \begin{matrix} \kappa \\ \rho \end{matrix} \right) = \left( \begin{matrix} 1111222333334 \\ \, 
4431532664112 \end{matrix} \right).\]
\end{example}

The minimal biword associated to a word $\rho \in [n]^{\otimes k}$ will be used to analyze models for crystals of a fixed highest weight and, in particular, their associated quasicrystals. See Section~\ref{section: quasicrystals}.

\medskip
To aid our analysis of Schur-Weyl duality (Section \ref{Section: S-W duality}), we also consider a different parsing of $\rho$:
\begin{definition}\label{def:maximal-biword}
The  \emph{maximal parsing} of the word $\rho = \rho_1 \rho_2\cdots \rho_k \in [n]^{\otimes k}$ is the set $\{1, 2, \ldots, k-1\}$. The \emph{maximal biword associated to $\rho$} is
     \[ \left( \begin{matrix} \kappa \\ \rho \end{matrix} \right) = \left(\begin{matrix} \kappa_1&  \kappa_2 & \cdots &\kappa_k \\
\rho_1&  \rho_2 & \cdots &\rho_k \end{matrix}\right) = \left(\begin{matrix} 1&  2 & \cdots &k \\
\rho_1&  \rho_2 & \cdots &\rho_k \end{matrix}\right).\]
\end{definition}

\begin{remark}
    For a word $\omega \in [n]^{\otimes k}$, the \emph{minimal parsing} of $\omega$ is the unique parsing of $\omega$ into the fewest possible number of weakly decreasing factors. In the associated minimal biword, the number of distinct letters appearing in the content word is minimal. The \textit{maximal parsing} of $\omega$ is a parsing of $\omega$ into factors of length $1$, and hence in the associated maximal biword, the number of distinct letters appearing in the content word is maximal.
\end{remark}

There are many content words $\kappa$ obtained from different parsings of $\rho$ lying between these two extremes (minimal and maximal parsings): given any parsing $\rho_{\mathcal P}$ of a row word $\rho$ into weakly decreasing factors (not necessarily of maximal length), we can construct a
biword $\kappa \choose \rho$ associated to $\rho_{\mathcal P}$ where the content word $\kappa$ is defined by putting all $1$'s above the first weakly decreasing factor of $\rho$ (relative to the given parsing), then all $2$'s over the second factor, etc. 
For example, the biword below associated to $\rho_{\mathcal P}=443|1|53|2|66|41|1|2$ is neither minimally nor maximally parsed:
\[ \left( \begin{matrix} 1112334556678 \\ \, 
4431532664112 \end{matrix} \right).\]

Such a parsing of $\rho$ is a \emph{refinement} of its minimal parsing.  We do not make much use of these parsings, although most of our results in Section~\ref{section: Parsed word quasicrystal} work for any parsing, and we present a result that uses a particular non-extreme parsing (neither minimal nor maximal) in Section \ref{homogeneous}.

    \subsection{Dual Biwords, Knuth equivalence, Dual (Knuth) Equivalence}

    \emph{Knuth equivalence} is a relation on words. We adopt the version used in \cite{Baker2000} and \cite{Lecouvey2002} that matches the Kashiwara convention for computing crystal operators used in this paper. This version of Knuth equivalence maintains the property that Knuth equivalent words have the same insertion tableau under the RSK algorithm.

    The notion of \emph{dual equivalence} (or dual Knuth equivalence) appeared in the work of Haiman \cite{Haiman1992} and was used by Assaf in \cite{Assaf07thesis, Assaf08fpsac-crystalops, Assaf15} to construct graphs. Through insertion schemes in RSK, dual equivalence is realized by tableaux and certain tableaux ``dual'' to them. There are two senses in which we proceed differently. First, we use a version of dual equivalence that corresponds to our version of Knuth equivalence (Definition \ref{Def:Knuth Equiv}). We also interpret these constructions in terms of associated biwords, not only tableaux constructed from them. This shift yields an  elementary proof of Theorem~\ref{theorem biword dual Knuth}, which relates Knuth and dual equivalences through a natural duality on biwords. 

    \begin{definition} \label{Def: dual biword}
        Let $\kappa \choose \rho$ be a biword (see Def. \ref{Def:biword}). We construct the biword \emph{dual} to $\kappa \choose \rho$, denoted $\widehat{\kappa \choose \rho}$, by swapping the rows of $\kappa \choose \rho$ and re-ordering its columns so that $\widehat{\kappa \choose \rho}$ is itself a biword.

We denote the dual of a biword $\kappa \choose \rho$ as
\[ \widehat{\left( \begin{matrix} \kappa \\ \rho \end{matrix} \right)} = \left( \begin{matrix} \rho^* \\ \widehat{\kappa}_{\rho} \end{matrix} \right),\]
where $\rho^*$ is the word made from the letters of $\rho$, written in weakly increasing order and the word $\widehat{\kappa}_{\rho}$ is uniquely determined by the condition that it is weakly decreasing under constant factors of $\rho^*$.
\end{definition}

For example, if 
\[ \left( \begin{matrix} \kappa \\ \rho \end{matrix} \right) = \left( \begin{matrix} 1111222333334 \\ \, 
4432532664111 \end{matrix} \right),\text{ then }
\left( \begin{matrix} \rho \\ \kappa \end{matrix} \right) = \left( \begin{matrix} 4432532664111 \\ \, 
 1111222333334\end{matrix} \right),\]
and
\[ \widehat{\left( \begin{matrix} \kappa \\ \rho \end{matrix} \right)} = 
\left( \begin{matrix} \rho^* \\ \widehat{\kappa}_{\rho} \end{matrix} \right)=
\left( \begin{matrix} 1112233444566 \\ \, 
4332121311233 \end{matrix} \right).\]

\medskip

\begin{definition}\label{Def:Knuth Equiv} \cite{Baker2000, Lecouvey2002}
    Let $\rho$ and $\rho'$ be two words of length $k$. Then $\rho$ is \emph{elementarily Knuth equivalent} to $\rho'$ if: 

 Case I. The word $\rho$ contains the factor $cab$ with $a<b \leq c$ and the word $\rho'$ contains the factor $acb$ in the same locations.  Otherwise $\rho$ and $\rho'$ are identical. If $a$ is the $i$-th letter of $\rho$, we denote this by $\rho'=K^{-}_{i}\rho$, indicating that we have switched the $i$-th letter of $\rho$ with the $(i-1)$-st letter via valid Knuth equivalence, and vice versa $K^{-}_{i} \rho' =\rho$. 

 \medskip
 \noindent Or
 \medskip

 Case II. The word $\rho$ contains the factor $bac$ with $a\leq b < c$, and the word $\rho'$ contains the factor $bca$ in the same locations.  Otherwise $\rho$ and $\rho'$ are identical. If $a$ is the $i$-th letter of $\rho$, we denote this by $\rho'=K^{+}_{i}\rho$, indicating that we have switched the $i$-th letter of $\rho$ with the letter to its immediate right under a valid Knuth equivalence, and vice versa $K_i^+\rho'=\rho$.

 If $\rho$ and $\rho'$ are related by a single swap of entries as described above (i.e. $\rho'=K_{i}^{\pm}\rho$),  we say they are \emph{elementarily Knuth equivalent}. If they are related by a sequence of elementary Knuth equivalences, we say they are \emph{Knuth equivalent}, denoted by:
 \[ \rho \equiv_{K} \rho'.\]
\end{definition}

We now define \emph{dual equivalence} for permutations. While the relation is equivalent to that found in \cite{Assaf15}, our notation differs somewhat. The need for the change will be apparent when we consider operators on graphs in Section~\ref{section dual equiv}.

\begin{definition}\label{Def:Dual equiv}
Suppose $\kappa$ and $\kappa'$ are two permutations on $[k]$ written in one-line notation.
Then $\kappa$ and $\kappa'$ are \emph{elementarily dual equivalent} if either:
    \medskip

   Case I.  
    \[\kappa = (\ldots i \ldots i+1 \ldots  i-1 \ldots ) , \ \ \hbox{and} \ \  \kappa' = (\ldots i-1 \ldots i+1 \ldots  i \ldots ), \]
but otherwise all entries of $\kappa$ and $\kappa'$ are the same. We denote this involution by $D_{i}^{-}\kappa = \kappa'$ because $D_i^-$ swaps the letters $i$ and $i-1$. 

\medskip
\noindent Or,
\medskip

Case II. 
 \[\kappa = (\ldots i+1 \ldots i-1 \ldots  i \ldots )  \ \ \hbox{and} \ \ \kappa' = (\ldots i \ldots i-1 \ldots  i+1 \ldots ). \]
 but otherwise all entries of $\kappa$ and $\kappa'$ are the same. In this case we denote the involution by $D_{i}^{+}\kappa = \kappa'$ because $D_i^+$ swaps the letters $i$ and $i+1$.
 
In either case, the individual maps $D_{i}^{+}$ or $D_{i}^{-}$ are \emph{elementary dual equivalence operators}.

If $\kappa$ and $\kappa'$ are related by a sequence of elementary dual equivalence operators, we write
\[ \kappa \equiv_{D} \kappa',\] 
and say that $\kappa$ and $\kappa'$ are \emph{dual equivalent.}

We \emph{extend} the definition of the involutions $D_{i}^{-}$ and $D_{i}^{+}$ to \emph{all} (one-rowed) permutations in $S_k$ by declaring $D_{i}^{\pm} (\rho) = \rho$ if $\rho$ does not satisfy the hypotheses of Case I or II above.
\end{definition}

\begin{theorem}\label{theorem biword dual Knuth}
   Suppose $\kappa \choose \rho$ and $\kappa \choose \rho'$  are maximal biwords (so $\kappa = 123\cdots k$). Then the row words $\rho$ and $\rho'$ are Knuth equivalent if and only if in their respective dual biwords $\left( \begin{matrix} \rho^* \\ \widehat{\kappa}_{\rho} \end{matrix} \right)$ and $\left( \begin{matrix} (\rho')^* \\ \widehat{\kappa}_{\rho'} \end{matrix} \right)$ the words $\widehat{\kappa}_\rho$ and $\widehat{\kappa}_{\rho'}$ are dual equivalent. That is,

\[ \rho \equiv_{K} \rho' \Longleftrightarrow \widehat{\kappa}_{\rho}  \equiv_{D} \widehat{\kappa}_{\rho'}  ,\] 
and, in particular,
\[ K_{i}^{\pm} \rho =\rho' \Longleftrightarrow D_{i}^{\pm} \widehat{\kappa}_{\rho}  = \widehat{\kappa}_{\rho'}  .\] 
\end{theorem}

This theorem generalizes \cite[Equation 2.18]{R13}, which is restricted to permutations. The two-rowed array of a permutation $\omega$ is a biword whose dual is the two-rowed array for $\omega^{-1}$. 

\begin{example}
Let $\kappa \choose \rho$ and $\kappa \choose \rho'$ be the two maximal biwords below. Note that we have the elementary Knuth equivalence $\rho'=K_{5}^{-}\rho$ (see the circled letters):
\emph{\[ \left( \begin{matrix} \kappa \\ \rho \end{matrix} \right) =\left( \begin{matrix} 123\hbox{\boxed{4}}\hbox{\boxed{5}}\hbox{\boxed{6}}789 \\ \, 
311\hbox{\textcircled{3}}\hbox{\textcircled{4}}\hbox{\textcircled{4}}252 \end{matrix} \right) ,\]}
  and
\emph{\[ \left( \begin{matrix} \kappa \\ \rho' \end{matrix} \right) =\left( \begin{matrix} 123\hbox{\boxed{4}}\hbox{\boxed{5}}\hbox{\boxed{6}}789 \\ \, 
311\hbox{\textcircled{4}}\hbox{\textcircled{3}}\hbox{\textcircled{4}}252 \end{matrix} \right),\]}

In the dual biwords
\emph{\[ \left( \begin{matrix} \rho^* \\ \widehat{\kappa}_{\rho} \end{matrix} \right) =\left( \begin{matrix} 1122\hbox{\textcircled{3}}3\hbox{\textcircled{4}}\hbox{\textcircled{4}}5 \\ \, 
3297\hbox{\boxed{4}}1\hbox{\boxed{6}}\hbox{\boxed{5}}8 \end{matrix} \right) \]}
and
\emph{\[ \left( \begin{matrix} (\rho')^* \\ \widehat{\kappa}_{\rho'} \end{matrix} \right) =  \left( \begin{matrix} 1122\hbox{\textcircled{3}}3\hbox{\textcircled{4}}\hbox{\textcircled{4}}5 \\ \, 
3297\hbox{\boxed{5}}1\hbox{\boxed{6}}\hbox{\boxed{4}}8 \end{matrix} \right),\]}
$\widehat{\kappa}_{\rho'}$ is obtained by applying the elementary dual equivalence operator $D_{5}^-$ to $\widehat{\kappa}_{\rho}$ (see the boxed letters).
\end{example}

\begin{proof}
 It suffices to prove, for maximal biwords $\kappa \choose \rho$ and $ \kappa \choose \rho'$, that $\rho$ and $\rho'$ are elementarily Knuth equivalent if and only if $\widehat{\kappa}_\rho$ and $\widehat{\kappa}_{\rho'}$ are elementarily dual equivalent. 

Case 1. Suppose that two maximal biwords are (Def. \ref{Def:Knuth Equiv}, Case I) Knuth equivalent.  Thus 

\[ \left( \begin{matrix} \kappa \\ \rho \end{matrix} \right)  =\left(\begin{array}{ccccc} 
  \ldots &  i-1 &  i  & i+1& \ldots \\
   \ldots  &  c & a  &   b  &\ldots
\end{array} \right) \]
where $a < b \leq c$, and

\[ \left( \begin{matrix} \kappa \\ \rho' \end{matrix} \right)  = \left(\begin{array}{ccccc} 
  \ldots &  i-1 &  i  & i+1& \ldots \\
   \ldots  &  a &  c   &   b  &\ldots
\end{array} \right) , \]
with all other entries of $\rho$ and $\rho'$(not represented) equal.

We compute the dual biwords of the above and obtain

\[ \left( \begin{matrix} \rho^* \\ \widehat{\kappa}_{\rho} \end{matrix} \right)  =  \left( \begin{array}{ccccccc}   \ldots  &  a & \ldots & b  & \ldots &   c  &\ldots \\
  \ldots &  i & \ldots & i+1 &\ldots & i-1& \ldots
\end{array} \right), \]
and
\[\left( \begin{matrix} (\rho')^* \\ \widehat{\kappa}_{\rho'} \end{matrix} \right)   = \left( \begin{array}{ccccccc}   \ldots  &  a & \ldots & b  & \ldots &   c  &\ldots \\
  \ldots &  i-1 & \ldots & i+1 &\ldots & i& \ldots
\end{array} \right). \]

The fact that $a<b$ ensures that $\widehat{\kappa}_\rho$ and $\widehat{\kappa}_{\rho'}$ are decreasing under constant factors of $\rho^*$ and $(\rho')^*$. 
Then $D_i^-\widehat{\kappa}_{\rho} = \widehat{\kappa}_{\rho'}$ and so the permutations $\widehat{\kappa}_{\rho}$ and $\widehat{\kappa}_{\rho'}$ are elementarily dual equivalent.  

To prove the converse, assume the permutations $\widehat{\kappa}_{\rho}$ and $\widehat{\kappa}_{\rho'}$ appearing in the dual biwords above are elementarily dual equivalent, specifically that $D_i^-\widehat{\kappa}_{\rho} = \widehat{\kappa}_{\rho'}$. We need only argue that, in the original biwords $\kappa \choose \rho$ and $\kappa \choose \rho'$, we must have $a< b \leq c$, implying $\rho$ and $\rho'$ are (Case I) elementarily Knuth equivalent. However, since in  $\rho^* \choose \widehat{\kappa}_{\rho}$ there is necessarily an ascent in the bottom row going from $i$ to $i+1$, there is necessarily an ascent going from $a$ to $b$ in the entries above, showing that $a<b$. The inequality $b \leq c$ follows because  $\rho^* \choose \widehat{\kappa}_{\rho}$ is a valid biword. Hence $\rho$ and $\rho'$ are (Case I) elementarily Knuth equivalent. 

Case 2 follows by virtually identical arguments using $D_i^+$.
\end{proof}

The corollary below follows immediately from Theorem \ref{theorem biword dual Knuth}, but will have important consequences in the analysis of Schur-Weyl duality (Section~\ref{Section: S-W duality}).
\begin{corollary}
    Let $\rho \in [n]^{\otimes k}$ and suppose $\sigma_{i}\in S_k$ is the transposition $(i-1,i)$ acting on the left on words by swapping the entries in positions $i$ and $(i-1)$.  If $\kappa \choose \rho$ is the maximal biword associated to $\rho$, then $K_{i}^{\pm} \rho = \rho'$ if and only if $\sigma_i \widehat{\kappa}_{\rho} =\widehat{\kappa}_{\rho'}$.

    In particular, when $D_{i}^{\pm}$ acts non-trivially on a one-rowed permutation $\widehat{\kappa}_{\rho}$, we have $\sigma_i \widehat{\kappa}_{\rho} = D_{i}^{\pm} \widehat{\kappa}_{\rho}$.
\end{corollary}
\begin{proof}
     We obtain a word that is Knuth equivalent to $\rho$ by swapping two adjacent letters (in positions $i-1$ and $i$) of $\rho$. By Theorem \ref{theorem biword dual Knuth}, the one-rowed permutation $\widehat{\kappa}_{\rho}$ is elementarily dual equivalent to $\widehat{\kappa}_{\sigma_i \rho}$, and is obtained by swapping the positions of $i-1$ and $i$ in $\widehat{\kappa}_{\rho}$.  This is precisely the effect of left multiplication of a one-rowed permutation by the transposition $\sigma_i$.
\end{proof}

Note that this corollary does \emph{not} say that elementary dual equivalence operators simply reduce to the action of elementary transpositions on the underlying permutations. Rather, dual equivalence establishes the \emph{conditions} under which this is true, and the action is trivial otherwise. This will be interpreted in Section~\ref{Section: S-W duality} as the appropriate analogue in the setting of $S_k$ representations to the role of crystal operators for $\mathfrak{gl}_{n}$ representations. Crystal operators perform idealized Lie algebra actions on words serving as proxies for the actual elements of the Lie algebra, preserving enough of the structure of these actions to allow us to compute important invariants of representations without carrying the full (and more complicated) action of the Lie algebra. Similarly, we claim that dual equivalence operators serve as an idealized proxy for the full action of $S_k$ on the tensor product $V^{\otimes k}$, where we let dual equivalence operators on words in $[n]^{\otimes k}$ serve as the appropriate discrete analogue of the action of $S_k$ on $V^{\otimes k}$ via permutation of coordinates.

The following is immediate by duality of biwords.
\begin{corollary}
    Suppose $\kappa \choose \rho$ and $\kappa \choose \rho'$ are biwords with $\rho$ and $\rho'$ permutations on $[k]$. Then 
    \[ \rho \equiv_{D} \rho' \Longleftrightarrow \widehat{\kappa}_{\rho}  \equiv_{K} \widehat{\kappa}_{\rho'}  .\] 
\end{corollary}

\section{Crystals}\label{crystals}
    \subsection{Crystal Graphs} Crystal graphs are directed graphs that model representations of Lie groups or Lie algebras.  Appearing in the work of Kashiwara \cite{Kashiwara91, Kashiwara.1990}, Lusztig \cite{Lusztig.1990} and Littelman \cite{Littelmann1995}, crystal graphs have connections to classical combinatorial theory including the Robinson-Schensted-Knuth correspondence,  Schur-Weyl duality, and symmetric functions. These graphs have the following features:
    \begin{itemize}
        \item \emph {Nodes} of the graph. To each node we associate a \emph {weight} (a composition).
        \item \emph {Directed edges} labeled by crystal operators $f_i$ (called lowering operators) and $e_i$ (raising operators) for $i=1,2\dots,n-1,$ where if $f_i(A)=B$, (so $f_i$ labels the directed edge $(A, B)$), then $e_i(B)=A$ and $e_i$ labels the directed edge $(B,A).$ 
        \item A \emph {highest weight node} and a \emph {lowest weight node} within each connected component. A node is highest weight if there are no $e_i$ edges from it, and similarly it is lowest weight if it has no $f_i$ edges from it.
    \end{itemize}
    The graph structure of a connected component of a crystal graph is determined by its highest weight $\lambda$, a partition.  We say that two connected components of highest weight $\lambda$ are crystal isomorphic.

    \subsection{Models for Crystals} 
    
   There are several models of crystal graphs for the $\mathfrak{gl}_n$ crystal structure (of type $A_{n-1}$), which use a variety of combinatorial objects as the nodes of the graph (see~\cite{BS},\cite{HenriquesKamnitzer2006}). The goal of such models is to make the determination of edges easier or more transparent, and to exemplify other features or related constructions (such as computing invariants of tensor products of representations by means of products of crystal graphs). 

In this section we briefly review several models for crystal graphs, namely:
\begin{enumerate}
    \item Crystals on Words (Section~\ref{subsection_word_crystals}).
    \item Crystals on Perforated Tableaux (Section~\ref{subsection_Ptab_crystals}).
    \item Crystals on Biwords (Section~\ref{subsection_biword_crystals}).
    \item The Distinguished Crystal, Crystals on SSYT (Section~\ref{distinguished})
\end{enumerate}
    
    Each connected component of a crystal graph, called an \emph{irreducible crystal}, has a unique highest weight node whose weight is given by a partition $\lambda.$ We denote an irreducible crystal of highest weight $\lambda$  by ${\mathcal B}_{\lambda}$.
    We use ${\mathcal B}_{\lambda}(T)$ to denote an irreducible crystal graph of highest weight $\lambda$ containing a node $T$ (where $T$ is a particular combinatorial object), thereby specifying the model for the crystal.
    For an excellent treatise on crystal graphs, see \cite{BS}.  
    
    We now describe several crystal graph models and their associated crystal operators. Each node is described by some combinatorial object (a word, tableau or perforated tableau for example) from which the weight of the node and the image of an operator $f_i$ or $e_i$ is computed.  When there is no edge labeled $f_i$ emanating from a given node, we say that the value of the operator $f_i$ at this node is \emph{NULL}.  A node $\omega$ in a crystal graph is called \emph{highest weight} if $e_i(\omega)=\NULL$ for $i=1,2,\dots n-1.$ 
    
    \subsubsection{Crystals on Words}\label{subsection_word_crystals}
    
    Let $[n]=\{1,2,\dots , n\}$ and $[n]^{\otimes k}$ be the set of words of length $k$ in the alphabet $[n]$. We define a crystal operator $f_i$ on $[n]^{\otimes k}$ for $i=1,2,\dots , n-1$.  To make the definition agree with crystal operators on ptableaux (see Definition \ref{ptab def}); our rule differs from that in \cite{BS}. 
    
    Suppose $\rho = \rho_{1}\rho_{2} \cdots \rho_k \in [n]^{\otimes k}$. 
 We first label each $i$ (resp. $i+1$) in $\rho$ with $-$ (resp. $+$). For example, if $\rho=3\,3\,2\,2\,3\,2\,1\,1\,1\,3\,1\,2\,2\,3$ and $i=2$, we label
\[\begin{matrix}
    3&3&2&2&3&2&1&1&1&3&1&2&2&3\\
    +&+&-&-&+&-&&&&+&&-&-&+.
\end{matrix}
\]
We then \emph{bracket} each $-$ with the closest $+$ to its right (if it exists) until all unbracketed $i$'s are to the right of any unbracketed $i+1$. We identify bracketed pairs in a set of parentheses. For our example, the bracketing is: 
\[\begin{matrix}
    3&3&2&2&3&2&1&1&1&3&1&2&2&3\\
    +&+&-&-&+&-&&&&+&&-&-&+\\
    +&+&-&(-&+)&(-&&&&+)&&-&(-&+).
\end{matrix}
\]
If there is no unbracketed $-$, then $f_i(\rho) = \NULL$.  Otherwise, we find the left-most unbracketed $-$ and change its corresponding $i$ to $i+1$.  Thus in our example, the left-most unbracketed $-$ corresponds to the first $2,$ so 
$$
f_2(3\,3\,2\,2\,3\,2\,1\,1\,1\,3\,1\,2\,2\,3) = 3\,3\,3\,2\,3\,2\,1\,1\,1\,3\,1\,1\,3\,3.
$$

The \emph{weight} of a word $\omega \in [n]^{\otimes k}$ is the composition $\wt(\omega) = (a_1, \ldots , a_n)$ where $a_i$ is the number of $i$'s in $\omega$. For example, if $\omega = 2\,1\,2\,2\,3\,3\,1\in [4]^{\otimes 7}$, then its weight is $\wt(\omega) = (2, 3, 2, 0)$.

These definitions determine an $A_{n-1}$ crystal structure on $[n]^{\otimes k}$. 

There is also a crystal operator $e_i$ for $i\in[n-1]$, which is the inverse of $f_i$.  In this case, we find the right-most unbracketed $i+1$ and change it to an $i$.  See \cite{BS}.

    \subsubsection{Crystals on Perforated Tableaux}\label{subsection_Ptab_crystals}
    Perforated tableaux were defined in \cite{AW2022} as an alternative combinatorial model for crystal graphs of type $A_{n-1}$. 
    Perforated tableaux are objects such as:
\ytableausetup{smalltableaux}
\[ T=\begin{ytableau} \ & & 1 && 3 & 4\\ & 1 &2 & 2&& \\3 & 3 & 4 & 4 && \end{ytableau} \]
that form the nodes of type $A_{n-1}$ crystal graphs. More formally:

\begin{definition}
    
(See~\cite{AW2022}) \label{ptab} A rectangular tableau $T$ is called a \emph{perforated tableau}  (or \emph{ptableau}) if:
 \begin{enumerate}
 \item All boxes in $T$ contain positive integers or are blank (unfilled, denoted $\,\,\Box$).
  \item For any positive integer $i$, the $i$'s in $T$ form a horizontal strip.
 \item If $i < j$, then no element in the horizontal strip of $j$'s is in a higher row, nor in a column to the left, of a column containing any element of the horizontal strip of $i$'s.
 \item The number of columns of $T$ subject to the above conditions is \emph{minimum}. In particular, $T$ has no column entirely composed of blanks (rows of all blanks are allowed).
\end{enumerate}

Given a perforated tableau $T$, we define the \emph{weight} of $T$ as the composition $\wt(T) = (a_1, \ldots , a_n)$ where $a_i$ is the number of filled boxes in row $i$ of $T$.

\end{definition}

A given perforated tableau can have several equivalent forms:

\begin{definition}
    
Suppose $S$ and $T$ are ptableaux with $n$ rows. Then $S$ and $T$ are \emph{row-equivalent} if we can transform $S$ into $T$ by a finite number of steps in which a blank in a given row is swapped with content immediately to its left or right, such that the result is a valid ptableau. Note that if $S$ and $T$ are row-equivalent, then $\wt(S) = \wt(T)$.
\end{definition} 
For example, 
$$\begin{ytableau} \ & & 1 && 3 & 4\\ & 1 &2 & 2&& \\3 & 3 & 4 & 4 && \end{ytableau} \text{ is row equivalent to }
    \begin{ytableau} \ & & 1 && 3 & 4\\ 1&  &2 & 2&& \\3 & 3 & 4 & &4&
\end{ytableau}$$
Two particular row-equivalent ptableaux will be used to define a crystal structure on ptableaux.
\begin{definition}
A ptableau $T$ is \emph{left-justified} if all of $T$'s content is as far to the \emph{left} as possible. That is, in any ptableau $T'$ that is row-equivalent to $T$, each entry of $T$ is weakly to the left of the corresponding entry in $T'$.

Similarly, a ptableau $T$ is \emph{right-justified} if all content of $T$ is as far right as possible.
\label{justified}
\end{definition}

The following are then easily proved:

\begin{lemma}
    
\begin{enumerate}
 \item Given any ptableau $T$, there is a unique ptableau $T^\leftarrow$ that is left-justified and row-equivalent to $T$.
\item Given a left-justified ptableau $T^\leftarrow$, there is a unique right-justified ptableau $T^\rightarrow$ that is row-equivalent to $T^\leftarrow$.
\end{enumerate}\label{left-right-justified}
\end{lemma}

So, for example, if
\ytableausetup{smalltableaux}
\[ T= \begin{ytableau}
\ &&&&1&1&4 \\
&1&&1&&2&5\\
1&2&3&4&4&5& \end{ytableau}\, , \]
then we have the row-equivalent ptableaux:
\[ T^\leftarrow =\begin{ytableau}
\ &&&1&1&4& \\
&1&1&2&&&5\\
1&2&3&4&4&5& \end{ytableau} \ \ \hbox{and} \   T^\rightarrow = \begin{ytableau}
\ &&&&1&1&4 \\
&&1&1&&2&5\\
1&2&3&4&4&5&\end{ytableau}\, . \]

Row-equivalence is clearly an equivalence relation on ptableaux. A consequence of  Lemma 3.6 in \cite{AW2022} is that all perforated tableaux in a given row-equivalence class have the same number of columns.

\begin{definition} Let $\PTab$ denote the set of row-equivalence classes of ptableaux, each of which contains ptableaux that are row-equivalent to the left-justified representative $T^\leftarrow$ and that have the same number of columns as $T^\leftarrow$.  Let $\PTab_{n}$ denote the subset of $\PTab$ composed of ptableaux with $n$ rows (some of those rows may be entirely blank).
Let $\PTab_{n,k}$ be the subset of $\PTab_n$ with $k$ filled boxes.
\label{ptab def}
\end{definition}

\begin{lemma}[\cite{AW2022}] A given ptableau $T \in \PTab$ is determined by its content, and the rows in which that content appears. \label{unique}
\end{lemma}

We now define the crystal operator $f_i$ on $\PTab_{n}$. 

\begin{definition}
Let $T \in \PTab_{n,k}$  and choose $i$, for $1 \leq i <n-1$. Let $T[i,i+1]$ denote the ptableau formed by the entries of $T$ in rows $i$ and $i+1$.\end{definition}

Note that, as columns of blanks are not allowed in a ptableau,  we omit any such columns in the construction of $T[i,i+1]$. For example, if
\ytableausetup{smalltableaux}
$$ T=  \begin{ytableau}
 \ &&&&&&&&4&4&5\\
&&&&1&1&2&&&6&6\\
&&1&1&2&&&4&5&&7\\
1&1&2&3&3&3&4&6&6&& \end{ytableau} \, , \quad \hbox{then} \ \  T[2,3] = \begin{ytableau}
\ &&1&1&2&6&6\\
1&1&2&4&5&&7 \end{ytableau} \, . $$

Lemma 4.6 in ~\cite{AW2022} shows that if content in row $(i+1)$ of $T[i,i+1]^\rightarrow$ has a blank below it, then the corresponding content in $T^\rightarrow$ also has a blank below it. With this, we have:

\begin{definition}  \label{f_i def} Suppose $T \in \PTab_{n,k}$ and that $\framebox{c}_{\, i}$ is an entry in row $i$ of $T$ such that in $T[i,i+1]^\rightarrow$ (the right-justified form of $T[i,i+1]$) the corresponding $\framebox{c}$ is the leftmost entry with a blank below it. Then,  define $f_i (T)$ to be the ptableau obtained by swapping $\framebox{c}_{\, i}$ with the blank below it (in row $i+1$) of $T$. If there is \emph{no} such entry $\framebox{c}$, then $f_i (T) = \NULL$.
\end{definition}
In~\cite{AW2022} it was shown the operators $f_i$ above put a type $A_{n-1}$ crystal structure on $\PTab_{n,k}$. 

\begin{example}\label{example ptab operator}
Consider $T$ and $T[2,3]$ as above.

We calculate $f_{2}(T)$:
\begin{equation*} T[2,3]^\rightarrow = \begin{ytableau}
\ &&1&1&2&6&6\\
1&1&2&4&5&&7 \end{ytableau} ,\label{example eq} \end{equation*}
and so the $6$ in the sixth column of $T[2,3]^\rightarrow$ is the leftmost entry over a blank. We circle this entry in $T^\rightarrow$ below:
\ytableausetup{centertableaux}

$$ T^\rightarrow=\begin{ytableau}
 \ &&&&&&&&4&4&5\\
&&&&&&1&1&2&\textcircled{$\scriptstyle 6$}&6\\
&&&&1&1&2&4&5&&7\\
1&1&2&3&3&3&4&6&6&& \end{ytableau}\, ,
 \hbox{ and so }
f_{2}(T) =  \begin{ytableau}
\ &&&&&&&&4&4&5\\
&&&&1&1&2&&&&6\\
&&1&1&2&&&4&5&\textcircled{$\scriptstyle 6$}&7\\
1&1&2&3&3&3&4&6&6&& \end{ytableau}\,.$$
\end{example}
 
 Theorem 6.2 of~\cite{AW2022} states:

\begin{theorem} A ptableau $T$ is highest weight if and only if its left-justified form $T^\leftarrow$ is a semistandard Young tableau. \label{highest weight ptab}
\end{theorem}

For example, the ptableau
\[ T = \begin{ytableau}
\ &&&&1&1&4 \\
&&1&1&&2&5\\
1&2&3&4&4&5& \end{ytableau} \ \
\hbox{is not highest weight since} \
 T^\leftarrow =  \begin{ytableau}
\ &&&1&1&4 &\\
&1&1&2&&&5\\
1&2&3&4&4&5& \end{ytableau} \, , \]
which is not a SSYT. However,
the ptableau
\[ T' =  \begin{ytableau}
 1&1&1&1&1&4&4 \\
& 2 && 4&4&6&\\
&&3&&5&& \end{ytableau} \, , \ \
\hbox{is highest weight since}\
(T')^\leftarrow = \begin{ytableau}
 1&1&1&1&1&4&4 \\
 2 & 4&4&6&&&\\
3&5&&&&& \end{ytableau}\, , \]
which is semistandard.

\subsubsection{Crystals on Biwords and Bijections}\label{subsection_biword_crystals}

To establish precise relationships between crystal models on words and on ptableaux, we use biwords as an intermediary. 

\begin{definition}\label{def_ ptab_biword} Let $_{\BW}\Phi_{\PTab} :\PTab  \rightarrow \BW$ denote the map taking a ptableau $T$ to the biword $\binom {\kappa}{\rho}=\binom {\kappa_1\kappa_2\dots \kappa_k}{\rho_1\rho_2\dots \rho_k}$ where there is a column $\binom {i}{j} \in \binom{\kappa}{\rho}$ for each entry $i$ in row $j$ of $T$, ordered so that $_{\scriptstyle \BW}\Phi_{\PTab}(T)$ is a valid biword.

The weight of a biword ${\kappa \choose \rho}= _{\BW}\Phi_{\PTab}(T)$ is
$$\wt\binom{\kappa}{\rho} = \wt(\rho),$$
which necessarily equals $\wt(T)$.
\end{definition}
\begin{example}\label{ex:ptab-biword}

     If $$T=\begin{ytableau} \ & & 1 && 3 & 4\\ & 1 &2 & 2&& \\3 & 3 & 4 & 4 && \end{ytableau}
\text {, then the biword of }T \text{ is } _{\BW}\Phi_{\PTab}(T)=
\binom{\kappa}{\rho}  = \left(\begin{matrix} 1 & 1& 2 & 2& 3 & 3& 3& 4&4&4 \\
2 & 1 &2 & 2& 3&3&1&3&3&1 \end{matrix}\right), $$
with weight
$$\wt\binom{\kappa}{\rho} = \wt(\rho) = \wt(T) = (3,3,4).$$
Note that the content of $T$ and the biword $_{\BW}\Phi_{\PTab}(T)$ coincide.
\end{example}

The following is easily established (see~\cite{AW2022}, Lemma 3.5).

\begin{lemma}\label{lemma bijections} Let $\BW$ be the set of all biwords. 
    \begin{enumerate}
        \item The map $_{\BW}\Phi_{\PTab} :\PTab  \rightarrow \BW$ is a bijection. Restricting $_{\BW}\Phi_{\PTab}$ to $\PTab_{n,k}$ is a bijection onto $\BW_{n,k}$.
        \item The map $_{\BW}\Phi_{\PTab} :\PTab  \rightarrow \BW$ (and its restrictions) is a crystal isomorphism under the action of crystal operators defined by:
        \begin{equation}\label{eq:biwords_inherited_fi}
   f_i\left(\binom{\kappa}{\rho}\right) = \binom{\kappa}{f_i(\rho)} 
\end{equation}
with $f_i$ defined on words in Section \ref{subsection_word_crystals}.
        \end{enumerate}
\end{lemma}

The second statement of the lemma is an immediate consequence of Theorem 5.1 of~\cite{AW2022}, that states that crystal operators on ptableaux commute with crystal operators on the row words of their associated biwords:

\[  _{\BW}\Phi_{\PTab}(f_i (T)) = f_{i} ({_{\BW}\Phi_{\PTab}}(T)) = f_i \left({\kappa \choose \rho} \right)= {\kappa \choose f_i (\rho)} . \]

In Section \ref{Subsection: WordstoBi}, we defined the minimum and maximum biword of a word $\rho \in [n]^{\otimes k}$. By a minor extension of Lemma \ref{lemma bijections}, we have the following:
\begin{proposition}\label{prop: words to biwords}
   \begin{enumerate}
   \item Let 
    \[ _{{\BW}_{n,k}}\Phi^{\mmin}_{[n]^{\otimes k}}: [n]^{\otimes k} \rightarrow \BW_{n,k}\] be the map that takes a word $\rho \in [n]^{\otimes k}$ to the \emph{minimal} biword ${\kappa \choose \rho}$, where $\kappa$ is constant over maximal, weakly decreasing factors of $\rho$ (see Definition~\ref{defn:minimal_biword}).
       \item  Let 
    \[ _{{\BW}_{n,k}}\Phi^{\mmax}_{[n]^{\otimes k}}: [n]^{\otimes k} \rightarrow \BW_{n,k}\] be the map that takes a word $\rho \in [n]^{\otimes k}$ to the \emph{maximal} biword ${\kappa \choose \rho}$, where $\kappa= 123\cdots k$ (see Definition~\ref{def:maximal-biword}).
   \end{enumerate}
   Then both maps are injective, and induce crystal isomorphisms onto their images.
\end{proposition}

By virtue of the above proposition, we have injections 
\[_{\PTab}\Phi^{\mmin}_{[n]^{\otimes k}} = _{\PTab}\Phi_{\BW}  \circ _{\BW}\Phi^{\mmin}_{[n]^{\otimes k}}\]
taking words to ptableaux associated to the minimal biwords of given row words, and 
\[_{\PTab}\Phi^{\mmax}_{[n]^{\otimes k}} = _{\PTab}\Phi_{\BW}  \circ _{\BW}\Phi^{\mmax}_{[n]^{\otimes k}}\]
taking words $[n]^{\otimes k}$ onto the set of ptableaux with content $[k]$. We use minimal biwords and associated ptableaux in Section~\ref{section: Parsed word quasicrystal} to construct transparent realizations of quasicrystals.  We use maximal biwords in Section~\ref{Section: S-W duality}, with all words in $[n]^{\otimes k}$ mapped to ptableaux and/or biwords with content $\{1,2,\dots, k\}$.

\subsubsection{The Distinguished Crystal} \label{distinguished}
As we have seen, the crystals on words model gives the decomposition of the $k$-fold tensor product $V^{\otimes k}$ of an $n$-dimensional vector space $V$ via crystal operators acting on words $[n]^{\otimes k}$. We can also model this decomposition with crystals on biwords, in bijection with crystals on perforated tableaux. 

Another well-studied model for irreducible crystals graphs of highest weight $\lambda$ uses SSYT of shape $\lambda$ \cite{KN1994}, \cite{Littelmann1995}. Chapters 3 and 6 of \cite{BS} have a full development of crystals on tableaux. All irreducible crystals of highest weight $\lambda$ are crystal isomorphic (\cite{Stembridge.2003}, \cite[Theorem 4.13]{BS}), but there is only a canonical isomorphism from the SSYT model to a subset of $[n]^{\otimes k}$, defined by mapping a SSYT to its reading word (see Definition~\ref{def:reading_word_right_left_top_btm}). That is, the crystal of SSYT corresponds uniquely to only \emph{one} irreducible crystal on words, and it does not differentiate isomorphic crystals within $[n]^{\otimes k}$ 

This problem stems from the fact that the reading word of some SSYT is defined by its \emph{content}. It is precisely this feature that prevents the SSYT model from being applied to distinct yet isomorphic crystals. However, as we saw in Proposition~\ref{prop: words to biwords}, the ptableau model identifies words in $[n]^{\otimes k}$ with the \emph{rows} in which content appears. 

To make the relations between our models for crystals and the SSYT model more precise, we define the following: 

\begin{definition}  \label{def:ptab_associated_parsed words}
Let $T$ be a ptableau with associated biword ${\kappa \choose \rho} = {\kappa_1\ldots \kappa_k \choose \rho_1\ldots \rho_k}$ (Definition~\ref{def_ ptab_biword}). 

\begin{itemize}
    \item The \emph{parsed row word associated to $T$ (or equivalently $\kappa \choose \rho$)},  denoted $\rho_{|\kappa|}$, is the row word $\rho$ with 
    $(\kappa^{i+1} - \kappa^i)$-many parsing bars ``$\,|$'' (for $i=1,2,\dots, k-1$) inserted between $\rho_i$ and $\rho_{i+1}$.
    \item The \emph{parsed content word associated to $T$ (or equivalently $\kappa \choose \rho$)},  denoted $\widehat{\kappa}_{|\rho|}$, is the row word $\widehat{\kappa}_{\rho}$ of the dual biword $\widehat{\binom{\kappa}{\rho}}$  with 
    $(\rho^{i+1} - \rho^i)$-many parsing bars ``$\,|$'' (for $i=1,2,\dots, k-1$) inserted between $\widehat{\kappa}_{{\rho}_i}$ and $\widehat{\kappa}_{{\rho}_{i+1}}$. 
    \item In particular, if $\kappa \choose \rho$ is the minimal biword of $\rho$, then we refer to the parsed content word $\widehat{\kappa}_{|\rho|}$ as a \emph{parsed minimal content word}. Similarly, if $\kappa \choose \rho$ is the maximal biword of $\rho$, then we refer to the parsed content word $\widehat{\kappa}_{|\rho|}$ as a \emph{parsed maximal content word}.
    (See Definitions~\ref{defn:minimal_biword} and~\ref{def:maximal-biword}.)
\end{itemize}
The letters between two consecutive parsing bars make up a \emph{parsed factor}. The \emph{parsed weight} of a parsed row or content word is the weak composition given by the sizes of the parsed factors  (including parts of size 0 for empty factors), denoted $\wt_{\mathcal{P}}(\cdot)$.
\end{definition}

\begin{remark}
    
    If $\rho_{|\kappa|}$ is the parsed row word associated to the biword $\binom{\kappa}{\rho}$, then the biword $\binom{\kappa}{\rho}$ is associated to the parsed word  $\rho_{|\kappa|}$ as in Definition~\ref{def:biword of a parsed word}. The parsed content word $\widehat{\kappa}_{|\rho|}$ of the biword $\binom{\kappa}{\rho}$ is associated to $\binom{\kappa}{\rho}$ under a different bijection between parsed words and biwords that is more crystal compatible, developed in Section~\ref{section: Parsed word quasicrystal}.
\end{remark}

\begin{example}
If
$$ \left( \begin{matrix} \kappa \\ \rho \end{matrix} \right) =\left( \begin{matrix} 1\,1\,2\,3\,3\,5\,5\,5\,6 \\ 2\,1\,3\,3\,1\,2\,2\,1\,4 \end{matrix} \right), $$
then the parsed row word of this biword is
$\rho_{ |\kappa|} =21|3|31||221|4,$ with parsed weight $(2,1,2,0,3,1)$.
The dual biword
\[\left( \begin{matrix} \rho^* \\ \widehat{\kappa}_{\rho} \end{matrix} \right)
=\left( \begin{matrix} 1\,1\,1\,2\,2\,2\,3\,3\,4 \\ 5\,3\,1\,5\,5\,1\,3\,2\,6 \end{matrix} \right),\]
has as its associated parsed content word: $\widehat{\kappa}_{|\rho|}=531|551|32|6$. The parsed weight of $\widehat{\kappa}_{|\rho|}$ is $\wt_{\mathcal P}(\widehat{\kappa}_{|\rho|})=(3,3,2,1)$. Note that $\wt_{\mathcal P}(\widehat{\kappa}_{|\rho|})=\wt(\rho)$.
\end{example}

The number of letters in a parsed row word $\rho_{|\kappa|}$ between the $(i-1)$-st and $i$-th parsing bar is the number of $i$'s in the content word $\kappa$ of its associated biword. Therefore, the biword of any parsed row word is uniquely recoverable from its associated parsed row word, or its associated parsed content word. 
\begin{lemma}
    The map from biwords (or ptableaux) to their associated parsed row words (or parsed content words) is a bijection.
\end{lemma}

We now extend the notion of duality for biwords (see Definition~\ref{Def: dual biword}) to ptableaux:
\begin{definition}
    Let $T$ be a perforated tableau with associated biword $\kappa \choose \rho$. The \emph{dual ptableau} of $T$, denoted $\widehat{T}$, is the perforated tableau associated to the dual biword $\widehat{\kappa \choose \rho}$. 
\end{definition}
Clearly, if $T$ has content $i$ in row $j$, then $\widehat{T}$ has content $j$ in row $i$.

The \emph{reading word} of a SSYT is a word obtained by listing its content in some order. We extend this notion to ptableaux, but need to be able to identifiy specific rows for specific content. Hence, the reading word of a ptableau is a parsed content word:

\begin{definition}\label{def:reading_word_right_left_top_btm}Let $T$ be a ptableau associated to the biword ${\kappa \choose \rho}$. The \emph{reading word of $T$}, $\operatorname{rw}(T)$, is the parsed word whose $i$-th parsed factor equals to the entries in row $i$ of $T$, read right to left.
\end{definition}

Since the factor between the $(i-1)$-st and $i$-th parsing bars in $\widehat{\kappa}_{|\rho|}$ determines the content of $T$ in row $i$, written in weakly decreasing order, we then have the following elementary observation:

\begin{lemma}\label{lem:reading_words_Ptab}
    If a ptableau $T$ and its dual $\widehat{T}$ have associated biwords ${\kappa \choose \rho}$ and ${\rho^* \choose \widehat{\kappa}_{ \rho}}$,   then $\text{rw}(T)=\widehat{\kappa}_{ |\rho|}$ and $\text{rw}(\widehat{T})=\rho_{|\kappa|}$.
\end{lemma}

\begin{example}
Recall Example~\ref{ex:ptab-biword}:
$$T=\begin{ytableau} \ & & 1 && 3 & 4\\ & 1 &2 & 2&& \\3 & 3 & 4 & 4 && \end{ytableau}
\text {, with associated biword}\ 
\binom{\kappa}{\rho}  = \left(\begin{matrix} 1 & 1& 2 & 2& 3 & 3& 3& 4&4&4 \\
2 & 1 &2 & 2& 3&3&1&3&3&1 \end{matrix}\right). $$
The dual biword
\[ \widehat{\kappa \choose \rho} = \left(\begin{matrix} 1 & 1& 1 & 2& 2 & 2& 3& 3&3&3 \\
4 & 3 &1 & 2& 2&1&4&4&3&3 \end{matrix}\right) \ \hbox{has associated ptableau (dual to $T$)} \ \widehat{T}=\begin{ytableau} \ & & 1 & &2 \\ &&2&2&\\&1&&3&3\\1&3&3&&\end{ytableau} .\]

Then $\widehat{\kappa}_{|\rho|}=4\,3\,1|2\,2\,1|4\,4\,3\,3$ is the reading word of 
$T$, while $\rho_{|\kappa|} = 21|22|331|331$ is the reading word of $\widehat{T}$.
\end{example}
The following definitions allow us to identify the SSYT model within our framework.

\begin{definition}\label{Def:distinguished crystal}
Given a partition $\lambda=(\lambda_1, \lambda_2,\dots)$, let $T_\lambda$ be the ptableau with $\lambda_1$ $1$s in row $1$, $\lambda_2$ $2$s in row $2$, etc (so, in particular, $T^\leftarrow_{\lambda}$ is a SSYT of shape $\lambda$).  Then the irreducible ptableaux crystal ${\mathcal B_\lambda}(T_{\lambda})$ whose highest weight node is $T_\lambda$, is called the \emph {distinguished crystal} of highest weight $\lambda$.\label{distinguished crystal def}
\end{definition}

\begin{definition} \label{def:word condition}

A ptableau $T \in \PTab$ satisfies the \emph{word condition} if in $T^\rightarrow$, every $i+1$ has an $i$ above it (for $i\in [n-1])$.
\end{definition} 

Note that the word condition on ptableaux does \emph{not} correspond to the Yamanouchi condition on words. The name is chosen only to be consistent with usage of the term in characterizing Littlewood-Richardson tableaux. The following is a consequence of the above (see also Corollary 7.8 of~\cite{AW2022}).

\begin{proposition}\label{prop_distinguished_ptab} The following are equivalent for a ptableau $T$ in a crystal of highest weight $\lambda.$
 \begin{enumerate}
 \item $T$ lies in the distinguished crystal of highest weight $\lambda$, where $\lambda_i$ equals the number of $i$'s in $T$.
\item $T$ satisfies the word condition. 
\item The left-justified form of $\widehat T$ (the dual of $T$) is a SSYT or, equivalently, $\widehat{T}$ is highest weight (as a ptableau).
\item The biword ${\kappa \choose \rho}$ associated to $T$ is minimal and $\rho_{|\kappa|}$ is the reading word of a SSYT.
\end{enumerate} \label{word equiv}
\end{proposition}

\begin{definition} \label{def: sorts of disinguished crystals}
    Suppose ${\mathcal B_\lambda}(T_{\lambda})$ is the distinguished crystal of ptableaux of highest weight $\lambda$. Then the associated crystal of words, given by the row words of ${\mathcal B_\lambda}(T_{\lambda})$ will be called the {\emph{distinguished crystal of words}}, and the crystal of biwords of the ptableaux in ${\mathcal B_\lambda}(T_{\lambda})$ will be called the {\emph {distinguished crystal of biwords}}.
\end{definition}

\begin{example}
    The ptableau
    $\begin{displaystyle}T = \begin{ytableau} \ &&&1&1&1 \\ &1&1&&2&2 \\1&&2&2&3&3 \end{ytableau}\end{displaystyle}$
    satisfies the word condition and is associated to the minimal biword
    \[ {\kappa \choose \rho} = {111111222233 \choose 322111332233 }.\]

Its left-justified dual, $\widehat{T}^\leftarrow$, is the SSYT
$\begin{displaystyle}
    \begin{ytableau} 1&1&1&2&2&3\\2&2&3&3&&\\3&3&&&& \end{ytableau}\end{displaystyle}$,
with biword (dual to $\binom{\kappa}{\rho}$) of
$${\rho^* \choose \widehat{\kappa}_{\rho}}={111222233333 \choose 111221133221}.$$
As noted in Lemma \ref{lem:reading_words_Ptab}, $\operatorname{rw}(T) = \widehat{\kappa}_{|\rho|} =111|2211|33221 $  and  $\operatorname{rw}(\widehat{T}) =\rho_{|\kappa|}=322111|3322|33$.
\end{example}

In the literature, the crystal structure on SSYT of shape $\lambda$ is induced by identifying a SSYT $T$ with its reading word and then using the fact that a crystal operator takes a reading word of a SSYT of shape $\lambda$ to another such reading word (this is the content of Corollary 7.8 of~\cite{AW2022}). This action takes the \emph{dual} of a ptableau in the distinguished crystal (which is a SSYT) to the dual of its crystal operator image. While the map gives a valid crystal, it cannot be generalized to other (non-distinguished) crystals. In addition, the use of reading words to determine the crystal structure makes analysis of quasicrystals more challenging than the alternatives presented here (Theorem \ref{thm:quasi constant content}).

\begin{definition}\label{defn:plactic_equivalence}
    Suppose $\Phi: {\mathcal C} \rightarrow {\mathcal C}'$ is a crystal isomorphism for two crystal graphs ${\mathcal C}, {\mathcal C}'$ (in any model). If $c \in {\mathcal C}$ is a node in the crystal graph ${\mathcal C}$, and $\Phi(c) = c' \in {\mathcal C}'$, we say $c$ and $c'$ are \emph{plactically equivalent}. That is, $c$ and $c'$ occupy the same position in isomorphic crystal graphs.
\end{definition}

We have the well-known result:

\begin{proposition}[\cite{BS}]\label{prop: Knuth=plactic}
    If $\omega$ and $\omega'\in [n]^{\otimes k}$ are Knuth equivalent (see Definition~\ref{Def:Knuth Equiv}), then they are plactically equivalent.
\end{proposition}

Below are examples of irreducible crystal graphs of highest weight $(3, 1, 0).$
The crystal labeled (1) is the distinguished crystal of ptableaux (Def. \ref{Def:distinguished crystal}). Figure (2) is the corresponding SSYT crystal (see Proposition \ref{prop_distinguished_ptab}). Figure (3) is the corresponding biword crystal (Def. \ref{def_ ptab_biword}). Replacing each biword by the bottom row of the biword gives (4), the corresponding crystal of words in $[3]^{\otimes 4}$ (Sec. \ref{subsection_word_crystals}).

\[
\scalebox{0.75}{ \ytableausetup{smalltableaux}
\begin{tikzpicture}
\node (A) { \ytableaushort{111,2,}* {3,3,3}};
\node [below of=A, left of=A,node distance = 2cm](B) {\ytableaushort{\none11,1\none2,}* {3,3,3}};
\node (C) [right of = A,below of=A,node distance = 2cm] {\ytableaushort{111,\none,\none\none2}* {3,3,3}};
\node (D) [left of=B, below of = B,node distance = 2cm] {\ytableaushort{\none\none1,112,}* {3,3,3}};
\node (E) [below of=A,node distance = 3.4cm] {\ytableaushort{\none11,1,\none\none2}* {3,3,3} };
\node (F) [below of=E,node distance =1.2cm] {\ytableaushort{\none11,\none\none2,1}* {3,3,3} };
\node (G) [below of =B,node distance= 3.5cm]  {\ytableaushort{\none\none1,11,\none\none2}* {3,3,3}};
\node (H) [below of = G,node distance = 1.2 cm] {\ytableaushort{\none\none1,\none12,1}* {3,3,3} };
\node (I) [below of =C,node distance = 4.2cm] {\ytableaushort{\none11,\none\none\none,1\none2}* {3,3,3}};
\node (J) [below of=D, node distance = 4.2cm] {\ytableaushort{\none\none\none,111,\none\none2}* {3,3,3}};
\node (K) [below of = F,node distance = 3cm] {\ytableaushort{\none\none1,\none1,1\none2}* {3,3,3}  };
\node (L) [below of = K,node distance = 1.2 cm] {\ytableaushort{\none\none1,\none\none2,11}* {3,3,3}};
\node (M) [below of = H,node distance = 4 cm] { \ytableaushort{\none\none\none,\none11,1\none2}* {3,3,3}};
\node (N) [below of = I,node distance = 4.4 cm]{\ytableaushort{\none\none1,\none\none\none,112}* {3,3,3}};
\node (O) [below of = L,node distance = 3.4cm] {\ytableaushort{\none\none\none,\none\none1,112}* {3,3,3}};
\draw[->] (A.225) to node [above] {$f_{1}$} (B.45);
\draw[->] (A.315) to node [above] {$f_{2}$} (C.135);
\draw[->] (B.225) to node [above] {$f_{1}$} (D.45);
\draw[->] (B.250) to node [above] {$f_{2}$} (F.165);
\draw[->] (C.225) to node [above] {$f_{1}$} (E.0);
\draw[->] (E.180) to node [left] {$f_{1}$} (G.90);
\draw[->] (D.240) to node [above] {$f_{2}$} (H.180);
\draw[->] (G.195) to node [left] {$f_{1}$} (J.115);
\draw[->] (F.270) to node [above] {$f_{1}$} (H.0);
\draw[->] (F.315) to node [above] {$f_{2}$} (I.180);
\draw[->] (H.270) to node [above] {$f_{2}$} (L.180);
\draw[->] (I.225) to node [above] {$f_{1}$} (K.0);
\draw[->] (J.270) to node [above] {$f_{2}$} (M.135);
\draw[->] (G.0) to node [right] {$f_{2}$} (K.90);
\draw[->] (M.300) to node [above] {$f_{2}$} (O.180);
\draw[->] (K.215) to node [left] {$f_{1}$} (M.115);
\draw[->] (N.250) to node [above] {$f_{1}$} (O.0);
\draw[->] (L.0) to node [above] {$f_{2}$} (N.90);
\end{tikzpicture} 
\qquad \quad

\begin{tikzpicture}
\node (A) { \ytableaushort{111,2,}* {3,1,0}};
\node [below of=A, left of=A,node distance = 2cm](B) {\ytableaushort{112,2}* {3,1}};
\node (C) [right of = A,below of=A,node distance = 2cm] {\ytableaushort{111,3}* {3,1}};
\node (D) [left of=B, below of = B,node distance = 2cm] {\ytableaushort{122,2}* {3,1}};
\node (E) [below of=A,node distance = 3.4cm] {\ytableaushort{112,3,}* {3,1} };
\node (F) [below of=E,node distance =1.2cm] {\ytableaushort{ 113,2}* {3,1} };
\node (G) [below of =B,node distance= 3.5cm]  {\ytableaushort{122,3}* {3,1}};
\node (H) [below of = G,node distance = 1.2 cm] {\ytableaushort{123,2}* {3,1} };
\node (I) [below of =C,node distance = 4.2cm] {\ytableaushort{ 113,3}* {3,1}};
\node (J) [below of=D, node distance = 4.2cm] {\ytableaushort{222,3}* {3,1}};
\node (K) [below of = F,node distance = 3cm] {\ytableaushort{123,3}* {3,1}  };
\node (L) [below of = K,node distance = 1.2 cm] {\ytableaushort{133,2}* {3,1}};
\node (M) [below of = H,node distance = 4 cm] { \ytableaushort{223,3}* {3,1}};
\node (N) [below of = I,node distance = 4.4 cm]{\ytableaushort{133,3}* {3,1}};
\node (O) [below of = L,node distance = 3.4cm] {\ytableaushort{233,3}* {3,1}};
\draw[->] (A.225) to node [above] {$f_{1}$} (B.45);
\draw[->] (A.315) to node [above] {$f_{2}$} (C.135);
\draw[->] (B.225) to node [above] {$f_{1}$} (D.45);
\draw[->] (B.250) to node [above] {$f_{2}$} (F.165);
\draw[->] (C.225) to node [above] {$f_{1}$} (E.0);
\draw[->] (E.180) to node [left] {$f_{1}$} (G.90);
\draw[->] (D.240) to node [above] {$f_{2}$} (H.180);
\draw[->] (G.195) to node [left] {$f_{1}$} (J.115);
\draw[->] (F.270) to node [above] {$f_{1}$} (H.0);
\draw[->] (F.315) to node [above] {$f_{2}$} (I.180);
\draw[->] (H.270) to node [above] {$f_{2}$} (L.180);
\draw[->] (I.225) to node [above] {$f_{1}$} (K.0);
\draw[->] (J.270) to node [above] {$f_{2}$} (M.135);
\draw[->] (G.0) to node [right] {$f_{2}$} (K.90);
\draw[->] (M.300) to node [above] {$f_{2}$} (O.180);
\draw[->] (K.215) to node [left] {$f_{1}$} (M.115);
\draw[->] (N.250) to node [above] {$f_{1}$} (O.0);
\draw[->] (L.0) to node [above] {$f_{2}$} (N.90);
\end{tikzpicture}}
\]
\hskip 1.5in (1) \hskip 3in (2)

\[
\scalebox{0.75}{\begin{tikzpicture}
\node (A) {$\binom{1112}{1112}$};
\node [below of=A, left of=A,node distance = 2cm](B) {$\binom{1112}{2112}$};
\node (C) [right of = A,below of=A,node distance = 2cm] {$\binom{1112}{1113}$};
\node (D) [left of=B, below of = B,node distance = 2cm] {$\binom{1112}{2212}$};
\node (E) [below of=A,node distance = 3.4cm] {$\binom{1112}{2113}$};
\node (F) [below of=E,node distance =1.2cm] {$\binom{1112}{3112}$};
\node (G) [below of =B,node distance= 3.5cm]  {$\binom{1112}{2213}$};
\node (H) [below of = G,node distance = 1.2 cm] {$\binom{1112}{3212}$};
\node (I) [below of =C,node distance = 4.2cm] {$\binom{1112}{3113}$};
\node (J) [below of=D, node distance = 4.2cm] {$\binom{1112}{2223}$};
\node (K) [below of = F,node distance = 3cm] {$\binom{1112}{3213}$};
\node (L) [below of = K,node distance = 1.2 cm] {$\binom{1112}{3312}$};
\node (M) [below of = H,node distance = 4 cm] {$\binom{1112}{3223}$};
\node (N) [below of = I,node distance = 4.4 cm]{$\binom{1112}{3313}$};
\node (O) [below of = L,node distance = 3.4cm] {$\binom{1112}{3323}$};
\draw[->] (A.225) to node [above] {$f_{1}$} (B.45);
\draw[->] (A.315) to node [above] {$f_{2}$} (C.135);
\draw[->] (B.225) to node [above] {$f_{1}$} (D.45);
\draw[->] (B.250) to node [above] {$f_{2}$} (F.165);
\draw[->] (C.225) to node [above] {$f_{1}$} (E.0);
\draw[->] (E.180) to node [left] {$f_{1}$} (G.90);
\draw[->] (D.240) to node [above] {$f_{2}$} (H.180);
\draw[->] (G.195) to node [left] {$f_{1}$} (J.115);
\draw[->] (F.270) to node [above] {$f_{1}$} (H.0);
\draw[->] (F.315) to node [above] {$f_{2}$} (I.180);
\draw[->] (H.270) to node [above] {$f_{2}$} (L.180);
\draw[->] (I.225) to node [above] {$f_{1}$} (K.0);
\draw[->] (J.270) to node [above] {$f_{2}$} (M.135);
\draw[->] (G.0) to node [right] {$f_{2}$} (K.90);
\draw[->] (M.300) to node [above] {$f_{2}$} (O.180);
\draw[->] (K.215) to node [left] {$f_{1}$} (M.115);
\draw[->] (N.250) to node [above] {$f_{1}$} (O.0);
\draw[->] (L.0) to node [above] {$f_{2}$} (N.90);
\end{tikzpicture} 
\qquad \quad

\begin{tikzpicture}
\node (A) {$1112$};
\node [below of=A, left of=A,node distance = 2cm](B) {$2112$};
\node (C) [right of = A,below of=A,node distance = 2cm] {${1113}$};
\node (D) [left of=B, below of = B,node distance = 2cm] {${2212}$};
\node (E) [below of=A,node distance = 3.4cm] {${2113}$};
\node (F) [below of=E,node distance =1.2cm] {${3112}$};
\node (G) [below of =B,node distance= 3.5cm]  {${2213}$};
\node (H) [below of = G,node distance = 1.2 cm] {${3212}$};
\node (I) [below of =C,node distance = 4.2cm] {${3113}$};
\node (J) [below of=D, node distance = 4.2cm] {${2223}$};
\node (K) [below of = F,node distance = 3cm] {${3213}$};
\node (L) [below of = K,node distance = 1.2 cm] {${3312}$};
\node (M) [below of = H,node distance = 4 cm] {${3223}$};
\node (N) [below of = I,node distance = 4.4 cm]{${3313}$};
\node (O) [below of = L,node distance = 3.4cm] {${3323}$};
\draw[->] (A.225) to node [above] {$f_{1}$} (B.45);
\draw[->] (A.315) to node [above] {$f_{2}$} (C.135);
\draw[->] (B.225) to node [above] {$f_{1}$} (D.45);
\draw[->] (B.250) to node [above] {$f_{2}$} (F.165);
\draw[->] (C.225) to node [above] {$f_{1}$} (E.0);
\draw[->] (E.180) to node [left] {$f_{1}$} (G.90);
\draw[->] (D.240) to node [above] {$f_{2}$} (H.180);
\draw[->] (G.195) to node [left] {$f_{1}$} (J.115);
\draw[->] (F.270) to node [above] {$f_{1}$} (H.0);
\draw[->] (F.315) to node [above] {$f_{2}$} (I.180);
\draw[->] (H.270) to node [above] {$f_{2}$} (L.180);
\draw[->] (I.225) to node [above] {$f_{1}$} (K.0);
\draw[->] (J.270) to node [above] {$f_{2}$} (M.135);
\draw[->] (G.0) to node [right] {$f_{2}$} (K.90);
\draw[->] (M.300) to node [above] {$f_{2}$} (O.180);
\draw[->] (K.215) to node [left] {$f_{1}$} (M.115);
\draw[->] (N.250) to node [above] {$f_{1}$} (O.0);
\draw[->] (L.0) to node [above] {$f_{2}$} (N.90);
\end{tikzpicture}}
\]
\hskip 1.75in (3) \hskip 3in (4)


\section{Quasisymmetric Functions and Quasicrystals}\label{section: quasicrystals}

In this section, we provide the necessary background on symmetric functions, quasisymmetric functions, and crystal characters to motivate the study of quasicrystals. We then include the original definitions of quasicrystals by Cain and Malheiro \cite{CM17} and Maas-Gariépy \cite{FMG23} through the model on words and on semistandard Young tableaux, respectively. 

  \subsection{Crystal characters } 
  For $\mu=(\mu_1, \mu_2,\ldots, \mu_n)$ a composition of $k$,  let 
$x^\mu$ denote the monomial $x_1^{\mu_1}x_2^{\mu_2}\ldots x_n^{\mu_n}.$ The \textit{character} of a crystal graph $\mathcal{B}$ is the sum of the monomials of its nodes, denoted $\chi(\mathcal{B}) = \sum_{v\in \mathcal{B}} x^{\mathsf{wt}(v)}$.   The character of a type $A_{n-1}$ irreducible crystal $\mathcal{B}_{\lambda}$ of highest weight $\lambda$ is the Schur polynomial $s_\lambda(x_1,\ldots, x_n)$, most clearly realized in the semistandard Young tableaux model. Let $T_{\lambda}$ denote the \emph{superstandard} SSYT of shape $\lambda$ containing only $i$'s in row $i$. Then the crystal character of $\mathcal{B}_{\lambda}(T_\lambda)$ is given by
 \[ \chi(\mathcal{B}_{\lambda}(T_\lambda))(x_1,\ldots,x_n) = \sum_{T}  x^{\wt(T)}, \]
 where the sum is over all SSYT of shape $\lambda$ such that the length of the composition $\ell(\wt(T)) \leq n$.

  Although we work with a fixed type $A_{n-1}$, the inductive limit of such polynomials (as $n \rightarrow \infty$) is the Schur function $s_\lambda$. We imagine that the character given above denotes the inductive limit of Schur polynomials on the right side, with a similar limit of crystal graphs (as $n \rightarrow \infty$) on the left.      It is a significant feature of crystal graphs that they form a coherent inductive limit not only as the number of nodes goes to infinity, but as the number of edges does, that is, as crystal operators $f_{n-1}, e_{n-1}$ are definable for arbitrarily large values of $n$.
We express this as:
   
    \[ \chi(\mathcal{B}_{\lambda}(T_{\lambda})) = \sum_{\substack{T \in SSYT \\ \operatorname{sh}(T) = \lambda}} x^{\wt(T)} = s_{\lambda} \label{crystal character} \]
    where, in fact, the second equality is the traditional combinatorial definition of a Schur function.

  Symmetric functions (and symmetric polynomials) have a long history in the analysis of multiple problems in combinatorics and representation theory (see, for example, \cite{macdonald1995symmetric}, \cite{mendes2015counting}, \cite{EC2}). Gessel (\cite{G84}) developed a broader class of functions called \textit{quasisymmetric functions}, which have become an important tool for studying symmetric functions and, in particular, Schur positivity.

  Let $\mathbf{x}=\{x_1, x_2, \dots\}$ be a countably infinite set of commuting variables, and let $\mathbb C[[\mathbf{x}]]$ be the algebra of formal power series in $\mathbf{x}$.

  \begin{definition}
        A function $f\in \mathbb C[[\mathbf{x}]]$ is quasisymmetric if for every composition $\alpha = (\alpha_1, \alpha_2, \ldots, \alpha_\ell)$, the coefficient of $x_1^{\alpha_1}x_2^{\alpha_2}\cdots x_{\ell}^{\alpha_\ell}$ is the same as the coefficient of $x_{i_1}^{\alpha_1}x_{i_2}^{\alpha_2}\cdots x_{i_\ell}^{\alpha_\ell}$ for all $0 < i_1 < i_2 < \ldots < i_\ell$. 

  \end{definition}

The \textit{fundamental quasisymmetric functions} given by Gessel \cite{G84} form an important basis of the ring of quasisymmetric functions. 
\begin{definition}
    For a composition $\alpha = (\alpha_1, \alpha_2, \ldots, \alpha_\ell) \models k$, the \emph{fundamental quasisymmetric functions} $F_\alpha$ is defined as 
\begin{equation}
	 F_{\alpha} := \sum_{\substack{1 \leqslant i_1 \leqslant \cdots \leqslant i_k \,: \\ \, i_j < i_{j+1} \text{ if } j \in D(\alpha)}} x_{i_1}\cdots x_{i_k},
\end{equation}
where $D(\alpha):=\{\alpha_1, \alpha_1 + \alpha_2, \ldots, \alpha_1 + \alpha_2 + \ldots + \alpha_{\ell-1}\}$ is the (descent) set associated to $\alpha$. 
\end{definition}
The expansion of the Schur function $s_\lambda$ in terms of fundamental quasisymmetric functions as a sum over standard Young tableau of shape $\lambda$ is a well-known consequence of Stanley's theory of P-partitions. 
\begin{theorem}[\cite{EC2} Theorem 7.19.7]
    For $\lambda \vdash n$, the fundamental quasisymmetric function expansion of the Schur function $s_\lambda$ is
    \begin{equation}
	 s_\lambda = \sum_{T} F_{D(T)},
\end{equation}
where the sum is over all standard Young tableau of shape $\lambda$ and $D(T)$ is the descent set of $T$, the set of letters $i$ such that $i+1$ appears in a row of strictly greater index than $i$ in $T$.
\end{theorem}

\subsection{Quasicrystals on words and semistandard Young tableaux} 

The expansion of a Schur polynomial into fundamental quasisymmetric polynomials can be represented as a decomposition of an irreducible type $A_{n-1}$ crystal graph into \textit{quasicrystals} (formally defined in Definition~\ref{def:quasicrystal_definition}), which were introduced by Cain and Malheiro and Maas-Gariépy \cite{CM17, FMG23}. See Sections 5 and 6 of \cite{CM17} and Section 2, Theorem 1 of \cite{FMG23} for full details and development of quasicrystals in the crystal on words and crystal on semistandard Young tableaux, respectively. In these models, a quasicrystal can be defined as an induced subgraph containing all of the nodes that \textit{standardize} to the same permutation or standard Young tableau.  We first define standardization in our conventions, which matches Kashiwara conventions and is opposite of the anti-Kashiwara conventions used in Cain--Malheiro and Maas-Gariépy.

   \begin{definition}\label{def:standardization-words-ssyt}
   \hskip .25in
       \begin{itemize}
           \item To \emph{standardize a word} $w = w_1w_2\ldots w_k 
 \in [n]^{\otimes k}$, we successively label the letters with $1, 2, \ldots, k$ by reading $w$ from right to left 
 and first labeling all of the $1$s, then labeling all of the $2$s, and so on. The standardization of $w$, denoted $\std(w)$ is the \emph{standard word} obtained by replacing the letters with their labels.
\item To \emph{standardize a semistandard Young tableau} $T$ with $k$ boxes, we successively label the entries with $1, 2, \ldots, k$ from left to right starting with the $1$s, then the $2$s and so on. The standardization of $T$, denoted $\std(T)$ is the \emph{standard Young tableau} obtained by replacing the entries with their labels. 
\item The \emph{maximal horizontal bands} in a semistandard Young tableau $T$ are ``bands" of boxes created by successively reading the entries $1,2,\ldots, k$ from the boxes in $\std(T)$ and beginning a new band if the box containing the next entry is not strictly to the right of the box containing the current entry. 
       \end{itemize}
   \end{definition}

\begin{example}\label{ex std} Let $w = 24123311 \in[4]^{\otimes 8}$. Then $\std(w) = 58347621$, as computed below.
\[
\begin{tabular}{cccccccc}
     $\cdot$&$\cdot$&3&$\cdot$&$\cdot$&$\cdot$&2&1  \\
     2&4&1&2&3&3&1&1
\end{tabular} \quad \rightarrow \begin{tabular}{cccccccc}
    5&$\cdot$&3&4&$\cdot$&$\cdot$&2&1  \\
     2&4&1&2&3&3&1&1
\end{tabular} \quad \rightarrow \begin{tabular}{cccccccc}
    5&8&3&4&7&6&2&1  \\
     2&4&1&2&3&3&1&1.
\end{tabular}
\]
Let $T = \begin{ytableau}  1& 1&1 &2 & 3 & 5\\ 2&  2 & 4 \\4 &4 \end{ytableau}$ be a SSYT of shape $(6,3,2)$. Then  $\std(T) = \begin{ytableau}  *(cyan) 1 & *(cyan)2&*(cyan) 3 &*(lime)6 & *(lime) 7 & *(olive) 11\\ *(lime) 4&  *(lime) 5 & *(olive) 10 \\*(olive) 8 &*(olive) 9 \end{ytableau}.$ The colors highlight the maximal horizontal bands of increasing consecutive content starting in the leftmost column.
\end{example}

\begin{definition}\label{def:quasicrystal_definition}
    Given a permutation $\pi$ (or resp. standard Young tableau $T$), the \emph{quasicrystal} containing $\pi$ (resp. $T$) is the induced subgraph of the crystal on words (resp. semistandard Young tableaux) on the subset of nodes that standardize to $\pi$ (resp. $T$). 
\end{definition}

A quasicrystal is indexed by the \textit{descent composition} of the standardization of its nodes, defined for both words and semistandard Young tableaux as follows. A word $w$ has descent composition $\alpha = (\alpha_1, \ldots, \alpha_\ell)$ if, in the standardization of $w$,
the letters $\{\alpha_1 + \ldots + \alpha_{i}+1,\ldots, \alpha_1 + \ldots + \alpha_{i} + \alpha_{i+1}\}$ appear in increasing order reading from right to left and the letter $\alpha_1 + \ldots + \alpha_{i+1}$ appears to the left of the letter $\alpha_1 + \ldots +\alpha_{i+1}+1$, for $0 \leq i\leq \ell$ ($\alpha_0 := 0$). A semistandard Young tableau $T$ has descent composition $\alpha = (\alpha_1, \ldots, \alpha_\ell)$ if, in the standardization of $T$, the letter $\alpha_1 + \ldots + \alpha_{i}+1$ appears in a row of greater index than the letter $\alpha_1 + \ldots + \alpha_{i}$ for $1 \leq i \leq \ell -1$. The descent composition of $T$ is also given by the number of boxes in its maximal horizontal bands (see Definition~\ref{def:standardization-words-ssyt} and \cite{FMG23}). 

\begin{example}
    Continuing Example~\ref{ex std}, in $\std( 24123311) = 58347621$, the letters $\{1,2,3\}, \{4,5\},$ \\$\{6,7,8\}$ appear in increasing order reading from right to left. Therefore $w$ has descent composition $(3, 5-3, 8-5) = (3,2,3)$. In $\std\left(\begin{ytableau}  1& 1&1 &2 & 3 & 5\\ 2&  2 & 4 \\4 &4 \end{ytableau}\right) = \begin{ytableau}  *(cyan) 1 & *(cyan)2&*(cyan) 3 &*(lime)6 & *(lime) 7 & *(olive) 11\\ *(lime) 4&  *(lime) 5 & *(olive) 10 \\*(olive) 8 &*(olive) 9 \end{ytableau}$, the letters $4$ and $8$ appear in rows of strictly greater index than $3$ and $7$. Therefore, $T$ has descent composition $(3, 7-3, 11-7) = (3,4,4)$, equal to the number of boxes in the maximal horizontal bands, colored blue, green, and olive.

\end{example}

Cain and Malheiro and Maas-Gariépy prove the following important facts that link quasicrystals to the fundamental quasisymmetric decomposition of type $A_{n-1}$ crystal characters. 

\begin{proposition}[\cite{FMG23} Theorem 1, \cite{CM17} Proposition 6.4]\label{prop:quasicrystals_are_connected_subgraphs}
    For any permutation $\rho$ in the crystal on words $[n]^{\otimes k}$ (resp. any standard Young tableau $T$ in the crystal on semistandard Young tableaux) with descent composition $\alpha$, the quasicrystal $Q_\alpha$ containing $\rho$ (resp. $T$) is a connected subgraph of the crystal on words (resp. semistandard Young tableaux). Furthermore, $Q_\alpha$ contains a unique highest weight node of weight $\alpha$ and, when $n=k$, a unique standard word (resp. standard Young tableau).
\end{proposition}

\begin{proposition}[\cite{FMG23} Theorem 1]\label{prop:quasicrystals_contain_all_destd_and_fqf_character}
    The character of $Q_\alpha$ is the fundamental quasisymmetric function $F_\alpha$, and thus, because they contain all words with the same standardization, quasicrystals in the irreducible crystal $\mathcal{B}_\lambda$ (of words or semistandard Young tableaux) decompose the vertex set of $\mathcal{B}_\lambda$ according to the fundamental quasisymmetric function expansion of the Schur function $s_\lambda$. 
\end{proposition}

Cain and Malheiro define explicit \textit{quasi-Kashiwara operators} that define edges within a quasicrystal graph on words. 

\begin{definition}\label{def-qKashiwara-op}\cite {CM17}
    For $1 \leq i \leq n-1$, the \emph{quasi-Kashiwara operator} $\Tilde{f}_i$ is NULL on words in which a letter $i$ appears to the left of a letter $i+1$, and otherwise acts as the usual crystal operator on words $f_i$ defined in Section~\ref{subsection_word_crystals}.  
\end{definition}

The crystal edges connecting nodes within different quasicrystals were characterized by Brauner-Corteel-Daugherty-Schilling in their construction of \emph{crystal skeletons} \cite{BCDS}. Importantly, the dual equivalence graphs defined by Assaf \cite{Assaf15} are connected subgraphs of crystal skeletons (see  \cite[Theorem 1.2] {BCDS}).

  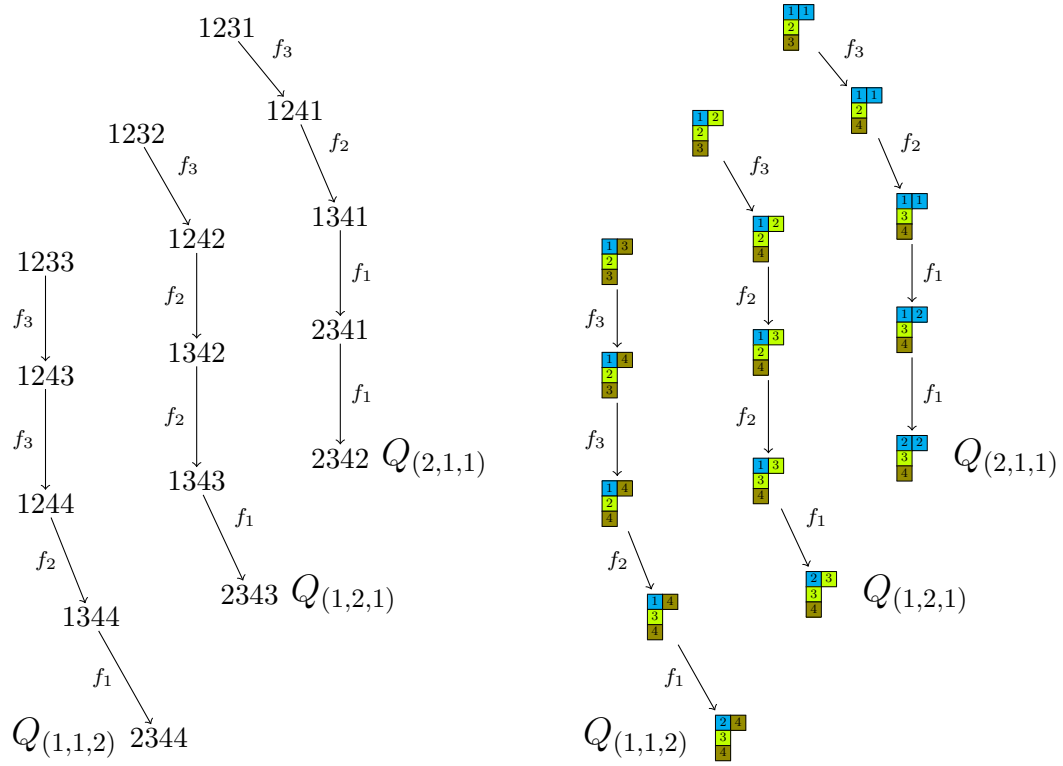
\begin{figure}[H]
      \centering

      \begin{tikzpicture}[
    scale=1,
    transform shape,
    every node/.style={font=\tiny},
    node/.style={inner sep=1.5pt, font=\small},
    every edge/.style={draw},
    label distance=-1pt
]
\useasboundingbox (-3.6,-10.0) rectangle (3.0,0.6);
 
\node[node] (A) at (0,0) {$1231$};
 
\node[node] (B) at (-1.2,-1.4) {$1232$};
\node[node] (C) at ( 0.9,-1.1) {$1241$};
 
\node[node] (D) at (-2.4,-3.1) {$1233$};
\node[node] (E) at ( -0.4,-2.8) {$1242$};
\node[node] (F) at ( 1.5,-2.5) {$1341$};
 
\node[node] (G) at (-2.4,-4.6) {$1243$};
\node[node] (H) at ( -0.4,-4.3) {$1342$};
\node[node] (I) at ( 1.5,-4.0) {$2341$};
 
\node[node] (J) at (-2.4,-6.3) {$1244$};
\node[node] (K) at (-0.4,-6.0) {$1343$};
\node[node] (L) at ( 1.5,-5.7) {$2342$};
 
\node[node] (M) at (-1.8,-7.8) {$1344$};
\node[node] (N) at ( 0.3,-7.5) {$2343$};
 
\node[node] (O) at (-0.9,-9.4) {$2344$};
 
\draw[->] (A) -- node[midway, above right] {$f_3$} (C);
\draw[->] (B) -- node[midway, above right] {$f_3$} (E);
\draw[->] (C) -- node[midway, above right] {$f_2$} (F);
 
\draw[->] (D) -- node[midway, left] {$f_3$} (G);
\draw[->] (E) -- node[midway, left] {$f_2$} (H);
\draw[->] (F) -- node[midway, right] {$f_1$} (I);
 
\draw[->] (G) -- node[midway, left] {$f_3$} (J);
\draw[->] (H) -- node[midway, left] {$f_2$} (K);
\draw[->] (I) -- node[midway, right] {$f_1$} (L);
 
\draw[->] (J) -- node[midway, left] {$f_2$} (M);
\draw[->] (K) -- node[midway, above right] {$f_1$} (N);
 
\draw[->] (M) -- node[midway, left] {$f_1$} (O);
 
\node[anchor=west] at (1.9,-5.7) {\large $Q_{(2,1,1)}$};
\node[anchor=west] at (0.7,-7.5) {\large $Q_{(1,2,1)}$};
\node[anchor=east] at (-1.3,-9.4) {\large $Q_{(1,1,2)}$};
 
\end{tikzpicture}
\qquad 
\begin{tikzpicture}[
    scale=1,
    transform shape,
    every node/.style={font=\tiny},
    node/.style={inner sep=1.5pt, font=\small},
    every edge/.style={draw},
    label distance=-1pt
]
\useasboundingbox (-3.6,-10.0) rectangle (3.0,0.6);
 
\node[node] (A) at (0,0) {$ \scalebox{0.6}{\begin{ytableau}
       *(cyan) 1 & *(cyan) 1 \\
*(lime) 2 \\
*(olive) 3
   \end{ytableau}}$};
 
\node[node] (B) at (-1.2,-1.4) {$ \scalebox{0.6}{\begin{ytableau}
       *(cyan) 1 & *(lime) 2 \\
*(lime) 2 \\
*(olive) 3
   \end{ytableau}}$};
\node[node] (C) at ( 0.9,-1.1) {$ \scalebox{0.6}{\begin{ytableau}
       *(cyan) 1 & *(cyan) 1 \\
*(lime) 2 \\
*(olive) 4
   \end{ytableau}}$};
 
\node[node] (D) at (-2.4,-3.1) {$\scalebox{0.6}{\begin{ytableau}
       *(cyan) 1 & *(olive) 3 \\
*(lime) 2 \\
*(olive) 3
   \end{ytableau}}$};
\node[node] (E) at ( -0.4,-2.8) {$\scalebox{0.6}{\begin{ytableau}
       *(cyan) 1 & *(lime) 2 \\
*(lime) 2 \\
*(olive) 4
   \end{ytableau}}$};
\node[node] (F) at ( 1.5,-2.5) {$\scalebox{0.6}{\begin{ytableau}
       *(cyan) 1 & *(cyan) 1 \\
*(lime) 3 \\
*(olive) 4
   \end{ytableau}}$};
 
\node[node] (G) at (-2.4,-4.6) {$\scalebox{0.6}{\begin{ytableau}
       *(cyan) 1 & *(olive) 4 \\
*(lime) 2 \\
*(olive) 3
   \end{ytableau}}$};
\node[node] (H) at ( -0.4,-4.3) {$\scalebox{0.6}{\begin{ytableau}
       *(cyan) 1 & *(lime) 3 \\
*(lime) 2 \\
*(olive) 4
   \end{ytableau}}$};
\node[node] (I) at ( 1.5,-4.0) {$\scalebox{0.6}{\begin{ytableau}
       *(cyan) 1 & *(cyan) 2 \\
*(lime) 3 \\
*(olive) 4
   \end{ytableau}}$};
 
\node[node] (J) at (-2.4,-6.3) {$\scalebox{0.6}{\begin{ytableau}
       *(cyan) 1 & *(olive) 4 \\
*(lime) 2 \\
*(olive) 4
   \end{ytableau}}$};
\node[node] (K) at (-0.4,-6.0) {$\scalebox{0.6}{\begin{ytableau}
       *(cyan) 1 & *(lime) 3 \\
*(lime) 3 \\
*(olive) 4
   \end{ytableau}}$};
\node[node] (L) at ( 1.5,-5.7) {$\scalebox{0.6}{\begin{ytableau}
       *(cyan) 2 & *(cyan) 2 \\
*(lime) 3 \\
*(olive) 4
   \end{ytableau}}$};
 
\node[node] (M) at (-1.8,-7.8) {$\scalebox{0.6}{\begin{ytableau}
       *(cyan) 1 & *(olive) 4 \\
*(lime) 3 \\
*(olive) 4
   \end{ytableau}}$};
\node[node] (N) at ( 0.3,-7.5) {$\scalebox{0.6}{\begin{ytableau}
       *(cyan) 2 & *(lime) 3 \\
*(lime) 3 \\
*(olive) 4
   \end{ytableau}}$};
 
\node[node] (O) at (-0.9,-9.4) {$\scalebox{0.6}{\begin{ytableau}
       *(cyan) 2 & *(olive) 4 \\
*(lime) 3 \\
*(olive) 4
   \end{ytableau}}$};
 
\draw[->] (A) -- node[midway, above right] {$f_3$} (C);
\draw[->] (B) -- node[midway, above right] {$f_3$} (E);
\draw[->] (C) -- node[midway, above right] {$f_2$} (F);
 
\draw[->] (D) -- node[midway, left] {$f_3$} (G);
\draw[->] (E) -- node[midway, left] {$f_2$} (H);
\draw[->] (F) -- node[midway, right] {$f_1$} (I);
 
\draw[->] (G) -- node[midway, left] {$f_3$} (J);
\draw[->] (H) -- node[midway, left] {$f_2$} (K);
\draw[->] (I) -- node[midway, right] {$f_1$} (L);
 
\draw[->] (J) -- node[midway, left] {$f_2$} (M);
\draw[->] (K) -- node[midway, above right] {$f_1$} (N);
 
\draw[->] (M) -- node[midway, left] {$f_1$} (O);
 
\node[anchor=west] at (1.9,-5.7) {\large{$Q_{(2,1,1)}$}};
\node[anchor=west] at (0.7,-7.5) {\large $Q_{(1,2,1)}$};
\node[anchor=east] at (-1.3,-9.4) {\large $Q_{(1,1,2)}$};
 
\end{tikzpicture}
      \caption{Left: The decomposition of the crystal on words $B_{(2,1,1)}(1231)$ into quasicrystals. Right: The decomposition of the crystal on SSYT $B_{(2,1,1)}(T_{(2,1,1)})$ into quasicrystals, with maximal horizontal bands highlighted.}
      \label{fig:quasicrystal-words-tab-ex}
  \end{figure}

\section{Crystals and Quasicrystals on Parsed Content Words}\label{section: Parsed word quasicrystal}
We now introduce a new model for $A_{n-1}$ crystals using parsed content words (see Definition~\ref{def:ptab_associated_parsed words}), which allows us to store all information from a ptableau (or its associated biword) in a single word with parsing bars. We characterize highest weight nodes in the parsed content words model, and provide a simple local rule for crystal operators. The parsed content word model easily identifies the quasicrystals within a crystal without requiring a standardization procedure, and nicely characterizes the distinguished crystal within its isomorphism class (see Theorem~\ref{thm:quasi constant content} and Theorem~\ref{prop:Yamanouchi criterion}).

\subsection{Crystal structure on parsed content words} 
Recall that for a perforated tableau $T$ (associated to the biword $\binom{\kappa}{\rho})$, the \emph{parsed content word of $T$ (or equivalently of $\binom{\kappa}{\rho}$)} is the row word $\widehat{\kappa}_{\rho}$ of the dual biword $\widehat{\binom{\kappa}{\rho}}$  with $(\rho^{i+1} - \rho^i)$-many parsing bars ``$\,|$'' (for $i=1,2,\dots, k-1$) inserted between the $i$-th and $i+1$-st letters in $\widehat{\kappa}_{{\rho}}$ (see Definitions~\ref{Def:parsed word} and \ref{def:ptab_associated_parsed words}). Recall also that the \emph{parsed weight} of $\widehat{\kappa}_{|\rho|}$ is the weak composition given by the sizes of the parsed factors, denoted $\wt_{\mathcal{P}}(\cdot)$.

\begin{example}
If
$$ \left( \begin{matrix} \kappa \\ \rho \end{matrix} \right) =\left( \begin{matrix} 1\,1\,2\,3\,3\,5\,5\,5\,6 \\ 2\,1\,3\,3\,1\,2\,2\,1\,4 \end{matrix} \right) \qquad \text{ and } \qquad \widehat{\binom{\kappa}{\rho}} = \binom{\rho^*}{\widehat{\kappa}_\rho} = \left( \begin{matrix} 1\,1\,1\,2\,2\,2\,3\,3\,4 \\ 5\,3\,1\,5\,5\,1\,3\,2\,6 \end{matrix} \right), $$
then the parsed content word of $\binom{\kappa}{\rho}$ is $\widehat{\kappa}_{|\rho|} = 531|551|32|6$ and $\wt_{\mathcal{P}}(\widehat{\kappa}_{|\rho|}) = (3,3,2,1)$. Note the difference between the  parsed content word $\widehat{\kappa}_{|\rho|}=531|551|32|6$ and the underlying word $\widehat{\kappa}_\rho = 531551326$. 
\end{example} 

From $\widehat{\kappa}_{|\rho|}$ we can uniquely recover the biword $\binom{\kappa}{\rho}$ by placing the letters in the $i$-th parsed factor in $\widehat{\kappa}_{|\rho|}$ above the letter $i$, and then sorting the columns to obtain a biword.

Let $\PCW_{n,k}$ denote the set of parsed content words associated to biwords in $\BW_{n,k}$:
\[
\PCW_{n,k} := \left \{ \widehat{\kappa}_{|\rho|} \mid \binom{\kappa}{\rho} \in \BW_{n,k}\right \}.
\]

Recall that the crystal structure on biwords inherited by Equation~\eqref{eq:biwords_inherited_fi} carries over to the crystal structure on their associated ptableaux (Lemma~\ref{lemma bijections}). 
Similarly, the set of parsed content words inherits a crystal structure from the crystal on biwords.  
We put a type $A_{n-1}$ crystal structure on $\PCW_{n,k}$ by defining
\begin{equation}\label{eq:parsed_words_inherited_fi}
  f_i(\widehat{\kappa}_{|\rho|}) := \widehat{\kappa}_{|f_i(\rho)|},
\end{equation}
where $\widehat{\kappa}_{|f_i(\rho)|}$ is the parsed content word associated to $\displaystyle \binom{\kappa}{f_i(\rho)}$, and $f_i(\rho)$ is the crystal operator on words defined in Section~\ref{subsection_word_crystals}. See Figure~\ref{fig:PCW (2,1,1) } for an example.
Note that $$\wt_{\mathcal{P}}(\widehat{\kappa}_{|\rho|}) = \wt\left(\displaystyle \binom{\kappa}{\rho}\right) = \wt(\rho),$$ since the number of $i$'s in $\rho$ determines the size of the $i$-th parsed factor in $\widehat{\kappa}_{|\rho|}$. To summarize, given a perforated tableau $T$, its associated biword $\binom{\kappa}{\rho}$ and parsed content word $\widehat{\kappa}_{|\rho|}$, the three objects $T$, $\binom{\kappa}{\rho}$, and $\widehat{\kappa}_{|\rho|}$ lie in the same position in isomorphic connected components in their respective crystal models, as does the word $\rho$ in the crystal on words. Furthermore, by Lemma~\ref{lem:reading_words_Ptab} the parsed content word $\widehat{\kappa}_{|\rho|}$ is the reading word of $T$. 

\begin{example}\label{ex:tpab-biword-pcw-f1}
    Below we showcase the crystal operator $f_1$ applied to objects associated via crystal isomorphisms.
    \[
    \begin{tabular}{cccc}
     Words & Biwords & Perforated Tableaux  & Parsed Content Words \\
       $2\,1\,3\,3\,1\,2\,2\,1\,4 $ & $\left( \begin{matrix} 1\,1\,2\,3\,3\,5\,5\,5\,6 \\ 2\,1\,3\,3\,1\,2\,2\,1\,4 \end{matrix} \right)$ & $\begin{ytableau} \  && 1 & 3 & 5\\  &1 &5 & 5& \\2 & 3   &&& \\ &&&&6 \end{ytableau}$ &  $531|551|32|6$ \\ 
      \footnotesize{$f_1$} $\downarrow $ &\footnotesize{$f_1$} $\downarrow $ & \footnotesize{$f_1$} $\downarrow $  & \footnotesize{$f_1$} $\downarrow $ \\
        $2\,1\,3\,3\,1\,2\,2\,2\,4$ & $\left( \begin{matrix} 1\,1\,2\,3\,3\,5\,5\,5\,6 \\ 2\,1\,3\,3\,1\,2\,2\,2\,4 \end{matrix} \right)$ &  $\begin{ytableau} \  && 1 & 3& \\  &1 &5 & 5&5 \\2 & 3   &&& \\ &&&&6 \end{ytableau}$   & $31|5551|32|6$
    \end{tabular}
    \]
\end{example}

\begin{figure}[H]
    \centering
  \scalebox{0.7}{\begin{tikzpicture}[
    node/.style={inner sep=3pt}, 
    every edge/.style={draw},
    label distance=-2pt
]
\node[node] (A) at (0,0) {$\left (\begin{matrix}
    1233 \\ 1231
\end{matrix} \right)$};
\node[node] (B) at (-1.8,-2.0) {$\left (\begin{matrix}
    1233 \\ 1232
\end{matrix} \right)$};
\node[node] (C) at ( 1.8,-2.0) {$\left (\begin{matrix}
    1233 \\ 1241
\end{matrix} \right)$};
\node[node] (D) at (-3.0,-4.0) {$\left (\begin{matrix}
    1233 \\ 1233
\end{matrix} \right)$};
\node[node] (E) at ( 0.0,-4.0) {$\left (\begin{matrix}
    1233 \\ 1242
\end{matrix} \right)$};
\node[node] (F) at ( 3.0,-4.0) {$\left (\begin{matrix}
    1233 \\ 1341
\end{matrix} \right)$};
\node[node] (G) at (-3.0,-6.0) {$\left (\begin{matrix}
    1233 \\ 1243
\end{matrix} \right)$};
\node[node] (H) at ( 0.0,-6.0) {$\left (\begin{matrix}
    1233 \\ 1342
\end{matrix} \right)$};
\node[node] (I) at ( 3.0,-6.0) {$\left (\begin{matrix}
    1233 \\ 2341
\end{matrix} \right)$};
\node[node] (J) at (-3.0,-8.0) {$\left (\begin{matrix}
    1233 \\ 1244
\end{matrix} \right)$};
\node[node] (K) at ( 0.0,-8.0) {$\left (\begin{matrix}
    1233 \\ 1343
\end{matrix} \right)$};
\node[node] (L) at ( 3.0,-8.0) {$\left (\begin{matrix}
    1233 \\ 2342
\end{matrix} \right)$};
\node[node] (M) at (-1.8,-10.0) {$\left (\begin{matrix}
    1233 \\ 1344
\end{matrix} \right)$};
\node[node] (N) at ( 1.8,-10.0) {$\left (\begin{matrix}
    1233 \\ 2342
\end{matrix} \right)$};
\node[node] (O) at (0,-12.0) {$\left (\begin{matrix}
    1233 \\ 2344
\end{matrix} \right)$};
\draw[->] (A) -- node[midway, above left] {$f_1$} (B);
\draw[->] (A) -- node[midway, above right] {$f_3$} (C);
\draw[->] (B) -- node[midway, above left] {$f_2$} (D);
\draw[->] (B) -- node[midway, above right] {$f_3$} (E);
\draw[->] (C) -- node[midway, above left] {$f_1$} (E);
\draw[->] (C) -- node[midway, above right] {$f_2$} (F);
\draw[->] (D) -- node[midway, left] {$f_3$} (G);
\draw[->] (E) -- node[midway, left] {$f_2$} (H);
\draw[->] (F) -- node[midway, right] {$f_1$} (I);
\draw[->] (G) -- node[midway, left] {$f_3$} (J);
\draw[->] (H) -- node[midway, left] {$f_2$} (K);
\draw[->] (I) -- node[midway, right] {$f_1$} (L);
\draw[->] (J) -- node[midway, left] {$f_2$} (M);
\draw[->] (K) -- node[midway, above left] {$f_3$} (M);
\draw[->] (K) -- node[midway, above right] {$f_1$} (N);
\draw[->] (L) -- node[midway, right] {$f_2$} (N);
\draw[->] (M) -- node[midway, left] {$f_1$} (O);
\draw[->] (N) -- node[midway, right] {$f_3$} (O);
\end{tikzpicture}}
   \quad
   \scalebox{0.7}{\begin{tikzpicture}[
    node/.style={inner sep=3pt}, 
    every edge/.style={draw},
    label distance=-2pt
]
\node[node] (A) at (0,0) {\ytableaushort{{1}{3},{2}{ },{3}{ },{ }{ }}};
\node[node] (B) at (-1.8,-2.0) {\ytableaushort{{1}{ },{2}{3},{3}{ },{ }{ }}};
\node[node] (C) at ( 1.8,-2.0) {\ytableaushort{{1}{3},{2}{ },{ }{ },{3}{ }}};
\node[node] (D) at (-3.0,-4.0) {\ytableaushort{{1}{ },{2}{ },{3}{3},{ }{ }}};
\node[node] (E) at ( 0.0,-4.0) {\ytableaushort{{1}{ },{2}{3},{ }{ },{3}{ }}};
\node[node] (F) at ( 3.0,-4.0) {\ytableaushort{{1}{3},{ }{ },{2}{ },{3}{ }}};
\node[node] (G) at (-3.0,-6.0) {\ytableaushort{{1}{ },{2}{ },{3}{ },{ }{3}}};
\node[node] (H) at ( 0.0,-6.0) {\ytableaushort{{1}{ },{ }{3},{2}{ },{3}{ }}};
\node[node] (I) at ( 3.0,-6.0) {\ytableaushort{{ }{3},{1}{ },{2}{ },{3}{ }}};
\node[node] (J) at (-3.0,-8.0) {\ytableaushort{{1}{ },{2}{ },{ }{ },{3}{3}}};
\node[node] (K) at ( 0.0,-8.0) {\ytableaushort{{1}{ },{ }{ },{2}{3},{3}{ }}};
\node[node] (L) at ( 3.0,-8.0) {\ytableaushort{{ }{ },{1}{3},{2}{ },{3}{ }}};
\node[node] (M) at (-1.8,-10.0) {\ytableaushort{{1}{ },{ }{ },{2}{ },{3}{3}}};
\node[node] (N) at ( 1.8,-10.0) {\ytableaushort{{ }{ },{1}{ },{2}{3},{3}{ }}};
\node[node] (O) at (0,-12.0) {\ytableaushort{{ }{ },{1}{ },{2}{ },{3}{3}}};
\draw[->] (A) -- node[midway, above left] {$f_1$} (B);
\draw[->] (A) -- node[midway, above right] {$f_3$} (C);
\draw[->] (B) -- node[midway, above left] {$f_2$} (D);
\draw[->] (B) -- node[midway, above right] {$f_3$} (E);
\draw[->] (C) -- node[midway, above left] {$f_1$} (E);
\draw[->] (C) -- node[midway, above right] {$f_2$} (F);
\draw[->] (D) -- node[midway, left] {$f_3$} (G);
\draw[->] (E) -- node[midway, left] {$f_2$} (H);
\draw[->] (F) -- node[midway, right] {$f_1$} (I);
\draw[->] (G) -- node[midway, left] {$f_3$} (J);
\draw[->] (H) -- node[midway, left] {$f_2$} (K);
\draw[->] (I) -- node[midway, right] {$f_1$} (L);
\draw[->] (J) -- node[midway, left] {$f_2$} (M);
\draw[->] (K) -- node[midway, above left] {$f_3$} (M);
\draw[->] (K) -- node[midway, above right] {$f_1$} (N);
\draw[->] (L) -- node[midway, right] {$f_2$} (N);
\draw[->] (M) -- node[midway, left] {$f_1$} (O);
\draw[->] (N) -- node[midway, right] {$f_3$} (O);
\end{tikzpicture}}
\qquad
\scalebox{0.7}{\begin{tikzpicture}[
    node/.style={inner sep=3pt}, 
    every edge/.style={draw},
    label distance=-2pt
]
\node[node] (A) at (0,0) {$31|2|3$};
\node[node] (B) at (-1.8,-2.0) {$1|32|3$};
\node[node] (C) at ( 1.8,-2.0) {$31|2||3$};
\node[node] (D) at (-3.0,-4.0) {$1|2|33$};
\node[node] (E) at ( 0.0,-4.0) {$1|32||3$};
\node[node] (F) at ( 3.0,-4.0) {$31||2|3$};
\node[node] (G) at (-3.0,-6.0) {$1|2|3|3$};
\node[node] (H) at ( 0.0,-6.0) {$1|3|2|3$};
\node[node] (I) at ( 3.0,-6.0) {$3|1|2|3$};
\node[node] (J) at (-3.0,-8.0) {$1|2||33$};
\node[node] (K) at ( 0.0,-8.0) {$1||32|3$};
\node[node] (L) at ( 3.0,-8.0) {$|31|2|3$};
\node[node] (M) at (-1.8,-10.0) {$1||2|33$};
\node[node] (N) at ( 1.8,-10.0) {$|1|32|3$};
\node[node] (O) at (0,-12.0) {$|1|2|33$};
\draw[->] (A) -- node[midway, above left] {$f_1$} (B);
\draw[->] (A) -- node[midway, above right] {$f_3$} (C);
\draw[->] (B) -- node[midway, above left] {$f_2$} (D);
\draw[->] (B) -- node[midway, above right] {$f_3$} (E);
\draw[->] (C) -- node[midway, above left] {$f_1$} (E);
\draw[->] (C) -- node[midway, above right] {$f_2$} (F);
\draw[->] (D) -- node[midway, left] {$f_3$} (G);
\draw[->] (E) -- node[midway, left] {$f_2$} (H);
\draw[->] (F) -- node[midway, right] {$f_1$} (I);
\draw[->] (G) -- node[midway, left] {$f_3$} (J);
\draw[->] (H) -- node[midway, left] {$f_2$} (K);
\draw[->] (I) -- node[midway, right] {$f_1$} (L);
\draw[->] (J) -- node[midway, left] {$f_2$} (M);
\draw[->] (K) -- node[midway, above left] {$f_3$} (M);
\draw[->] (K) -- node[midway, above right] {$f_1$} (N);
\draw[->] (L) -- node[midway, right] {$f_2$} (N);
\draw[->] (M) -- node[midway, left] {$f_1$} (O);
\draw[->] (N) -- node[midway, right] {$f_3$} (O);
\end{tikzpicture}}

    \caption{Irreducible crystals of highest weight $(2,1,1)$ in the crystal on biwords (left), crystal on perforated tableaux (middle), and crystal on parsed content words (right); these components are isomorphic under maps given in \eqref{lemma bijections} and \eqref{eq:parsed_words_inherited_fi}. The corresponding irreducible crystal on words can be formed by the row words in the biword model. The nodes of the parsed content word model are the reading words of the corresponding ptableaux in the middle crystal. Note that the choice of content word $\kappa = 1233$ in the biwords model informs the content filling the perforated tableaux and the letters forming the underlying words $3123$, $1323$, $1233$ in the parsed content words model. A different content word, for example $1234$ or $1377$, would give different isomorphic irreducible crystals.} 
    \label{fig:PCW (2,1,1) }
\end{figure}
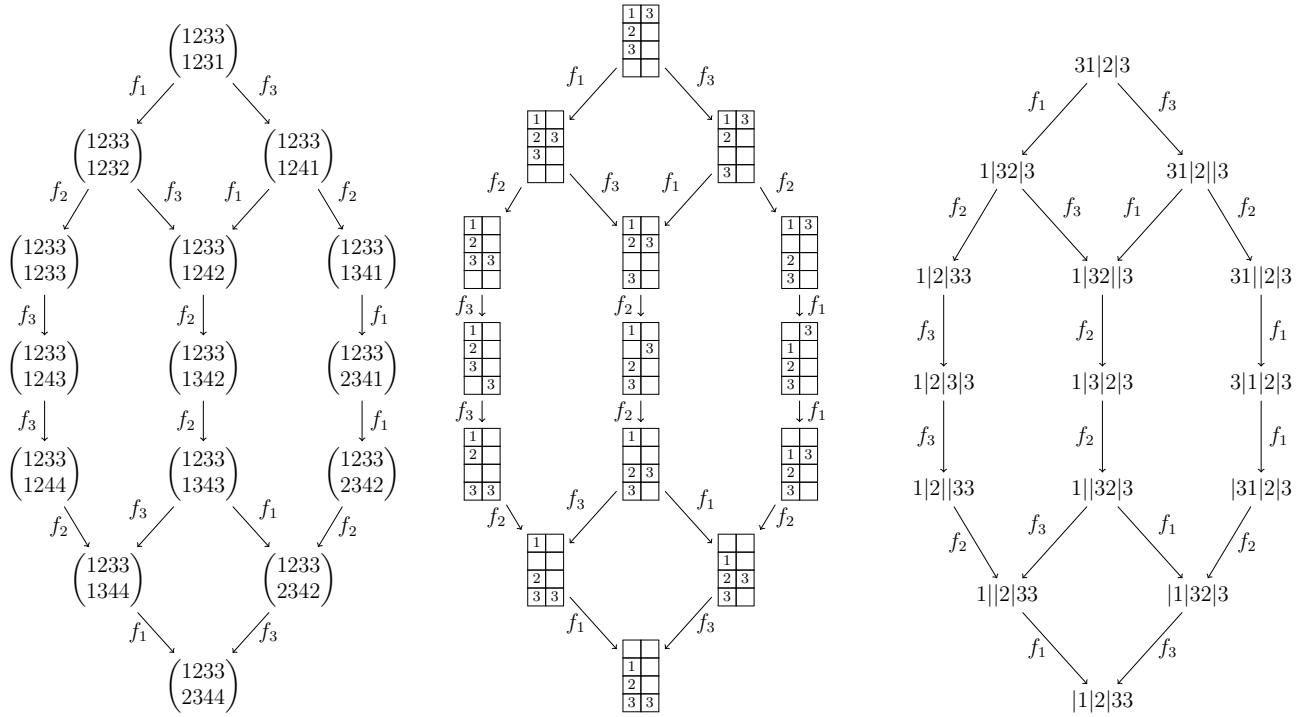

Using the classification of highest weight elements in the perforated tableaux model from Theorem~\ref{highest weight ptab}, we classify highest weight elements in the parsed content words model. 

\begin{theorem}
    A parsed content word is highest weight if and only if it is the reading word of a semistandard Young tableau, according to Definition~\ref{def:reading_word_right_left_top_btm}.
\end{theorem}
\begin{proof}
    Let  $\widehat{\kappa}_{|\rho|}$ be the parsed content word associated to the perforated tableau $T$. By Lemma~\ref{lem:reading_words_Ptab}, $\widehat{\kappa}_{|\rho|}$ is the reading word of $T$. Theorem~\ref{highest weight ptab} classifies the highest weight elements in the ptableaux model as the ptableau with left-justified form equal to a semistandard Young tableaux, and the result follows.
\end{proof}

\subsection{Local crystal operators on parsed content words}

We now give a local combinatorial rule for the crystal operator on $\PCW_{n,k}$ in the parsed content words model that agrees with that given in Equation~\eqref{eq:parsed_words_inherited_fi}. That is, we provide a rule for computing $f_{i}(\widehat{\kappa}_{|\rho|})$ defined entirely in terms of the parsed content word $\widehat{\kappa}_{|\rho|}$.
We first make a provisional definition of operators $\hat{f}_i$ for $i = 1,2,\ldots, n-1$ acting on the parsed content words, and then show that these match the operators $f_i$ on parsed content words given above.

\begin{definition} \label{Def: f_i on PCW}
Let $\widehat{\kappa}_{|\rho|} \in \PCW_{n,k}$. For $i \in [n-1]$, define an operator $\hat{f}_i$ on $\widehat{\kappa}_{|\rho|}$ by restricting its action to the  $i$-th and $(i+1)$-st parsed factors of $\widehat{\kappa}_{|\rho|}$ (noting that parsed factors may be empty). Pairs of letters $(a,b)$, with $a$ in the $i$-th parsed factor and $b$ in the $(i+1)$-st parsed factor are paired iteratively, reading the parsed factors from left to right: the leftmost unpaired $b$ in the $(i+1)$-st parsed factor is paired with the leftmost unpaired $a$ in the $i$-th parsed factor satisfying $a < b$, if such a letter $a$ exists.  Repeat this for all content in the $(i+1)$-st parsed factor. After this process terminates, move the smallest (rightmost) unpaired letter in the $i$-th parsed factor (if it exists) into the $(i+1)$-st parsed factor and reorder into weakly decreasing order. If there is no such unpaired letter to move, $\hat{f}_i$ is $\NULL$. 
\end{definition}

We also define the operators $\hat{e}_i$ for $i = 1, 2, \ldots, n-1$ as inverses to the $\hat{f}_i$ operators whenever their action is non-null.

\begin{definition} \label{Def: e_i on PCW}
 Let $\widehat{\kappa}_{|\rho|} \in \PCW_{n,k}$. For $i \in [n-1]$, we define an operator $\hat{e}_i$ on $\widehat{\kappa}_{|\rho|}$ by restricting its action to the  $i$-th and $(i+1)$-st parsed factors of $\widehat{\kappa}_{|\rho|}$. Pairs of letters $(a,b)$ with $a$ in the $i$-th parsed factor and $b$ in the $(i+1)$-st parsed factor are paired iteratively, reading the factors right to left: the rightmost unpaired $b$ in the $(i+1)$-st parsed factor is paired with the rightmost unpaired $a$ in the $i$-th parsed factor satisfying $a < b$, if it exists.  Repeat this for all content in the $(i+1)$-st parsed factor. After this process terminates, move the largest (leftmost) unpaired letter in the $i+1$-st parsed factor (if it exists) into the $i$-th parsed factor and reorder into weakly decreasing order. If there is no such unpaired letter to move, $\hat{e}_i$ is $\NULL$. 
 \end{definition}

\begin{example}
Continuing Example~\ref{ex:tpab-biword-pcw-f1}, to compute $\hat{f}_1$ on $531|551|32|6$ we pair the letters $(3,5)$ and $(1,5)$ from the 1st and 2nd parsed factors:
$$531|551|32|6\rightarrow 5{\color {red} 3}1|{\color{red}5}51|32|6\rightarrow 
5{\color {red} 3}{\color{red} 1}|{\color{red}5}{\color{red}5}1|32|6$$

The smallest unpaired letter in the 1st factor is $5$, and so we move $5$ to the second parsed factor and reorder the letters into weakly decreasing order:
$$\hat{f}_1(5{\color {red} 3}{\color{red} 1}|{\color{red}5}{\color{red}5}1|32|6) =  31|5551|32|6.$$

If we compute $\hat{f_1}$ on $31|5551|32|6$, we again pair the letters $(3,5)$ and $(1,5)$ from the 1st and 2nd parsed factors. Since no unpaired letter remains in the 1st factor, the result is $\NULL$: 
$$\hat{f}_1({\color {red} 3}{\color{red} 1}|{\color{red}5}{\color{red}5}51|32|6) =  \NULL.$$
\end{example}

\begin{theorem}\label{thm:f hat crystal op}
    The operators $\hat{f}_i$ defined above are the crystal operators on parsed content words $\PCW_{n,k}$ inherited from the biwords model $\BW_{n,k}$ given in Equation~\eqref{eq:parsed_words_inherited_fi}, meaning
    \[
    \hat{f}_i (\widehat{\kappa}_{|\rho|}) = \widehat{\kappa}_{|f_{i}(\rho)|}.
    \]
\end{theorem} 

\begin{proof} 

Given the parsed content word $\widehat{\kappa}_{|\rho|}$, we wish to show that $\hat{f}_{i}(\widehat{\kappa}_{|\rho|})$ equals $f_i(\widehat{\kappa}_{|\rho|})=\widehat{\kappa}_{|f_i(\rho)|}$. We recover the biword $\kappa \choose \rho$ from the parsed content word $\widehat{\kappa}_{|\rho|}$, and from this determine the unique associated ptableau $T$. By Lemma~\ref{lem:reading_words_Ptab}, $\widehat{\kappa}_{|\rho|}$ is the reading word of $T$, so the $i$-th and $(i+1)$-st rows of $T$ are the $i$-th and $(i+1)$-st factors of $\widehat{\kappa}_{|\rho|}$, respectively, read from right to left. The letters of $\widehat{\kappa}_{|\rho|}$ that are paired in the $i$-th and $(i+1)$-st factors of $\widehat{\kappa}_{|\rho|}$ when computing $\hat{f}_{i}(\widehat{\kappa}_{|(\rho)|})$ are simply the letters in row $i$ that lie directly above letters in row $(i+1)$ in $T^{\rightarrow}$, the right-justified form of $T$. The rightmost unpaired letter in the $i$-th factor of $\widehat{\kappa}_{|\rho|}$ is thus the \emph{leftmost} letter in $T^{\rightarrow}$ over a blank, and hence both such letters are moved under their respective operators, giving the result.
\end{proof}

\begin{example}
Recalling Example~\ref{example ptab operator}, suppose we have the ptableau

\[  T=  \begin{ytableau}
 \ &&&&&&&&4&4&5\\
&&&&1&1&2&&&6&6\\
&&1&1&2&&&4&5&&7\\
1&1&2&3&3&3&4&6&6&& \end{ytableau}\,,\]
with associated parsed content word
\[ \widehat{\kappa}_{|\rho|} = 544|66211|754211|664333211,\]
which is the reading word of $T$ by Lemma~\ref{lem:reading_words_Ptab}.
We calculate $f_{2}(T)$ and $\hat{f_{2}}(\widehat{\kappa}_{|\rho|})$. First,
\begin{equation*} T[2,3]^\rightarrow = \begin{ytableau}
\ &&1&1&2&6&6\\
1&1&2&4&5&&7 \end{ytableau} \,,\label{example eq} \end{equation*}
and so the $6$ in the 6th column of $T[2,3]^\rightarrow$ is the leftmost entry over a blank; applying $f_2$ to $T$ drops this letter into row $3$. The $2$nd and $3$rd factors of $\widehat{\kappa}_{|\rho|}$ are 
\[ 66211|754211.\]
 Computing the pairing of letters in the 2nd and 3rd factors gives

\[ 66211|754211 \rightarrow {\color{red}6}6211|{\color{red} 7}54211 \rightarrow{\color{red}6}6{\color{red}2}11|{\color{red} 75}54211 \rightarrow {\color{red}6}6{\color{red}21}1|{\color{red} 755}4211 \rightarrow {\color{red}6}6{\color{red}211}|{\color{red} 7554}211.\]
Thus, the corresponding unpaired $6$ in the 2nd factor will be moved into the $3$rd factor. To summarize,
\[  f_2(T)=  \begin{ytableau}
 \ &&&&&&&&4&4&5\\
&&&&1&1&2&&&&6\\
&&1&1&2&&&4&5&6&7\\
1&1&2&3&3&3&4&6&6&& \end{ytableau}\, \quad \text{ and } \quad \hat{f}_2(\widehat{\kappa}_{|\rho|}) = 544|6211|7654211|664333211.\]

\end{example}

Thus the directed graph on $\PCW_{n,k}$ with edges created using $\hat{f}_i$ operators matches the inherited the crystal structure established in Equation~\eqref{eq:parsed_words_inherited_fi}. 
We now refer to $\hat{f}_i$ as $f_i$, having established they are crystal operators. As a corollary, the partial inverse $\hat{e}_i$ is also a crystal operator, and so we likewise now refer to $\hat{e}_i$ as $e_i$.  See Figure~\ref{fig:PCW (2,1,1) } for an example.


\subsection{Quasicrystals on parsed content words} 

Quasicrystals in the parsed content words model are inherited from the crystal on words via Equation~\eqref{eq:parsed_words_inherited_fi}.  If  $\binom{\kappa}{\rho}$ and $\binom {\kappa'} {\rho'}$ are nodes in the same connected component of the crystal on biwords, then their content words are the same; that is, $\kappa = \kappa'$. 
However, in taking the dual biword to compute the parsed content word, the letters in $\kappa$ may get rearranged so that $\widehat{\kappa}_{\rho} \neq \widehat{\kappa}_{\rho'}$. The next proposition states that  $\widehat{\kappa}_\rho = \widehat{\kappa}_{\rho'}$ if and only if $\rho, \rho'$ are in the same quasicrystal in the words model, and therefore shows that quasicrystals in the parsed content words model are easily identified as induced subgraphs on nodes sharing a common underlying word.

\begin{theorem}\label{thm:quasi constant content}
Two parsed content words $\widehat{\kappa}_{|\rho|}, \widehat{\kappa}_{|\rho'|} \in \PCW_{n,k}$ lie in the  same quasicrystal if and only if $\widehat{\kappa}_{\rho}= \widehat{\kappa}_{\rho'}$. That is, each quasicrystal in the parsed content word model consists of all parsings of a common underlying word.
\end{theorem}

For example, consider the crystal on words in Figure~\ref{fig:quasi_words_pcw_comp}. The quasicrystal structure is not especially evident without computing the descent composition of each node (e.g. computing the standardization of each node, identifying appropriate maximal runs, applying the Krob-Thibbon quasi-ribbon insertion algorithm to each node to convert to the quasi-ribbons). However, the quasicrystal structure is immediately apparent in the parsed minimal content word model. Examining the isomorphic crystal component on the minimally parsed $\PCW_{(4,4)}$ model given below, we can identify the quasicrystals as the three strips comprised of parsings of the common underlying words $3123, 1323,$ and $1233$, respectively. 

Note that (as we show in Corollary \ref{cor: highest weight elemnt in quasi crystal PCW model}), the highest weight element of each quasicrystal is the minimal parsing of the underlying word, with bars only between strict ascents: $31|2|3, 1|32|3, 1|2|33$.

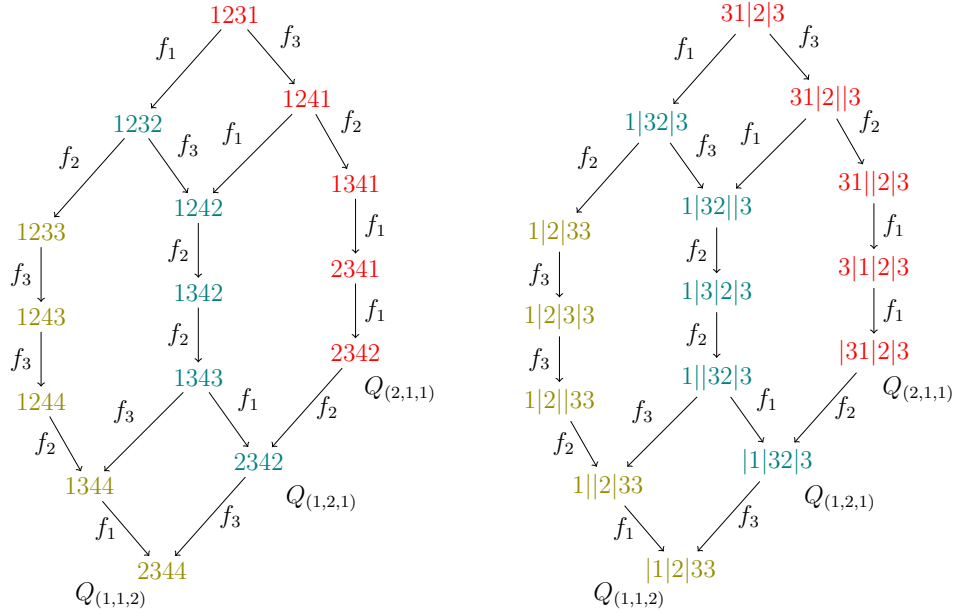
\begin{figure}[H] 
    \centering
\scalebox{0.8}{\begin{tikzpicture}[
    node/.style={inner sep=3pt}, 
    every edge/.style={draw},
    label distance=-2pt
]
\node[node] (A) at (0,0) {\color{red}1231};
\node[node] (B) at (-1.6,-1.8) {\color{teal}1232};
\node[node] (C) at ( 1.2,-1.4) {\color{red}1241};
\node[node] (D) at (-3.2,-3.6) {\color{olive}1233};
\node[node] (E) at ( -0.6,-3.2) {\color{teal}1242};
\node[node] (F) at ( 2.0,-2.8) {\color{red}1341};
\node[node] (G) at (-3.2,-5.0) {\color{olive}1243};
\node[node] (H) at ( -0.6,-4.6) {\color{teal}1342};
\node[node] (I) at ( 2.0,-4.2) {\color{red}2341};
\node[node] (J) at (-3.2,-6.4) {\color{olive}1244};
\node[node] (K) at (-0.6,-6.0) {\color{teal}1343};
\node[node] (L) at ( 2.0,-5.6) {\color{red}2342};
\node[node] (M) at (-2.4,-7.8) {\color{olive}1344};
\node[node] (N) at ( 0.4,-7.4) {\color{teal}2342};
\node[node] (O) at (-1.2,-9.2) {\color{olive}2344};
\draw[->] (A) -- node[midway, above left] {$f_1$} (B);
\draw[->] (A) -- node[midway, above right] {$f_3$} (C);
\draw[->] (B) -- node[midway, above left] {$f_2$} (D);
\draw[->] (B) -- node[midway, above right] {$f_3$} (E);
\draw[->] (C) -- node[midway, above left] {$f_1$} (E);
\draw[->] (C) -- node[midway, above right] {$f_2$} (F);
\draw[->] (D) -- node[midway, left] {$f_3$} (G);
\draw[->] (E) -- node[midway, left] {$f_2$} (H);
\draw[->] (F) -- node[midway, right] {$f_1$} (I);
\draw[->] (G) -- node[midway, left] {$f_3$} (J);
\draw[->] (H) -- node[midway, left] {$f_2$} (K);
\draw[->] (I) -- node[midway, right] {$f_1$} (L);
\draw[->] (J) -- node[midway, left] {$f_2$} (M);
\draw[->] (K) -- node[midway, above left] {$f_3$} (M);
\draw[->] (K) -- node[midway, above right] {$f_1$} (N);
\draw[->] (L) -- node[midway, right] {$f_2$} (N);
\draw[->] (M) -- node[midway, left] {$f_1$} (O);
\draw[->] (N) -- node[midway, right] {$f_3$} (O);

\node[anchor=west] at (2,-6.2) {{$Q_{(2,1,1)}$}};
\node[anchor=west] at (0.7,-8) { $Q_{(1,2,1)}$};
\node[anchor=east] at (-1.3,-9.6) { $Q_{(1,1,2)}$};
\end{tikzpicture}}
\qquad 
\scalebox{0.8}{\begin{tikzpicture}[
    node/.style={inner sep=3pt}, 
    every edge/.style={draw},
    label distance=-2pt
]
\node[node] (A) at (0,0) {\color{red}$31|2|3$};
\node[node] (B) at (-1.6,-1.8) {\color{teal}$1|32|3$};
\node[node] (C) at ( 1.2,-1.4) {\color{red}$31|2||3$};
\node[node] (D) at (-3.2,-3.6) {\color{olive}$1|2|33$};
\node[node] (E) at ( -0.6,-3.2) {\color{teal}$1|32||3$};
\node[node] (F) at ( 2.0,-2.8) {\color{red}$31||2|3$};
\node[node] (G) at (-3.2,-5.0) {\color{olive}$1|2|3|3$};
\node[node] (H) at ( -0.6,-4.6) {\color{teal}$1|3|2|3$};
\node[node] (I) at ( 2.0,-4.2) {\color{red}$3|1|2|3$};
\node[node] (J) at (-3.2,-6.4) {\color{olive}$1|2||33$};
\node[node] (K) at (-0.6,-6.0) {\color{teal}$1||32|3$};
\node[node] (L) at ( 2.0,-5.6) {\color{red}$|31|2|3$};
\node[node] (M) at (-2.4,-7.8) {\color{olive}$1||2|33$};
\node[node] (N) at ( 0.4,-7.4) {\color{teal}$|1|32|3$};
\node[node] (O) at (-1.2,-9.2) {\color{olive}$|1|2|33$};
\draw[->] (A) -- node[midway, above left] {$f_1$} (B);
\draw[->] (A) -- node[midway, above right] {$f_3$} (C);
\draw[->] (B) -- node[midway, above left] {$f_2$} (D);
\draw[->] (B) -- node[midway, above right] {$f_3$} (E);
\draw[->] (C) -- node[midway, above left] {$f_1$} (E);
\draw[->] (C) -- node[midway, above right] {$f_2$} (F);
\draw[->] (D) -- node[midway, left] {$f_3$} (G);
\draw[->] (E) -- node[midway, left] {$f_2$} (H);
\draw[->] (F) -- node[midway, right] {$f_1$} (I);
\draw[->] (G) -- node[midway, left] {$f_3$} (J);
\draw[->] (H) -- node[midway, left] {$f_2$} (K);
\draw[->] (I) -- node[midway, right] {$f_1$} (L);
\draw[->] (J) -- node[midway, left] {$f_2$} (M);
\draw[->] (K) -- node[midway, above left] {$f_3$} (M);
\draw[->] (K) -- node[midway, above right] {$f_1$} (N);
\draw[->] (L) -- node[midway, right] {$f_2$} (N);
\draw[->] (M) -- node[midway, left] {$f_1$} (O);
\draw[->] (N) -- node[midway, right] {$f_3$} (O);

\node[anchor=west] at (2,-6.2) {{$Q_{(2,1,1)}$}};
\node[anchor=west] at (0.7,-8) { $Q_{(1,2,1)}$};
\node[anchor=east] at (-1.3,-9.6) { $Q_{(1,1,2)}$};
\end{tikzpicture}}
   
    \caption{Irreducible crystal components of highest weight $(2,1,1)$ in the crystal on words (left) and crystal on parsed content words (right). The quasicrystal decomposition is highlighted by the colors of the nodes: $Q_{(1,1,2)}$ in olive, $Q_{(1,2,1)}$ in teal, and $Q_{(2,1,1)}$ in red. In the crystal on words, the quasicrystals are found by computing the standardization of each element, whereas in the crystal on parsed content words, the quasicrystals are found by identifying the common underlying word.}
    \label{fig:quasi_words_pcw_comp}
\end{figure}

\begin{proof}
We examine two nodes $\rho, f_i(\rho) \in [n]^{\otimes k}$ connected by an edge labeled $i$ in a quasicrystal in the model on words. Our first goal is to show that $\widehat{\kappa}_\rho = \widehat{\kappa}_{f_i(\rho)}$ whenever $\rho, f_i(\rho)$ are in the same quasicrystal in the words model. Consider $\rho \in [n]^{\otimes k}$, and suppose there exists an index $i$ such that $f_i(\rho)$ is in the same quasicrystal as $\rho$. This implies that in $\rho$, all of the $i+1$'s appear to the left of all of the $i$'s (see Definition~\ref{def-qKashiwara-op}). Thus, $f_i(\rho)$ changes the leftmost $i$ in $\rho$ to $i+1$. Let $\binom{\kappa}{\rho} = \binom {\kappa_1\kappa_2\cdots \kappa_k }{\rho_1  \rho_2 \cdots  \rho_k}$ be a biword. Consider the biwords $\binom{\kappa} {\rho}$ and $\binom {\kappa} {f_i(\rho)}$, connected by an edge labeled $i$ in the crystal on biwords (see Equation~\eqref{eq:biwords_inherited_fi}).

    If the leftmost $i$ in $\rho$ is in position $j$ (i.e. $\rho_j = i$), then 
    \begin{align*}
        \kappa_j \geq \kappa_s &\text{ for all columns }\left( \begin{matrix} \cdots \kappa_s\cdots \\ \cdots i+1 \cdots\end{matrix} \right) \text{ in }\binom{\kappa}{\rho}, \text{ and}\\
        \kappa_j \leq \kappa_t &\text{ for all columns }\left( \begin{matrix} \cdots \kappa_t \cdots \\ \cdots i\cdots \end{matrix} \right) \text{ in } \binom{\kappa}{\rho}. 
    \end{align*}

    Thus, when $\rho$, and $f_i(\rho)$ are sorted into weakly increasing order in the dual biwords $\widehat{\binom {{\kappa}} {\rho}}$ and $\widehat{\binom {\kappa} {f_i(\rho)}}$, $\kappa_j$ is placed below the last $i$ in $\rho^*$ in $\binom{\rho^* } {\widehat{\kappa}_\rho}$ and below the first $i+1$ in $f_i(\rho)^*$ in $\binom{f_i(\rho)^* }{\widehat{\kappa}_{f_i(\rho)}}$.  Therefore, $\widehat{\kappa}_\rho = \widehat{\kappa}_{f_i(\rho)}$  and furthermore, we may conclude that there is not an ascent between the $i$-th and $(i+1)$-st factor in $\widehat{\kappa}_{|\rho|}$ and in $ \widehat{\kappa}_{|f_i(\rho)|}$.  Since quasicrystals are connected subgraphs (see Proposition~\ref{prop:quasicrystals_are_connected_subgraphs}), this proves the forward direction. 

Our second goal is to show that quasicrystals contain all of the parsings of the common underlying word. That is, we show that if $w = \widehat{\kappa}_{\rho} = \widehat{\kappa}_{\rho'}$, then $\rho, \rho'$ are in the same quasicrystal in the words model. Suppose that $\widehat{\kappa}_\rho = \widehat{\kappa}_{\rho'}$ for two words $\rho, \rho' \in [n]^{\otimes k}$, and let $\std(\rho), \std(\rho')$ denote the standardizations of $\rho, \rho'$ (see Definition~\ref{def:standardization-words-ssyt}.) Since $\rho, \std(\rho)$ are in the same quasicrystal in the words model, we know $\widehat{\kappa}_\rho = \widehat{\kappa}_{\std(\rho)}$. Similarly, $\widehat{\kappa}_{\rho'} = \widehat{\kappa}_{\std(\rho')}$. Therefore, $w = \widehat{\kappa}_{\std(\rho)} = \widehat{\kappa}_{\std(\rho')}$, which implies
\[
\binom{\std(\rho)^* }{w} = \binom{\std(\rho')^* }{ w},
\]
and so under the dual involution on biwords,
\[
\widehat{\left( \begin{matrix} \std(\rho)^* \\ w \end{matrix} \right)} =\left( \begin{matrix} w^* \\ \std(\rho)\end{matrix} \right) =\left( \begin{matrix} w^* \\ \std(\rho')\end{matrix} \right) = \widehat{ \left( \begin{matrix} \std(\rho')^* \\ w \end{matrix} \right)}.
\]
Therefore, $\std(\rho) = \std(\rho')$, proving that $\rho, \rho'$ are in the same quasicrystal in the words model. Thus, if $\widehat{\kappa}_\rho = \widehat{\kappa}_{\rho'}$, then $\widehat{\kappa}_{|\rho|}$ and $\widehat{\kappa}_{|\rho'|}$ are in the same quasicrystal in the parsed content words model. 
\end{proof}

Theorem~\ref{thm:quasi constant content} allows us to characterize quasi-Kashiwara operators in the parsed content words model, extending Definition~\ref{def-qKashiwara-op} of Cain--Malheiro. This characterization then allows us to characterize the highest weight node in a given quasicrystal component.

\begin{corollary}\label{characterizing_quasi-kashi_operators_in_PCW}
For $1 \leq i \leq n-1$, the \emph{quasi-Kashiwara} operator $\Tilde{f}_i$ is $\NULL$ on parsed content words $\omega_\mathcal{P}$ with parsing multiset $\mathcal{P} = \{\!\{p_1 \leq \ldots \leq p_{i-1} < p_i < p_{i+1} \leq \ldots \leq p_{n-1}\}\!\}$ in which the position $p_i$ of the $i$-th parsing bar is an ascent of $\omega$. Otherwise, $\Tilde{f}_i$ acts on the parsed content word by moving the parsing bar one position to the left; i.e. $\Tilde{f}_i(\omega_{\mathcal{P}}) = \omega_{\mathcal{S}}$
where $\mathcal{S} = (\mathcal{P}\setminus \{p_i\}) \cup \{p_i-1\}$.

    \end{corollary}

\begin{proof} 
   Let $\widehat{\kappa}_{|\rho|}$ be a parsed content word. By Theorem~\ref{thm:quasi constant content},  $\Tilde{f}_i(\widehat{\kappa}_{|\rho|}) \neq \NULL$ if and only if $\widehat{\kappa}_\rho = \widehat{\kappa}_{f_i(\rho)}$. The equality of these underlying words is only possible if there is not an ascent between the smallest (rightmost) letter in the $i$-th factor and the largest (leftmost) letter in the $(i+1)$-st factor of $\widehat{\kappa}_{|\rho|}$. If there is no ascent between the factors, then there is no pairing of letters between factors and the smallest (rightmost) letter in the $i$-th factor moves into the $(i+1)$-st factor, as described in the local crystal operator $f_i$ on parsed content words from Definition~\ref{Def: f_i on PCW}. This is equivalent to exchanging the position for the $i$-th parsing bar $p_i$ with $p_i - 1$ in the parsing multiset.
\end{proof}

\begin{corollary}\label{cor: highest weight elemnt in quasi crystal PCW model}
 The highest weight element within a quasicrystal in the parsed content words model is the minimal parsing of the common underlying word (Definition \ref{Def:parsed word}). 

 \label{quasi highest weight}
\end{corollary}
\begin{proof} 
 Analogous arguments and results to Corollary~\ref{characterizing_quasi-kashi_operators_in_PCW} imply that the $e_i$ operators for $1 \leq i \leq n-1$ acting within a quasicrystal move parsing bars to the \emph{right}. Each quasicrystal in the parsed content words model contains a unique minimally parsed word in which there is exactly one parsing bar at each ascent in the underlying word (and the remaining parsing bars are placed at the end of the word), which is to say the parsing bars are as right justified as possible. Thus, the application of $e_i$ for $1 \leq i \leq n-1$ to the minimally parsed word in a quasicrystal would either result in $\NULL$ or a change to the underlying word. 
\end{proof}

\begin{remark}
For any quasicrystal indexed by descent composition $\alpha$, the highest weight element is the unique node of weight $\alpha$. Corollary~\ref{cor: highest weight elemnt in quasi crystal PCW model} states that highest weight elements of quasicrystals in the parsed content words model can be immediately identified by whether or not the parsed content word is minimally parsed.
In the crystal on words model, determining if an arbitrary word is the highest weight element in its quasicrystal requires a standardization procedure (more accurately described as destandardization). In the crystal on semistandard Young tableaux, the highest weight elements are the tableaux with only $i$'s in their $i$-th maximal horizontal band. 
\end{remark}

\begin{remark}
To find the descent composition $\alpha = (\alpha_1, \ldots, \alpha_\ell)$ indexing a quasicrystal $Q_\alpha$ in the crystal on words and semistandard Young tableaux models, one must either compute the descent set of the standardization of any node in $Q_\alpha$, or identify the weight of the unique highest weight element in $Q_\alpha$. 

 However, in a quasicrystal on parsed content words, the descent composition indexing the quasicrystal can be computed from the ascent set of the common underlying word, and thus found from any node in the quasicrystal without requiring a standardizaton process. 
\end{remark}

Various approaches have been used to study the relations between the quasicrystals within a given crystal (see \cite{BCDS},\cite{FMG23}). In a crystal of SSYT (represented by the \emph{distinguished crystal}, see Definition~\ref{Def:distinguished crystal}, Proposition~\ref{prop_distinguished_ptab}, and Section~\ref{subsection-Distinguished Crystal on Parsed Words}), a quasicrystal can be associated to a \emph{standard node} in the quasicrystal, which is shown to be the unique SYT in the quasicrystal. Then, using results of Assaf (see~\cite{Assaf15}) we can put a structure of a \emph{dual equivalence graph} on the collection of standard nodes (and hence the quasicrystals), where two quasicrystals are connected by an edge if the associated standard nodes are elementarily dual equivalent  (see Definition~\ref{Def:Dual equiv}). In our model of parsed content words, we are able to (1) interpret this dual equivalence graph structure as one \emph{also} induced by elementary Knuth equivalence (on appropriate words), and (2) easily determine this graph structure for \emph{all} crystals of parsed content words, and not only those in the distinguished crystal (discussed in the next section).

\begin{theorem}
    Let $Q_1, \ldots , Q_{\ell}$ be the quasicrystals of parsed content words in a fixed irreducible crystal ${\mathcal B}_{\lambda}$ of highest weight $\lambda$, and let $\kappa_{\rho_{1}}, \ldots, \kappa_{\rho_{\ell}}$ be the underlying words of each quasicrystal (so $Q_i$ is composed of the parsings of $\kappa_{\rho_{i}}$). Then the set $\{ \kappa_{\rho_{1}}, \ldots, \kappa_{\rho_{\ell}} \}$ forms a complete Knuth equivalence class.
\end{theorem}

\begin{proof}Suppose $n = k$, for now. Then, as in \cite{Assaf15} and \cite{R13} and noted above, each quasicrystal in a crystal of SSYT of highest weight $\lambda$ (where $\lambda \vdash k$) contains a unique SYT called the \emph{standard node}. Viewed as a ptableau, by Proposition~\ref{prop_distinguished_ptab}, the standard node  is dual to a ptableau T in the distinguished crystal with biword ${\kappa \choose \rho}$ and $\rho \in S_k$. Therefore in any associated crystal of biwords there is a unique biword in each quasicrystal whose row word forms a one-row permutation. By the comment prior to Theorem 2.4 in \cite{R13}, the permutations determined by the standard nodes of the distinguished crystal are dual equivalent and   form a complete dual equivalence class. Further, under RSK insertion, for permutations $\omega \in S_k$ mapping to a pair of standard Young tableaux $(P,Q)$, the SYT $Q$ is constant on dual equivalence classes. As is known (see Theorem 8.7 in~\cite{BS}) this is sufficient to imply that dual equivalent words lie in the same irreducible crystal, which implies the same for biwords built over such row words. By Theorem~\ref{theorem biword dual Knuth}, the underlying words of the parsed content words of these dual equivalent standard nodes are Knuth equivalent and form a complete Knuth equivalence class. But then by Theorem~\ref{thm:quasi constant content}, the underlying content words of these Knuth equivalent standard nodes are constant on each quasicrystal, proving that the collection of underlying content words among the quasicrystals of parsed content words form a Knuth equivalence class. Lastly, we note that the determination of the underlying content words is independent of $n$, provided $n \geq k$. That is, if $n> k$, each crystal (and quasicrystal) may have more nodes, but these will only alter the parsing of the underlying content words, and not the word itself.

\end{proof}

\subsection{Distinguished crystal in the parsed content words model}\label{subsection-Distinguished Crystal on Parsed Words}
We call the crystal consisting of the reading words of the ptableaux in the distinguished crystal (Def. \ref{Def:distinguished crystal}) the \emph{distinguished crystal of parsed content words}. 
While we have emphasized that restricting crystal models to one isomorphism type can introduce some difficulties with problems involving multiple isomorphic crystals, analyzing versions of the distinguished crystal can still provide new insights. 

Let ${\mathcal B_\lambda}(T_{\lambda})$ be the distinguished crystal of ptableaux of highest weight $\lambda$ as defined in Definition~\ref{Def:distinguished crystal}, and let $\widehat{\kappa}_{|\lambda|}$ be the parsed content word associated to $T_\lambda$. (For example, if $\lambda = (4,2,1,1)$, then $\widehat{\kappa}_{|\lambda|} = 1111|22|3|4$.) With this, let ${\mathcal B}(\widehat{\kappa}_{|\lambda|})$ be the associated crystal of parsed content words with nodes associated to the ptableaux in the distinguished crystal. See Figure~\ref{fig:distinguished_pcw_crystal} for an example.
Clearly, every underlying word in ${\mathcal B}(\widehat{\kappa}_{|\lambda|})$ is a word of weight $\lambda$, but we can say more.

\begin{definition}
   A word $w = w_1\ldots w_k\in [n]^{\otimes k}$ is a \emph{(left) Yamanouchi $\lambda$-word} if $\wt(w) = \lambda$ and each prefix of $w$ contains at least as many $i$'s as ($i+1$)'s.  \label{Yamanouchi word def}
\end{definition}

\begin{example}
    Let $\lambda = (5,3,3,2,1)$. Then $11111222333445$ and $11212234533141$ are Yamanouchi $\lambda$-words. However $12121135433412$ is not, because the prefix $12121135$ contains more $5$'s than $4$'s and the prefix  $12121135433$ contains more $3$'s than $2$'s.
\end{example}

In the crystal on words $[n]^{\otimes k}$, each isomorphic crystal of highest weight $\lambda$ is indexed by a Yamanouchi $\lambda$-word.
 We next show that in the parsed content words model, all Yamanouchi $\lambda$-words appear within
 one irreducible crystal.

\begin{proposition}
    For a partition $\lambda$, all Yamanouchi $\lambda$-words appear within the distinguished crystal ${\mathcal B}(\widehat{\kappa}_{|\lambda|})$ as the underlying words indexing each of the quasicrystals in ${\mathcal B}(\widehat{\kappa}_{|\lambda|})$. 
    \label{prop:Yamanouchi criterion}
\end{proposition}

\begin{proof} 
    Suppose $T$ is a ptableau in the distinguished crystal ${\mathcal B}_\lambda(T_{\lambda})$. We first show that the underlying word of the parsed content word associated to $T$ is a Yamanouchi $\lambda$-word.  Let $\widehat{\kappa}_{|\rho|}$ be the reading word of $T$. Since $T$ is in the distinguished crystal, by Proposition~\ref{prop_distinguished_ptab} (1) and (2) we know that the reading word of $T$ contains $\lambda_i$ copies of the letter $i$ and that $T$ satisfies the word condition (Definition \ref{def:word condition}), namely that every $i+1$ in $T^\rightarrow$ has an $i$ in a row above it. Thus, the reading word of $T$ contains at least as many $i$'s as $(i+1)$'s in every prefix, and hence the underlying word $\widehat{\kappa}_{\rho}$ is a Yamanouchi $\lambda$-word. Conversely, every minimally parsed Yamanouchi $\lambda$-word is the reading word of a ptableau whose right-justified form satisfies the word condition, and hence is a node in the distinguished crystal. 
\end{proof}

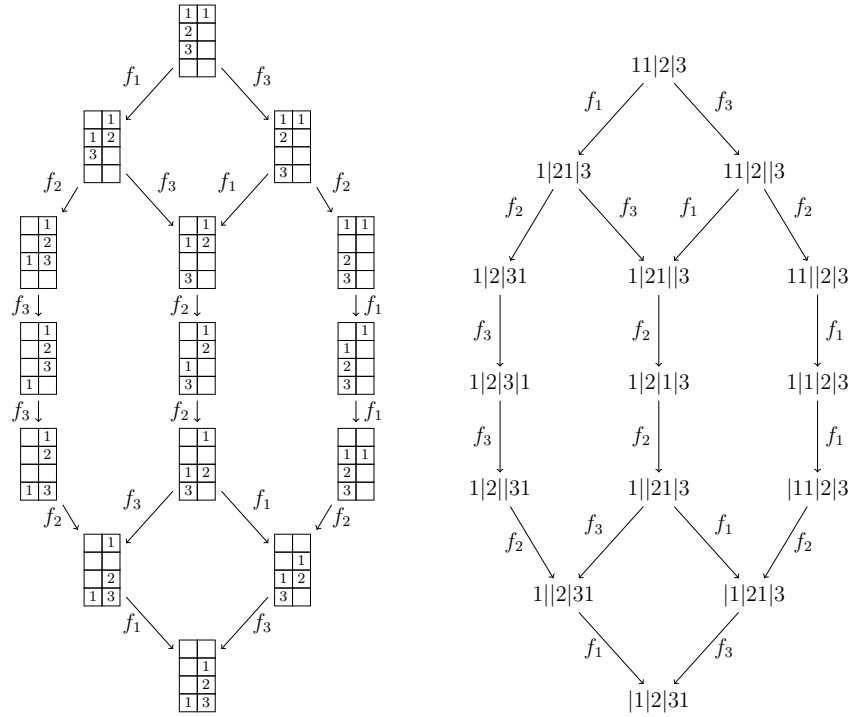
\begin{figure}
    \centering
       \scalebox{0.7}{\begin{tikzpicture}[
    node/.style={inner sep=3pt}, 
    every edge/.style={draw},
    label distance=-2pt
]
\node[node] (A) at (0,0) {\ytableaushort{{1}{1},{2}{ },{3}{ },{ }{ }}};
\node[node] (B) at (-1.8,-2.0) {\ytableaushort{{}{1},{1}{2},{3}{ },{ }{ }}};
\node[node] (C) at ( 1.8,-2.0) {\ytableaushort{{1}{1},{2}{ },{ }{ },{3}{ }}};
\node[node] (D) at (-3.0,-4.0) {\ytableaushort{{}{1},{}{2},{1}{3},{ }{ }}};
\node[node] (E) at ( 0.0,-4.0) {\ytableaushort{{}{1},{1}{2},{ }{ },{3}{ }}};
\node[node] (F) at ( 3.0,-4.0) {\ytableaushort{{1}{1},{ }{ },{2}{ },{3}{ }}};
\node[node] (G) at (-3.0,-6.0) {\ytableaushort{{}{1},{}{2},{}{3},{1}{}}};
\node[node] (H) at ( 0.0,-6.0) {\ytableaushort{{}{1},{ }{2},{1}{ },{3}{ }}};
\node[node] (I) at ( 3.0,-6.0) {\ytableaushort{{ }{1},{1}{ },{2}{ },{3}{ }}};
\node[node] (J) at (-3.0,-8.0) {\ytableaushort{{}{1},{}{2},{ }{ },{1}{3}}};
\node[node] (K) at ( 0.0,-8.0) {\ytableaushort{{}{1},{ }{ },{1}{2},{3}{ }}};
\node[node] (L) at ( 3.0,-8.0) {\ytableaushort{{ }{ },{1}{1},{2}{ },{3}{ }}};
\node[node] (M) at (-1.8,-10.0) {\ytableaushort{{}{1},{ }{ },{}{2},{1}{3}}};
\node[node] (N) at ( 1.8,-10.0) {\ytableaushort{{ }{ },{}{1},{1}{2},{3}{ }}};
\node[node] (O) at (0,-12.0) {\ytableaushort{{ }{ },{}{1},{}{2},{1}{3}}};
\draw[->] (A) -- node[midway, above left] {$f_1$} (B);
\draw[->] (A) -- node[midway, above right] {$f_3$} (C);
\draw[->] (B) -- node[midway, above left] {$f_2$} (D);
\draw[->] (B) -- node[midway, above right] {$f_3$} (E);
\draw[->] (C) -- node[midway, above left] {$f_1$} (E);
\draw[->] (C) -- node[midway, above right] {$f_2$} (F);
\draw[->] (D) -- node[midway, left] {$f_3$} (G);
\draw[->] (E) -- node[midway, left] {$f_2$} (H);
\draw[->] (F) -- node[midway, right] {$f_1$} (I);
\draw[->] (G) -- node[midway, left] {$f_3$} (J);
\draw[->] (H) -- node[midway, left] {$f_2$} (K);
\draw[->] (I) -- node[midway, right] {$f_1$} (L);
\draw[->] (J) -- node[midway, left] {$f_2$} (M);
\draw[->] (K) -- node[midway, above left] {$f_3$} (M);
\draw[->] (K) -- node[midway, above right] {$f_1$} (N);
\draw[->] (L) -- node[midway, right] {$f_2$} (N);
\draw[->] (M) -- node[midway, left] {$f_1$} (O);
\draw[->] (N) -- node[midway, right] {$f_3$} (O);
\end{tikzpicture}}
\qquad
\scalebox{0.7}{\begin{tikzpicture}[
    node/.style={inner sep=3pt}, 
    every edge/.style={draw},
    label distance=-2pt
]
\node[node] (A) at (0,0) {$11|2|3$};
\node[node] (B) at (-1.8,-2.0) {$1|21|3$};
\node[node] (C) at ( 1.8,-2.0) {$11|2||3$};
\node[node] (D) at (-3.0,-4.0) {$1|2|31$};
\node[node] (E) at ( 0.0,-4.0) {$1|21||3$};
\node[node] (F) at ( 3.0,-4.0) {$11||2|3$};
\node[node] (G) at (-3.0,-6.0) {$1|2|3|1$};
\node[node] (H) at ( 0.0,-6.0) {$1|2|1|3$};
\node[node] (I) at ( 3.0,-6.0) {$1|1|2|3$};
\node[node] (J) at (-3.0,-8.0) {$1|2||31$};
\node[node] (K) at ( 0.0,-8.0) {$1||21|3$};
\node[node] (L) at ( 3.0,-8.0) {$|11|2|3$};
\node[node] (M) at (-1.8,-10.0) {$1||2|31$};
\node[node] (N) at ( 1.8,-10.0) {$|1|21|3$};
\node[node] (O) at (0,-12.0) {$|1|2|31$};
\draw[->] (A) -- node[midway, above left] {$f_1$} (B);
\draw[->] (A) -- node[midway, above right] {$f_3$} (C);
\draw[->] (B) -- node[midway, above left] {$f_2$} (D);
\draw[->] (B) -- node[midway, above right] {$f_3$} (E);
\draw[->] (C) -- node[midway, above left] {$f_1$} (E);
\draw[->] (C) -- node[midway, above right] {$f_2$} (F);
\draw[->] (D) -- node[midway, left] {$f_3$} (G);
\draw[->] (E) -- node[midway, left] {$f_2$} (H);
\draw[->] (F) -- node[midway, right] {$f_1$} (I);
\draw[->] (G) -- node[midway, left] {$f_3$} (J);
\draw[->] (H) -- node[midway, left] {$f_2$} (K);
\draw[->] (I) -- node[midway, right] {$f_1$} (L);
\draw[->] (J) -- node[midway, left] {$f_2$} (M);
\draw[->] (K) -- node[midway, above left] {$f_3$} (M);
\draw[->] (K) -- node[midway, above right] {$f_1$} (N);
\draw[->] (L) -- node[midway, right] {$f_2$} (N);
\draw[->] (M) -- node[midway, left] {$f_1$} (O);
\draw[->] (N) -- node[midway, right] {$f_3$} (O);
\end{tikzpicture}}
    \caption{The distinguished crystal $\mathcal{B}_{(2,1,1)}(T_{(2,1,1)})$ and $\mathcal{B}_{(2,1,1)}(\widehat{\kappa}_{|(2,1,1)|})$. Note the underlying words $1123, 1213, 1231$ are all of the Yamanouchi $(2,1,1)$-words.}
    \label{fig:distinguished_pcw_crystal}
\end{figure}

By Proposition~\ref{prop:Yamanouchi criterion}, the nodes in the crystal ${\mathcal B}(\widehat{\kappa}_{|\lambda|})$ on parsed content words are all of the parsings of Yamanouchi $\lambda$-words. The crystal model on parsed content words thus gives a new combinatorial interpretation of the monomials and the fundamental quasisymmetric functions appearing in a Schur function.
\begin{proposition}
   The Schur function $s_\lambda$ can be obtained as:
    \[
    s_\lambda = \sum_{\omega_{\mathcal{P}} \in {\mathcal B}(\widehat{\kappa}_{|\lambda|})} x^{\wt_{\mathcal{P}}(\omega_{\mathcal{P}})},
    \]
    when we allow an unlimited number of parsing bars. 
        The multiplicity of the fundamental quasisymmetric function $F_{\alpha}$ in $s_{\lambda}$ is the number of distinct Yamanouchi $\lambda$-words that appear as underlying words in ${\mathcal B}(\widehat{\kappa}_{|\lambda|})$ so that the weight of their minimal parsing is $\alpha$.
    That is,  

    \[
s_\lambda  = \sum_{\substack{\omega_{\mathcal{P}_{\text{min}}} \in {\mathcal B}(\widehat{\kappa}_{|\lambda|}) }} F_{\wt_{\mathcal{P}}(\omega_{\mathcal{P}_{\text{min}}})}.
\]
where $\omega$ is a Yamanouchi $\lambda$-word.
\end{proposition}

\begin{corollary} 
    There is bijection between parsed content words of weight $\lambda$ and SSYT of shape $\lambda$ given by taking a parsed content word $\widehat{\kappa}_{|\rho|}$ to the SSYT associated to the biword ${\rho^* \choose \widehat{\kappa}_{\rho}}$. This bijection is an isomorphism of crystals. In particular, this map gives a bijection between Yamanouchi $\lambda$-words and standard Young tableaux of shape $\lambda$.
\end{corollary}

It remains open to classify the underlying words contained in an irreducible crystal of parsed content words beyond the distinguished crystal.

\section{Schur-Weyl Duality}\label{Section: S-W duality}

    \subsection{What is a Combinatorial View of Schur-Weyl Duality?}
   
Let $V$ denote an $n$-dimensional complex vector space. We analyze combinatorially the isomorphism of $\mathfrak{gl}_n \times S_k$ representations known as \emph{Schur-Weyl Duality:}
\[ V^{\otimes k} \cong \bigoplus_{|\lambda|=k} ({B}_\lambda \otimes {S}^{\lambda}),\]
where ${B}_{\lambda}$ denotes an irreducible $\mathfrak{gl}_n$ representation of highest weight $\lambda$ and ${S}^\lambda$ denotes an irreducible representation of the symmetric group $S_k$ of shape $\lambda$ (where $\lambda$ has at most $n$ parts). The Lie algebra $\mathfrak{gl}_n$ acts on $V^{\otimes k}$ via $g(\vec{v}_1 \otimes \cdots \otimes v_k) = (g \vec{v}_1) \otimes \cdots (g\vec{v}_k)$, and $\sigma \in S_k$ acts by permuting tensor factors: $\sigma(\vec{v}_1 \otimes \vec{v}_2 \cdots \otimes v_k) = \vec{v}_{\sigma(1)} \otimes \vec{v}_{\sigma(2)} \otimes \cdots \otimes \vec{v}_{\sigma(k)}$. These actions commute with each other.
\medskip

In particular, in the isomorphism above, we have:
\begin{enumerate}
    \item The multiplicity of ${S}^\lambda$ in $V^{\otimes k}$ is 
    $\dim({B}_{\lambda})$.
    \item The multiplicity of ${B}_{\lambda}$ in $V^{\otimes k}$ is 
    $\dim({S}^{\lambda})$.
    \item The action of $f_i,e_i\in\mathfrak{gl}_n$, when nonzero, maps 
    copies of ${S}^\lambda$ to copies isomorphic to ${S}^\lambda$.
    \item The action of $\sigma\in S_k$ maps copies of ${B}_{\lambda}$ 
    to copies isomorphic to ${B}_{\lambda}$.
\end{enumerate}

By a \emph{copy} of ${S}^{\lambda}$ we mean a subspace 
$u\otimes {S}^{\lambda}$ in the tensor product above, 
and similarly for copies of ${B}_{\lambda}$.

Crystal graphs act as a combinatorial proxy for $\mathfrak{gl}_n$, allowing many questions about
its representation theory to be studied with a graph model. To analyze Schur-Weyl
duality combinatorially, we need an analogous graph model for $S_k$ representations. We can then study how these two collections of graphs interact. We use dual
equivalence graphs (denoted $\mathcal{S}^{\lambda}$ and described more
fully in Section~\ref{section dual equiv}) as this combinatorial proxy for
$S_k$ representations. We summarize the correspondence between the representation theoretic and combinatorial settings that we will establish in the following ``dictionary'':

\renewcommand{\arraystretch}{1.3} 
\[ \begin{array}{|c|c|}\hline
\hbox{Representations} & \hbox{Combinatorics} \\
\hline \hline
V^{\otimes k}, \ \text{dim}(V) = n & [n]^{\otimes k} \\ \hline
\hbox{Basis vector $\vec{v}_{\rho_1} \otimes \vec{v}_{\rho_2} \otimes \cdots \otimes \vec{v}_{\rho_k}$} & \hbox{Word $\rho = (\rho_1 \rho_2 \cdots \rho_k)$} \\ \hline
\begin{array}{c} \hbox{Irreducible $\mathfrak{gl}_n$ Representation}
\\{ B}_{\lambda} \end{array} & \begin{array}{c}\hbox{Connected Crystal Graph} \\ {\mathcal {B}}_{\lambda} \end{array} \\ \hline
\text{dim}({B}_{\lambda}) & \hbox{Number of nodes in $\mathcal{B}_{\lambda}$}\\ \hline
\begin{array}{c} \hbox{Irreducible $S_k$ Representation}\\
{ S}^{\lambda} \end{array} & \begin{array}{c} \hbox{Dual Equivalence Graph} \\ {\mathcal{S}}^{\lambda} \end{array} \\ \hline
\text{dim}({S}^{\lambda}) & \hbox{Number of nodes in $\mathcal{S}^{\lambda}$} \\ \hline
\begin{array}{c} \hbox{Tensor Product of Representations}\\
{ B}_{\lambda} \otimes { S}^{\lambda}
\end{array}& \begin{array}{c} \hbox{Cartesian Product of Graphs} \\ {\mathcal B}_{\lambda} \times {\mathcal S}^{\lambda} \end{array} \\ \hline
\begin{array}{c} \hbox{Isomorphism of $\mathfrak{gl}_n \times S_k$ Representations} \\ V^{\otimes k} \cong \bigoplus_{\lambda} \left( { B}_{\lambda} \otimes { S}^{\lambda} \right) \end{array} 
& 
\begin{array}{c} \hbox{Isomorphism of Graphs} \\
\hbox{$[n]$}^{\otimes k} \cong \bigsqcup_{\lambda} ({\mathcal B}_{\lambda} \times {\mathcal S}^{\lambda}) \end{array} \\ \hline
\begin{array}{c}
\hbox{ Lie operators $e_i$ and $f_i$ map copies of ${S}^{\lambda}$}\\
\hbox{to isomorphic copies}
\end{array} & \begin{array}{c}
\hbox{Crystal operators $e_i$ and $f_i$ map copies of} \\
\hbox{${\mathcal S}^{\lambda}$ to isomorphic copies}
\end{array} \\ \hline
\begin{array}{c}
\hbox{The action of $S_k$ maps copies of ${B}_{\lambda}$}\\
\hbox{to isomorphic copies}
\end{array} &\begin{array}{c}
\hbox{Dual equivalences $D_i^{\pm}$ map copies of} \\
\hbox{${\mathcal B}_{\lambda}$ to isomorphic copies}
\end{array}\\ \hline
\begin{array}{c}
\hbox{Lie operators $e_i$ and $f_i$ commute}\\ \hbox{with $\sigma_i \in S_k$}
\end{array} & \begin{array}{c}
\hbox{Crystal operators $e_i$ and $f_i$ commute} \\
\hbox{with $D_{i}^{\pm}$}
\end{array} \\ \hline
\end{array}
\]
In the chart above, the graph structure on $[n]^{\otimes k}$ in the
isomorphism:
\[
[n]^{\otimes k} \cong
\bigsqcup_{\lambda}(\mathcal B_{\lambda}\times\mathcal S^{\lambda})
\]
is defined by connecting two nodes (words) by an edge if \emph{either}
one word is the image of the other under a single crystal operator, or if the two words are elementarily Knuth equivalent. To establish these interpretations, we first define dual equivalence graphs more intrinsically and then show how they can be used as a proxy for $S_k$ representations. Together with the crystal graph structure we provide a combinatorial proxy for Schur-Weyl duality, completing the dictionary above.  We then briefly describe how the
classical RSK correspondence can be viewed as a related model for
Schur-Weyl duality.

\subsection{Dual Equivalence Graphs and Representations of $S_k$}\label{section dual equiv} 
In this section, we use crystals determined by \emph{maximal} biwords and
their associated ptableaux to establish combinatorial connections between crystals and $S_k$ representations. Hence we assume that biwords ${\kappa \choose \rho}$ in this (and the next) section have $\kappa = 12\cdots k$, so $\widehat{\kappa}_{\rho}$ is a one-row permutation.

Many combinatorial models have been used to construct irreducible representations of symmetric groups.  The Specht module construction (see \cite{sagan2001symmetric}) uses
standard tableaux, identified with standard polytabloids, to index basis elements and count dimensions of irreducible representations of shape $\lambda$. We can identify a standard polytabloid with a unique SYT of shape $\lambda$, but this construction yields only a single representative of an isomorphism type of a particular irreducible representation. In order to enumerate and identify distinct yet isomorphic representations (as in Schur-Weyl duality), we need a more flexible combinatorial proxy.

Dual equivalence was initially proposed by Haiman \cite{Haiman1992} as an equivalence relation on permutations, dual to Knuth equivalence. Knuth equivalent permutations share the same RSK insertion tableau, whereas dual equivalent permutations share the same RSK recording tableau. Assaf \cite{Assaf15,Assaf07thesis,Assaf08fpsac-crystalops} defined a class of \emph{dual equivalence graphs} whose nodes were SYT. 
For example, the dual equivalence graph of SYT of shape $\lambda=(3,1,1)$ is:

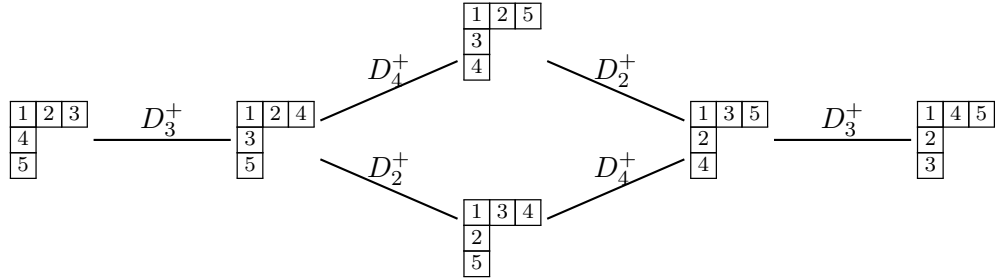
\begin{center}
\begin{tikzpicture}[  xscale=1.5,
    node/.style={inner sep=2pt},
    edge/.style={thick},
    edgelabel/.style={font=\small, inner sep=1pt}
]
\node[node] (A) at (0,0) {$\begin{ytableau} 1&2&3\\ 4\\5 \end{ytableau}$};
\node[node] (B) at (2,0) {$\begin{ytableau} 1&2&4\\ 3\\5 \end{ytableau}$};
\node[node] (C) at (4,1.3) {$\begin{ytableau} 1&2&5\\ 3\\4 \end{ytableau}$};
\node[node] (D) at (4,-1.3) {$\begin{ytableau} 1&3&4\\ 2\\5 \end{ytableau}$};
\node[node] (E) at (6,0) {$\begin{ytableau} 1&3&5\\ 2\\4 \end{ytableau}$};
\node[node] (F) at (8,0) {$\begin{ytableau} 1&4&5\\ 2\\3 \end{ytableau}$};

\draw[edge] (A) -- node[edgelabel,above] {$D_{3}^{+}$} (B);
\draw[edge] (B) -- node[edgelabel,above] {$D_{4}^{+}$}(C);
\draw[edge] (C) -- node[edgelabel,above] {$D_{2}^{+}$}(E);
\draw[edge] (E) -- node[edgelabel,above] {$D_{4}^{+}$}(D);
\draw[edge] (D) -- node[edgelabel,above] {$D_{2}^{+}$}(B);
\draw[edge] (E) -- node[edgelabel,above] {$D_{3}^{+}$}(F);

\end{tikzpicture}
\end{center}

A dual equivalence operator acts on a SYT by using the definition of dual equivalence on its (unparsed) reading word, viewed as a one-row permutation, so we can re-express the above graph in the form below:

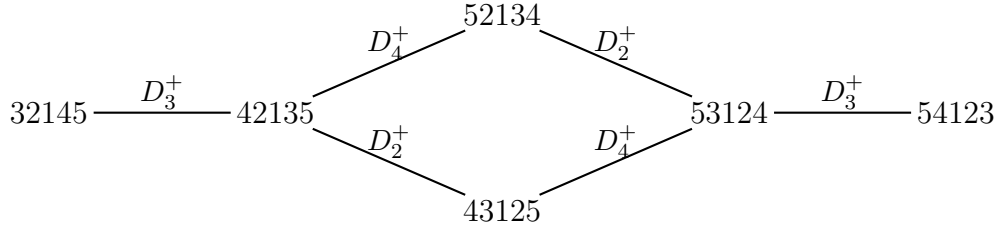
\begin{center}
\begin{tikzpicture}[  xscale=1.5,
    node/.style={inner sep=2pt},
    edge/.style={thick},
    edgelabel/.style={font=\small, inner sep=1pt}
]
\node[node] (A) at (0,0) {$32145$};
\node[node] (B) at (2,0) {$42135$};
\node[node] (C) at (4,1.3) {$52134$};
\node[node] (D) at (4,-1.3) {$43125$};
\node[node] (E) at (6,0) {$53124$};
\node[node] (F) at (8,0) {$54123$};

\draw[edge] (A) -- node[edgelabel,above] {$D_{3}^{+}$} (B);
\draw[edge] (B) -- node[edgelabel,above] {$D_{4}^{+}$}(C);
\draw[edge] (C) -- node[edgelabel,above] {$D_{2}^{+}$}(E);
\draw[edge] (E) -- node[edgelabel,above] {$D_{4}^{+}$}(D);
\draw[edge] (D) -- node[edgelabel,above] {$D_{2}^{+}$}(B);
\draw[edge] (E) -- node[edgelabel,above] {$D_{3}^{+}$}(F);

\end{tikzpicture}
\end{center}

Assaf \cite{Assaf15} gave local characterizations for dual equivalence
graphs; using these local characterizations, Roberts \cite{R13} described isomorphism classes of dual equivalence graphs on all of $S_k$. Below we recover these and other results but derive them intrinsically in the context of crystal graphs. First, we need to understand the relationship between dual equivalence and crystal operators.

\begin{proposition}\label{Prop: dual equiv - valid parsings}
    Suppose $\omega,\omega' \in S_k$ are elementarily dual equivalent one-rowed permutations, where $\omega'=D_i^+\omega$ (or $\omega'=D_i^-\omega$). Let ${\mathcal P}$ denote a fixed parsing of $\omega$, so $\omega_{\mathcal P}$ is a parsed content word (Definition~\ref{Def:parsed word}).  Then
    
    \begin{enumerate}
        \item The parsing ${\mathcal P}$ is a valid parsing of $\omega'$.
        \item The underlying words of $f_j(\omega_{\mathcal P})$ and $f_j(\omega'_{\mathcal P})$ are elementarily dual equivalent
(see Definition~\ref{Def: f_i on PCW}).
    \end{enumerate}
 \label{prop f_i to dual}
\end{proposition}
 Note that if $D_{i}^{+}\omega= \omega'$, it is not necessarily true that $D_{i}^{+}f_j (\omega) = f_j (\omega')$. However, the proof will show that if
$D_i^{+}f_j(\omega) \neq f_j(\omega')$, then
$D_i^{-}f_j(\omega) = f_j(\omega')$, and so the $f_j$ images of $\omega$ and $\omega'$ will still be elementarily dual equivalent.

 Before proving the theorem, we need the following:
\begin{lemma}\label{Lemma: ordering prop dual equiv words}
    Let $\omega =\omega_1 \cdots \omega_k \in S_k$, and $\omega'= D_{i}^{\pm} \omega=\omega_{1}' \cdots  \omega_{k}'$. Then $\omega_\ell < \omega_{\ell+1}$ if and only if $\omega_{\ell}' < \omega_{\ell+1}'$ for $\ell=1, \ldots , (k-1).$
\end{lemma}

\begin{proof} 
For ease of reading, we represent words in a set of parentheses.  

It suffices to consider the case $D_i^+\omega=\omega'$, since the case $D_i^-\omega=\omega'$ is analogous. Thus the letters in $\omega$ and $\omega'$ are identical except for the placement of the letters $i$ and $i+1$, which have been swapped:
 \[\omega = (\ldots i+1 \ldots i-1 \ldots  i \ldots )  \ \ \hbox{and} \ \ \omega' = (\ldots i \ldots i-1 \ldots  i+1 \ldots ).  \]
 In particular, $i$ and $(i+1)$ do not appear consecutively in $\omega$ or $\omega'$. From this the lemma follows: any inequality
$\omega_{\ell}<\omega_{\ell+1}$ is either identical to
$\omega_{\ell}'<\omega_{\ell+1}'$, or exactly one of the two entries $\omega_{\ell},\omega_{\ell+1}$ is $i$ or $(i+1)$, and is replaced by $(i+1)$ or $i$, respectively, in $\omega'$. In either case the inequalities hold.

\end{proof}
We can now prove Proposition \ref{Prop: dual equiv - valid parsings}.
\begin{proof} 
Assume $D_{i}^{+}\omega = \omega'$. The case $D_{i}^{-}\omega = \omega'$ follows similarly. 

By Lemma~\ref{Lemma: ordering prop dual equiv words}, the locations of all ascents and descents of $\omega$ and $\omega'$ are the same; thus they have the same minimal parsing. Since $\mathcal P$ is a valid parsing of $\omega$, it is a refinement of this shared minimal parsing, and hence
is also a valid parsing of $\omega'$. This proves the first claim.

To prove the second claim, fix the common parsing $\mathcal P$.  For clarity, we suppress the subscript $\mathcal P$ on $\omega_{\mathcal P}$ and $\omega'_{\mathcal P}$ in what follows.

Assume that $f_j(\omega) \ne \NULL$. The words $\omega$ and $\omega'$ are identical except for the position of $i$ and $i+1$:
 \[\omega = (\ldots i+1 \ldots i-1 \ldots  i \ldots )  \ \ \hbox{and} \ \ \omega' = (\ldots i \ldots i-1 \ldots  i+1 \ldots ).  \]

Let $W_j$ and $W_{j+1}$ (resp. $W_{j}'$ and $W_{j+1}'$) denote the $j$th and $(j+1)$st factors of $\omega$ (resp. $\omega'$). Thus, $(i+1)$ appears in $W_j$ if and only if $i$ appears in $W_{j}'$, and similarly for $W_{j+1}$ and $W_{j+1}'$.

The crystal operator $f_j$ on $\omega$ (and $\omega'$) changes content only in $W_j$ and $W_{j+1}$ (and similarly in $W_{j}'$ and $W_{j+1}'$).  In particular, $\omega$ and $f_j (\omega)$ are identical except one letter in $W_j$ of $\omega$ has been moved into $W_{j+1}$ in $f_j (\omega)$. 

We claim that $f_j(\omega')\ne\NULL$.  Indeed, by Lemma \ref{Lemma: ordering prop dual equiv words}, in computing $f_j (\omega)$ and $f_j (\omega')$, if $(i+1)$ appears in $W_{j}$, it is paired with a letter $a$ (to its right) in $W_{j+1}$ if and only if $i$ is paired with $a$ in $W_{j+1}'$. Similarly, if $(i+1)$ appears in $W_{j+1}$ it is paired with a letter $b$ (to its left) if and only if $i$ is paired with $b$ in $W_{j}'$. Thus, the collection of pairings in $\omega$ and $\omega'$ are the same, except that the letter paired with $(i+1)$ in $\omega$ is necessarily paired with $i$ in $\omega'$, and conversely, except in Case IIb given below, in which the $(i+1)$ in $\omega$ is paired, and an $(i-1)$ is paired in $\omega'$. In particular, when $f_j(\omega)\ne\NULL$, there is an unpaired letter in $W_j$ that can be moved to $W_{j+1}$, and likewise some unpaired letter in $W_j'$  that can be moved to $W_{j+1}'$ and so $f_j(\omega')\ne\NULL$.

As a first case, assume that $f_j$ acting on $\omega$ moves some letter $x$, with $x\in W_j$ and $x\notin\{i-1,i,i+1\}$. Then the same letter $x$ in factor $W_j'$ is moved by $f_j$ acting on $\omega'$. Thus, the relative order of $(i-1),i,(i+1)$ is unchanged upon applying $f_j$: in $\omega$ and $f_j(\omega)$ it is $(\ldots i+1\ldots i-1\ldots i\ldots)$, while in $\omega'$ and $f_j(\omega')$ it is $(\ldots i\ldots i-1\ldots i+1\ldots)$. Hence, $D_i^+f_j(\omega)=f_j(\omega'),$ so $f_j(\omega)$ and $f_j(\omega')$ are elementarily dual equivalent.

To complete the proof, we need only consider when the moved letter of $W_j$ is in $\{i-1, i, i+1\}$. Note that $(i+1)$ and $i$ must appear in distinct weakly decreasing factors of $\omega$, with the factor containing $(i+1)$ to the left of the factor containing $i$, since the letter $i-1$ lies between them.

Further, if $(i-1), i$ or $(i+1)$ in $W_j$ is moved into $W_{j+1}$ by applying $f_j$ to $\omega$, and none of the letters $(i-1), i$ or $(i+1)$ occupies a spot in $W_{j+1}$ prior to the move, then the order of $(i-1), i$ and $(i+1)$ is the same in $\omega$ and $f_j(\omega)$:
\[ (\ldots i+1 \ldots i-1 \ldots  i \ldots ), \]
and in $\omega'$ and $f_j(\omega')$:
\[(\ldots i \ldots i-1 \ldots  i+1 \ldots ). \]
In particular, this includes the case where the moved letter is $i$ in $\omega$ , and thus $(i+1)$ in $\omega'$.

So we need only consider the effect of applying $f_j$ under the hypothesis that either $(i+1)$ or $(i-1)$ is in $W_j$ and is moved, and that $i-1$ or $i$ is in $W_{j+1}$.

\medskip
Case I. Suppose $(i-1)$ is in $W_j$, but $(i+1)$ is \emph{not} in $W_j$ (and hence $i$ is not in $W_j'$ of $\omega'$), while the letter $i$ \emph{is} in the factor $W_{j+1}$:

\begin{align*} \omega &=(\ldots (i+1) \ldots| \underbrace{\ldots (i-1) \ldots}_{W_j} |\underbrace{  \ldots i \ldots }_{W_{j+1}} |\ldots) \\
\omega' &=(\ldots i \ldots| \underbrace{\ldots (i-1) \ldots}_{W_{j}'} |\underbrace{  \ldots (i+1) \ldots }_{W_{j+1}'} |  \ldots )
\end{align*}

  We claim that $(i-1)$ is not moved into $W_{j+1}$ by applying $f_j$ to $\omega$, so this case is vacuous.  Indeed, if the letter $i$ exists in $W_{j+1}$, then it is large enough to be paired with $(i-1)$, but too small to be paired with any letter to the left of $(i-1)$ in $W_j$. Hence $(i-1)$ would be paired with $i$ or some entry in $W_{j+1}$ larger than $i$, implying that $(i-1)$ cannot be the right-most unpaired entry in $W_j$. A similar argument shows that $(i-1)$ cannot be the moved letter of $\omega'$, as it would be paired with a letter at least as big as $(i+1)$. 

 \medskip
 Case II. Suppose that \emph{both} $(i+1)$ and $(i-1)$ are in $W_j$ of $\omega$ (where they necessarily appear consecutively because $W_j$ is weakly decreasing) and that $i$ is in factor $W_{j+1}$. Correspondingly $i$ and $(i-1)$ appear consecutively in $W_j'$ and $(i+1)$ appears in $W_{j+1}'$. 

 \begin{align*} \omega &=(\ldots| \underbrace{ \ldots (i+1) (i-1) \ldots}_{W_j} |\underbrace{ \ldots i \ldots}_{W_{j+1}} | \ldots ) \\
\omega' &=( \ldots| \underbrace{\ldots i(i-1) \ldots}_{W_{j}'} |\underbrace{\ldots (i+1) \ldots}_{W_{j+1}'} | \ldots )
\end{align*}

We consider two subcases.

\medskip
Case II(a): Suppose that $(i+1)$ is paired with some letter of $W_{j+1}$ in $\omega$ (and hence this letter is paired with $i$ in $\omega'$). Then, similar to the previous case, $(i-1)$ is also paired with some letter (either $i$ in $W_{j+1}$ or a letter to its right), so that neither $(i+1)$ nor $(i-1)$ is moved by $f_j$. Similarly, $(i+1)$ is paired with $(i-1)$ or some letter to the left of $(i+1)$ in $W_{j+1}'$ of $\omega'$, and again the result holds vacuously.

\medskip

Case II(b): If \emph{no} letter of $W_{j+1}$ is paired with $(i+1)$, then $i$ in $W_{j+1}$ \emph{will} pair with $(i-1)$, and hence the letter $(i+1)$ is moved by applying $f_j$:

 \begin{align*} \omega &=(\ldots| \underbrace{ \ldots {\color{teal}(i+1)} {\color {red}(i-1)} \ldots}_{W_j} |\underbrace{ \ldots {\color {red}i}\ldots}_{W_{j+1}} | \ldots ) \xrightarrow{f_j} f_j (\omega) =(\ldots |\underbrace{ \ldots {\color{red}(i-1)} \ldots}_{W_{j}} |\underbrace{\ldots {\color{teal}(i+1)}{\color{red}i} \ldots}_{W_{j+1}} | \ldots )
\end{align*}
Then, since $(i+1)$ in $W_j$ is unpaired with every letter preceding $i$ in $W_{j+1}$, the corresponding letter $i$ in $W_j'$ is likewise unpaired with every letter preceding $(i+1)$ in $W_{j+1}'$.  But, unlike the action of $f_j$ on $\omega$, the letter $i$ in $W_j'$ will be paired with $(i+1)$ in $W_{j+1}'$, and $(i-1)$ will move to $W_{j+1}'$. 

  \begin{align*} 
\omega' &=( \ldots| \underbrace{\ldots {\color{red}i}{\color{teal}(i-1)} \ldots}_{W_{j}'} |\underbrace{\ldots {\color{red}(i+1)} \ldots}_{W_{j+1}'} | \ldots )\xrightarrow{f_j} f_j(\omega')= ( \ldots |\underbrace{ \ldots {\color{red}i }\ldots}_{W_{j}'} |\underbrace{\ldots {\color{red}(i+1)}{\color{teal}(i-1)} \ldots}_{W_{j+1}'} | \ldots  )
\end{align*}

Thus, while the relative order of $(i-1), i$ and $(i+1)$ in $\omega$ and $\omega'$ was

\[ \omega = (\ldots i+1 \ldots i-1 \ldots  i \ldots ), \quad \omega' =(\ldots i \ldots i-1 \ldots  i+1 \ldots ) \]

upon applying $f_j$ to each their relative orders are now

\[ f_j (\omega) = (\ldots i-1 \ldots i+1 \ldots  i \ldots ),\quad  f_j(\omega') =(\ldots i \ldots i+1 \ldots  i-1 \ldots ). \]

Hence, $D_{i}^{+}\omega= \omega'$, but $D_{i}^{-}f_j(\omega )= f_j (\omega')$ and so $f_j(\omega)$ and $f_j(\omega')$ are elementarily dual equivalent.

\medskip
Case III. Finally, suppose $(i+1)$ appears in $W_j$ and is moved by applying $f_j$ to $\omega$, while $(i-1)$ is in $W_{j+1}$. Then, upon applying $f_j$, $(i+1)$ and $(i-1)$ will necessarily appear consecutively in $W_{j+1}$, retaining the same relative order in $f_j (\omega)$ as in $\omega$. Similarly, $i$ in $\omega'$ will be moved into $W_{j+1}'$, where it
appears immediately to the left of $(i-1)$, so that $f_j(\omega)$ and $f_j(\omega')$ remain elementarily dual equivalent:

 \begin{align*} \omega &=(\ldots| \underbrace{ \ldots (i+1)  \ldots}_{W_j} |\underbrace{ \ldots (i-1) \ldots}_{W_{j+1}} | \ldots i \ldots ) \xrightarrow{f_j} f_j(\omega) =(\ldots |\underbrace{  \ldots}_{W_{j}} |\underbrace{\ldots (i+1)(i-1) \ldots}_{W_{j+1}} | \ldots i \ldots )\\
  \omega' &=(\ldots| \underbrace{ \ldots i  \ldots}_{W_j'} |\underbrace{ \ldots (i-1) \ldots}_{W_{j+1}'} | \ldots (i+1) \ldots ) \xrightarrow{f_j} f_j(\omega') =(\ldots |\underbrace{  \ldots}_{W_{j}'} |\underbrace{\ldots i(i-1) \ldots}_{W_{j+1}'} | \ldots (i+1) \ldots )
\end{align*}

This exhausts all possible cases, and the proposition is proved.  
    \end{proof}

We now turn to our analysis of dual equivalence graphs of parsed maximal content words, beginning with a formal definition of their dual equivalence operators. 
\begin{definition}\label{Def:PCW_Dual_Equiv}
    Let $\PCW_{n,k}^{\max} \subset \PCW_{n,k}$ denote the set of  parsed maximal content words, and suppose $\omega_{\mathcal{P}}, \omega_{\mathcal{P}'}' \in \PCW_{n,k}^{\max} $. Then $\omega_{\mathcal{P}}$ and $\omega_{\mathcal{P}'}'$ are \emph{elementarily dual equivalent} if and only if their underlying words $\omega, \omega'$ (which are permutations) are elementarily dual equivalent, as in Definition~\ref{Def:Dual equiv}.
\end{definition}

The set of parsed maximal content words associated with row words in $[n]^{\otimes k}$ inherits a dual equivalence graph structure from that on permutations by the above definition. This allows us to consider dual equivalence graphs in several different combinatorial models. We distinguish the particular realizations of these graphs in $S_k$ and in collections of crystals and their abstract graph isomorphism types:

\begin{definition} 
\begin{enumerate}
\hskip .25in
\item A \emph{Knuth equivalence graph} is a graph whose nodes are a complete Knuth equivalence class of words $G \subseteq [n]^{\otimes k}$.  Two words $\rho, \rho' \in G$ are connected by an edge if and only if there is an \emph{elementary} Knuth equivalence between them. 

    \item $\DE(S_k)$ denotes the set of all dual equivalence graphs on $S_k$,
each of whose nodes are (one-rowed) permutations in $S_k$.
    \item $\DE(\PCW^{\text{max}}_{n,k})$ denotes the set of all graphs on parsed maximal content words associated to row words in $[n]^{\otimes k}$, in which two parsed maximal content words are connected by an edge if and only if they are related by the elementary dual equivalence operator $D_i^+$ or $D_i^-$. By Definition~\ref{Def:PCW_Dual_Equiv}, these are dual equivalence graphs on their underlying words (permutations). 
    \item $\DE_{\lambda}(\PCW_{n,k}^\mmax)$ denotes the subset of
$\DE(\PCW_{n,k}^\mmax)$ consisting of those dual equivalence graphs
whose nodes lie in irreducible crystals of highest weight $\lambda$.
    \item The abstract underlying graph of any element of $\DE_{\lambda}(\PCW_{n,k}^\mmax)$ will be called \emph{the} dual equivalence graph of shape $\lambda$, denoted $\mathcal{S}^\lambda$. 
\end{enumerate}
    
\end{definition}

Using Proposition~\ref{prop f_i to dual} and
Theorem~\ref{theorem biword dual Knuth}, we establish the following relations among the classes of dual equivalence graphs:

\begin{corollary} \label{cor DE types}
\hskip .2in
    \begin{enumerate}
        \item Define a map ${\mathcal D}: [n]^{\otimes k} \rightarrow S_k$ by mapping a word $\rho \in [n]^{\otimes k}$ to the underlying permutation ${\mathcal D}(\rho)=\widehat{\kappa}_\rho$ of the parsed maximal content word $\widehat{\kappa}_{|\rho|}$. Then ${\mathcal D}$ maps a Knuth equivalence graph $\Gamma$ in
$[n]^{\otimes k}$ onto a dual equivalence graph ${\mathcal D}(\Gamma)$ in $\DE(S_k)$ that is graph isomorphic to $\Gamma$.
    \item Let ${\mathcal C}_{\lambda}$ be the collection of all irreducible crystals in $[n]^{\otimes k}$ of highest weight $\lambda$. Then the Knuth equivalence graphs partition the nodes of the crystals in ${\mathcal C}_{\lambda}$ into disjoint isomorphic graphs, each of which is isomorphic to the dual equivalence graph $\mathcal{S}^\lambda$.

    \item The map
\[
{\mathcal D}:\DE(\PCW_{n,k}^{\max})\longrightarrow\DE(S_k),
\]
obtained by mapping each parsed maximal content word to its underlying
permutation, is a surjection. Define
\[
\DE_{\lambda}(S_k)
:=
{\mathcal D}\bigl(\DE_{\lambda}(\PCW_{n,k}^{\max})\bigr).
\]
Thus, the graphs in $\DE_{\lambda}(S_k)$ are precisely the graphs
obtained from those in $\DE_{\lambda}(\PCW_{n,k}^{\max})$ by removing
the parsings, and
\[
\DE(S_k)=\coprod_{\lambda}\DE_{\lambda}(S_k).
\]
    
    \item Applying the same non-$\NULL$ crystal operator to the nodes in a dual equivalence graph in
$\DE_{\lambda}(\PCW_{n,k}^{\max})$ (which, when viewed as nodes in isomorphic irreducible crystals, are plactically equivalent),  produces an isomorphic, yet distinct, dual
equivalence graph within $\DE_{\lambda}(\PCW_{n,k}^{\max})$.

\item Any two graphs $S,S'\in\DE_{\lambda}(S_k)$ can be related by first placing a common valid parsing on the nodes of $S$, viewing the resulting nodes as parsed maximal content words, and then applying a sequence of crystal operators simultaneously to those nodes to obtain $S'$.
    \end{enumerate}
\end{corollary}

\begin{proof}
For (1), Theorem~\ref{theorem biword dual Knuth} shows that an elementary Knuth equivalence between words $\rho$ and $\rho'$ corresponds precisely to an elementary dual equivalence between the underlying permutations ${\mathcal D}(\rho)=\widehat{\kappa}_\rho$ and ${\mathcal D}(\rho')=\widehat{\kappa}_{\rho'}$. Moreover, distinct words in a fixed Knuth equivalence graph give distinct underlying permutations of their associated parsed maximal content words. Hence ${\mathcal D}$ restricts to a graph isomorphism from each Knuth equivalence graph $\Gamma$ onto the dual equivalence graph ${\mathcal D}(\Gamma)$ in $\DE(S_k)$.

For (2), words occupying corresponding positions in isomorphic irreducible crystals are plactically equivalent, and hence lie in the same Knuth equivalence graph. Conversely, Knuth equivalent words are plactically equivalent and therefore occupy corresponding positions in isomorphic irreducible crystals. Thus the Knuth equivalence graphs partition the nodes of the crystals in ${\mathcal C}_\lambda$ transversely to the crystal components. By (1), each of these graphs is graph isomorphic to a dual equivalence graph on permutations, and since the crystal components have highest weight $\lambda$, each has abstract underlying graph $\mathcal S^\lambda$. 

For (3), every permutation admits its minimal parsing, so the map ${\mathcal D}$ that removes parsings is onto. By Proposition~\ref{Prop: dual equiv - valid parsings}, dual equivalent permutations have the same valid parsings, so removing or restoring a common parsing preserves the dual equivalence graph. In particular, the shape $\lambda$ associated to a graph is independent of the chosen valid parsing. It follows that the sets
\[
\DE_\lambda(S_k)
=
{\mathcal D}\bigl(\DE_\lambda(\PCW_{n,k}^{\max})\bigr)
\]
are well defined and partition $\DE(S_k)$, giving
\[
\DE(S_k)=\coprod_\lambda\DE_\lambda(S_k).
\]

For (4), suppose that two nodes of a graph in $\DE_\lambda(\PCW_{n,k}^{\max})$ are related by an elementary dual equivalence $D_i^+$ or $D_i^-$. Proposition~\ref{prop f_i to dual} shows that, after applying the same crystal operator to these nodes, provided it is not $\NULL$, their underlying permutations remain elementarily dual equivalent. The type of elementary dual equivalence may change from $D_i^+$ to $D_i^-$, or conversely, but the corresponding edge is preserved. Applying the crystal operator simultaneously to all plactically equivalent nodes therefore carries the entire dual equivalence graph to a graph-isomorphic dual equivalence graph. Since a (not $\NULL$) crystal operator moves these nodes to distinct positions in their respective crystal components, the resulting dual equivalence graph is distinct from the original one and remains in $\DE_\lambda(\PCW_{n,k}^{\max})$.

For (5), let $S,S'\in\DE_\lambda(S_k)$, and choose a common valid parsing of the nodes of $S$. By Proposition~\ref{Prop: dual equiv - valid parsings}, this parsing is valid on every node of $S$, so the resulting parsed maximal content words form a graph in $\DE_\lambda(\PCW_{n,k}^{\max})$. Since $S$ and $S'$ have the same shape $\lambda$, choose a node of $S$ and the corresponding node of $S'$ under the graph isomorphism between them. The parsed maximal content word associated to the first node and a parsed maximal content word having the second as its underlying permutation occupy positions in irreducible crystals of the same highest weight $\lambda$. Since these crystals are isomorphic and connected, a sequence of crystal operators carries the first position to the second. Applying this sequence simultaneously to all nodes of the lifted graph, and using (4) at each step, produces a sequence of graph-isomorphic dual equivalence graphs. The underlying permutations of the final graph are precisely the nodes of $S'$. Thus $S'$ is obtained from $S$ by first placing a common valid parsing on the nodes of $S$ and then applying a sequence of crystal operators simultaneously to those nodes.
\end{proof}

\begin{remark}
Although we have distinguished the elementary dual equivalence operators $D_i^+$ and $D_i^-$, these are typically identified as a single operator $D_i$. A crystal operator may interchange the $+$ and $-$ forms, but after this identification, the two structures commute in the sense that
\[
D_i f_j(\widehat{\kappa}_{|\rho|})
=
f_j(D_i\widehat{\kappa}_{|\rho|}),
\]
whenever the operators involved are defined, and similarly for $e_j$. Equivalently, on words in $[n]^{\otimes k}$, crystal operators preserve the Knuth equivalence relations corresponding to these elementary dual equivalences. We return to the significance of this compatibility in our discussion of Schur-Weyl duality.
\end{remark}

\bigskip
Consider our previous example of a dual equivalence graph of SYT of
shape $(3,1,1)$, but now realize the same dual equivalence graph on the
parsed reading words of those SYT:

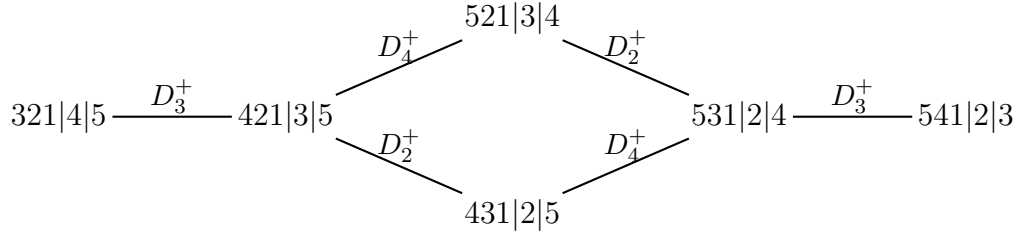
\begin{center}
\begin{tikzpicture}[  xscale=1.5,
    node/.style={inner sep=2pt},
    edge/.style={thick},
    edgelabel/.style={font=\small, inner sep=1pt}
]
\node[node] (A) at (0,0) {$321|4|5$};
\node[node] (B) at (2,0) {$421|3|5$};
\node[node] (C) at (4,1.3) {$521|3|4$};
\node[node] (D) at (4,-1.3) {$431|2|5$};
\node[node] (E) at (6,0) {$531|2|4$};
\node[node] (F) at (8,0) {$541|2|3$};

\draw[edge] (A) -- node[edgelabel,above] {$D_{3}^{+}$} (B);
\draw[edge] (B) -- node[edgelabel,above] {$D_{4}^{+}$}(C);
\draw[edge] (C) -- node[edgelabel,above] {$D_{2}^{+}$}(E);
\draw[edge] (E) -- node[edgelabel,above] {$D_{4}^{+}$}(D);
\draw[edge] (D) -- node[edgelabel,above] {$D_{2}^{+}$}(B);
\draw[edge] (E) -- node[edgelabel,above] {$D_{3}^{+}$}(F);

\end{tikzpicture}
\end{center}
When we view the reading words of the given SYT as parsed content words, we can apply $f_1$ to each and obtain a distinct dual equivalence graph,
in which two of the previous $D_{2}^{+}$ elementary dual equivalences have become $D_{2}^{-}$ (in bold):

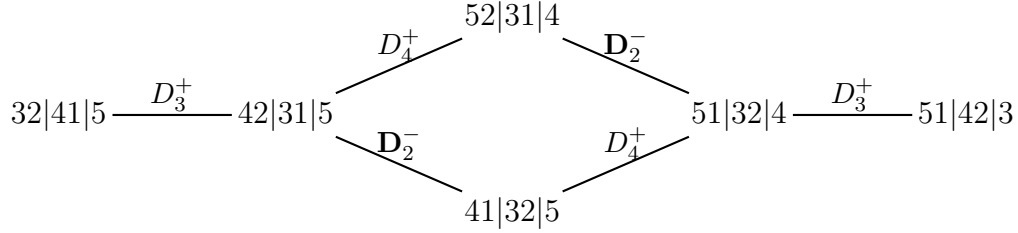
\begin{center}
\begin{tikzpicture}[  xscale=1.5,
    node/.style={inner sep=2pt},
    edge/.style={thick},
    edgelabel/.style={font=\small, inner sep=1pt}
]
\node[node] (A) at (0,0) {$32|41|5$};
\node[node] (B) at (2,0) {$42|31|5$};
\node[node] (C) at (4,1.3) {$52|31|4$};
\node[node] (D) at (4,-1.3) {$41|32|5$};
\node[node] (E) at (6,0) {$51|32|4$};
\node[node] (F) at (8,0) {$51|42|3$};

\draw[edge] (A) -- node[edgelabel,above] {$D_{3}^{+}$} (B);
\draw[edge] (B) -- node[edgelabel,above] {$D_{4}^{+}$}(C);
\draw[edge] (C) -- node[edgelabel,above] {$\mathbf D_{2}^{-}$}(E);
\draw[edge] (E) -- node[edgelabel,above] {$ D_{4}^{+}$}(D);
\draw[edge] (D) -- node[edgelabel,above] {$\mathbf D_{2}^{-}$}(B);
\draw[edge] (E) -- node[edgelabel,above] {$D_{3}^{+}$}(F);

\end{tikzpicture}
\end{center}

In other words, while we can use dual equivalence relations on permutations to define edges in dual equivalence graphs, we can also apply a common parsing to each node (by virtue of Proposition~\ref{Prop: dual equiv - valid parsings}) and view permutations as parsed maximal content words. We can then use crystal operators to define maps between isomorphic, yet distinct, dual equivalence graphs. This gives an intrinsic way to associate a shape $\lambda$ to every dual equivalence graph in $S_k$: choose any
permutation in the graph and any valid parsing of that permutation, regard the resulting parsed maximal content word as a node in a crystal, and let $\lambda$ be the highest weight of that crystal. Equivalently, $\lambda$ is the shape for which the graph is graph isomorphic to the dual equivalence graph on the reading words of SYT
of shape $\lambda$. Thus
\[
\DE(S_k)=\coprod_{\lambda}\DE_{\lambda}(S_k),
\]
where $\DE_{\lambda}(S_k)$ consists precisely of the dual equivalence graphs of shape $\lambda$. Roberts (see~\cite{R13}) identified isomorphism classes of dual equivalence graphs by analyzing the local definitions of such graphs given by Assaf~\cite{Assaf15}. Here, we identify the shape of each dual equivalence graph and define explicit isomorphisms between dual equivalence graphs of the same shape intrinsically, through their connection to crystal structures.

The intrinsic identification of the shapes of dual equivalence graphs through crystal structure, together with the explicit crystal-theoretic maps between distinct dual equivalence graphs of the same shape, appears to be new.

The parsed content word model allows us to make sense of applying dual equivalence operators to the parsed maximal content word
associated to any node within a crystal. Furthermore, the parsed content words model makes it possible to interpret dual equivalence operators as moving between different irreducible crystals. Previously, the dual equivalence operators used by Assaf \cite{Assaf15} were applied to reading words of SYT. Assaf was able to find explicit sequences of crystal operators connecting the distinguished standard nodes of different quasicrystals within a single irreducible crystal; see
\cite[Theorem~4.2]{Assaf08fpsac-crystalops}.

We state the following straightforward observation to underscore its connection to the representation theory of $S_k$ and (in the next section) its role in Schur-Weyl duality:

\begin{corollary}
The order of the dual equivalence graph $\mathcal{S}^{\lambda}$ is equal to the number of SYT of shape $\lambda$, and hence equal to the dimension of $S^{\lambda}$, the irreducible $S_k$ representation of shape $\lambda$.
\end{corollary}

\subsection{Schur-Weyl Duality}

Our combinatorial view of Schur-Weyl duality involves decomposing $[n]^{\otimes k}$, for each partition $\lambda$, into two transverse collections of connected graphs: irreducible crystal graphs and dual equivalence graphs. This is most easily accomplished by representing words $\rho\in[n]^{\otimes k}$ by parsed maximal content words. Recall that we embed a word $\rho$ in a parsed maximal biword and then compute the row word of its dual:
\[
\rho \longrightarrow {\kappa \choose \rho}
\longrightarrow {\rho^{*} \choose \widehat{\kappa}_{\rho}}
\longrightarrow \widehat{\kappa}_{|\rho|}, \quad \hbox{where $\kappa=1234\cdots k$.}
\]

\begin{definition}
Given a partition $\lambda$, let
$\PCW_{n,k}^{\text{max},\lambda}\subset
\PCW_{n,k}^{\max}$ denote the set of parsed maximal content words $\widehat{\kappa}_{|\rho|}$ associated to words $\rho\in[n]^{\otimes k}$ lying in irreducible crystal graphs of highest weight $\lambda$.
\end{definition}
The set $\PCW_{n,k}^{\text{max},\lambda}$ can be partitioned into a finite collection of disjoint, isomorphic irreducible crystal graphs of highest weight $\lambda$. By Corollary~\ref{cor DE types}, its nodes can also be partitioned into disjoint dual equivalence graphs, each
graph isomorphic to $\mathcal S^\lambda$. Thus
$\PCW_{n,k}^{\text{max},\lambda}$ carries two transverse graph structures, both of which will be used in the combinatorial Schur-Weyl decomposition.

The idea is portrayed in the diagrams below. Suppose the vertical graphs denote isomorphic irreducible crystals, connecting nodes (parsed maximal content words) $A_i$ to $B_i$ to $C_i$, etc., while the light edges denote edges of dual equivalence graphs. The correspondence between nodes determined by these horizontal dual equivalence graphs gives crystal isomorphisms between the vertical crystals.

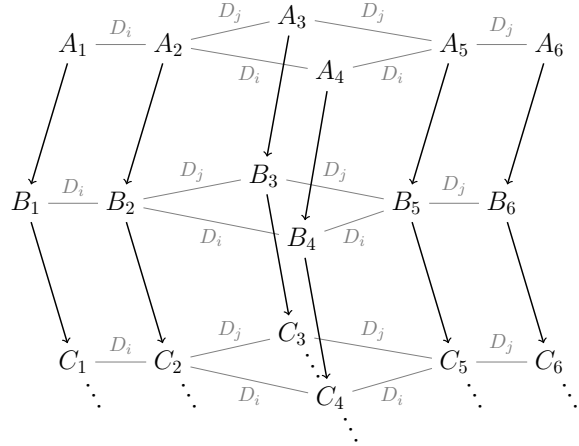
\begin{center}
       \scalebox{0.7}{\begin{tikzpicture}[
    x=1.8cm,
    y=1.5cm,
    every node/.style={font=\large}
]

\node (A1) at (0,2) {$A_1$};
\node (A2) at (1,2) {$A_2$};

\node (A3) at (2.3,2.35) {$A_3$};
\node (A4) at (2.7,1.65) {$A_4$};

\node (A5) at (4,2) {$A_5$};
\node (A6) at (5,2) {$A_6$};

\node (B1) at (-0.5,0) {$B_1$};
\node (B2) at (0.5,0) {$B_2$};

\node (B3) at (2.0,0.35) {$B_3$};
\node (B4) at (2.4,-0.45) {$B_4$};

\node (B5) at (3.5,0) {$B_5$};
\node (B6) at (4.5,0) {$B_6$};

\node (C1) at (0,-2) {$C_1$};
\node (C2) at (1,-2) {$C_2$};

\node (C3) at (2.3,-1.65) {$C_3$};
\node (C4) at (2.7,-2.45) {$C_4$};

\node (C5) at (4,-2) {$C_5$};
\node (C6) at (5,-2) {$C_6$};


\draw[gray] (A1) -- node[above,font=\small] {$D_i$} (A2);

\draw[gray] (A2) -- node[above,font=\small] {$D_j$} (A3);
\draw[gray] (A2) -- node[below,font=\small] {$D_i$} (A4);

\draw[gray] (A3) -- node[above,font=\small] {$D_j$} (A5);
\draw[gray] (A4) -- node[below,font=\small] {$D_i$} (A5);

\draw[gray] (A5) -- node[above,font=\small] {$D_j$} (A6);


\draw[gray] (B1) -- node[above,font=\small] {$D_i$} (B2);

\draw[gray] (B2) -- node[above,font=\small] {$D_j$} (B3);
\draw[gray] (B2) -- node[below,font=\small] {$D_i$} (B4);

\draw[gray] (B3) -- node[above,font=\small] {$D_j$} (B5);
\draw[gray] (B4) -- node[below,font=\small] {$D_i$} (B5);

\draw[gray] (B5) -- node[above,font=\small] {$D_j$} (B6);


\draw[gray] (C1) -- node[above,font=\small] {$D_i$} (C2);

\draw[gray] (C2) -- node[above,font=\small] {$D_j$} (C3);
\draw[gray] (C2) -- node[below,font=\small] {$D_i$} (C4);

\draw[gray] (C3) -- node[above,font=\small] {$D_j$} (C5);
\draw[gray] (C4) -- node[below,font=\small] {$D_i$} (C5);

\draw[gray] (C5) -- node[above,font=\small] {$D_j$} (C6);


\draw[thick][->] (A1) -- (B1);
\draw[thick][->] (A2) -- (B2);
\draw[thick][->] (A3) -- (B3);
\draw[thick][->] (A4) -- (B4);
\draw[thick][->] (A5) -- (B5);
\draw[thick][->] (A6) -- (B6);

\draw[thick][->] (B1) -- (C1);
\draw[thick][->] (B2) -- (C2);
\draw[thick][->] (B3) -- (C3);
\draw[thick][->] (B4) -- (C4);
\draw[thick][->] (B5) -- (C5);
\draw[thick][->] (B6) -- (C6);

\node[rotate=-63] at (0.22,-2.45) {$\cdots$};
\node[rotate=-63] at (1.22,-2.45) {$\cdots$};

\node[rotate=-63] at (2.52,-2.08) {$\cdots$};
\node[rotate=-63] at (2.92,-2.88) {$\cdots$};

\node[rotate=-63] at (4.22,-2.45) {$\cdots$};
\node[rotate=-63] at (5.22,-2.45) {$\cdots$};

\end{tikzpicture}}
\end{center}

However, from another perspective, we can regard the horizontal graphs as distinct, isomorphic dual equivalence graphs (the dark edges below), consisting successively of all the $A_i$'s, all the $B_i$'s, etc. Crystal operators, applied simultaneously to corresponding nodes, carry one such dual equivalence graph to another, as in Corollary~\ref{cor DE types} (4).

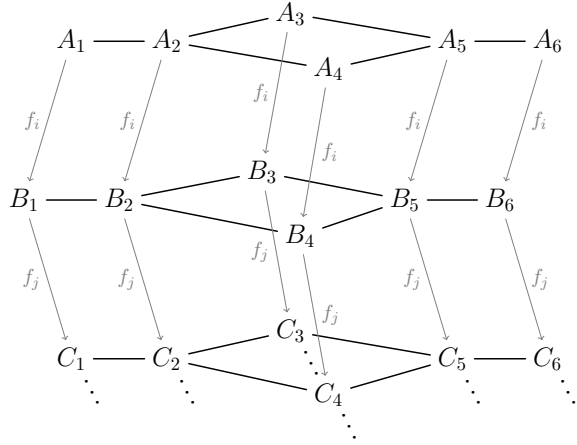
\begin{center}
       \scalebox{0.7}{\begin{tikzpicture}[
    x=1.8cm,line cap=round,
line join=round,
    y=1.5cm,
    every node/.style={font=\large}
]
line cap=round,
line join=round
\node (A1) at (0,2) {$A_1$};
\node (A2) at (1,2) {$A_2$};

\node (A3) at (2.3,2.35) {$A_3$};
\node (A4) at (2.7,1.65) {$A_4$};

\node (A5) at (4,2) {$A_5$};
\node (A6) at (5,2) {$A_6$};

\node (B1) at (-0.5,0) {$B_1$};
\node (B2) at (0.5,0) {$B_2$};

\node (B3) at (2.0,0.35) {$B_3$};
\node (B4) at (2.4,-0.45) {$B_4$};

\node (B5) at (3.5,0) {$B_5$};
\node (B6) at (4.5,0) {$B_6$};

\node (C1) at (0,-2) {$C_1$};
\node (C2) at (1,-2) {$C_2$};

\node (C3) at (2.3,-1.65) {$C_3$};
\node (C4) at (2.7,-2.45) {$C_4$};

\node (C5) at (4,-2) {$C_5$};
\node (C6) at (5,-2) {$C_6$};


\draw[thick] (A1) -- (A2);

\draw[thick] (A2) -- (A3);
\draw[thick] (A2) -- (A4);

\draw[thick] (A3) -- (A5);
\draw[thick] (A4) -- (A5);

\draw[thick] (A5) -- (A6);


\draw[thick] (B1) -- (B2);

\draw[thick] (B2) -- (B3);
\draw[thick] (B2) -- (B4);

\draw[thick] (B3) -- (B5);
\draw[thick] (B4) -- (B5);

\draw[thick] (B5) -- (B6);


\draw[thick] (C1) -- (C2);

\draw[thick] (C2) -- (C3);
\draw[thick] (C2) -- (C4);

\draw[thick] (C3) -- (C5);
\draw[thick] (C4) -- (C5);

\draw[thick] (C5) -- (C6);

node[left,font=\small]
\draw[gray][->] (A1) -- node[left,font=\small] {$f_i$} (B1);
\draw (A2)[gray][->] -- node[left,font=\small] {$f_i$} (B2);
\draw (A3)[gray][->] -- node[left,font=\small] {$f_i$} (B3);
\draw (A4)[gray][->] -- node[right,font=\small] {$f_i$} (B4);
\draw (A5)[gray][->] -- node[left,font=\small] {$f_i$} (B5);
\draw (A6)[gray][->] -- node[right,font=\small] {$f_i$} (B6);

\draw (B1)[gray][->] -- node[left,font=\small] {$f_j$} (C1);
\draw (B2)[gray][->] -- node[left,font=\small] {$f_j$} (C2);
\draw (B3)[gray][->] -- node[left,font=\small] {$f_j$} (C3);
\draw (B4)[gray][->] -- node[right,font=\small] {$f_j$} (C4);
\draw (B5)[gray][->] -- node[left,font=\small] {$f_j$} (C5);
\draw (B6)[gray][->] -- node[right,font=\small] {$f_j$} (C6);



\node[rotate=-63] at (0.22,-2.45) {$\cdots$};
\node[rotate=-63] at (1.22,-2.45) {$\cdots$};

\node[rotate=-63] at (2.52,-2.08) {$\cdots$};
\node[rotate=-63] at (2.92,-2.88) {$\cdots$};

\node[rotate=-63] at (4.22,-2.45) {$\cdots$};
\node[rotate=-63] at (5.22,-2.45) {$\cdots$};

\end{tikzpicture}}
\end{center}

If we consider these two graph structures simultaneously, the diagrams suggest that $\PCW_{n,k}^{\text{max},\lambda}$ has the structure of a Cartesian product of an irreducible crystal graph of highest weight $\lambda$ with a dual equivalence graph of shape $\lambda$. We now make this precise.

\begin{theorem}\label{Thm:S-W Duality} Let ${\mathcal G}_{\lambda}$ be the graph whose nodes are the parsed maximal content words in $\PCW_{n,k}^{\text{max},\lambda}$, with two nodes $\omega$ and $\omega'$ connected if either they are connected by a single crystal operator $f_i$ or $e_i$, or they are elementarily dual equivalent. Thus the first type of edge lies within an irreducible crystal graph, while the second connects plactically equivalent nodes in isomorphic irreducible crystals. Then ${\mathcal G}_{\lambda}$ is graph isomorphic to the Cartesian product of $\mathcal B_{\lambda}$, the abstract irreducible crystal graph of highest weight $\lambda$, and $\mathcal S^{\lambda}$, the dual equivalence graph of shape $\lambda$. \end{theorem}

\begin{proof} Choose a dual equivalence graph $S$ in $\DE_{\lambda}(\PCW_{n,k}^{\max})$, so that $S$ is graph isomorphic to $\mathcal S^\lambda$, and choose one of the irreducible crystal components $B$ that shares at least one node with $S$, so that $B$ is graph isomorphic to $\mathcal B_\lambda$. By Corollary~\ref{cor DE types}, the nodes of $S$ occupy corresponding, plactically equivalent positions in distinct isomorphic crystal components. Moreover, applying a crystal operator simultaneously to the nodes of $S$, when it is not $\NULL$, carries $S$ to a distinct isomorphic dual equivalence graph, and every dual equivalence graph in $\DE_{\lambda}(\PCW_{n,k}^{\max})$ is obtained in this way. Hence every node of $\PCW_{n,k}^{\text{max},\lambda}$ is uniquely determined by a node of $B$, specifying its position within a crystal component, together with a node of $S$, specifying the crystal component containing it. Under this bijection, a crystal edge changes only the $B$-coordinate, while a dual equivalence edge changes only the $S$-coordinate. Thus ${\mathcal G}_{\lambda}$ is graph isomorphic to the Cartesian product $B\times S$, and hence to $\mathcal B_\lambda\times\mathcal S^\lambda$. \end{proof}

Having established the Cartesian product structure on parsed maximal content words, we now transport this structure to $[n]^{\otimes k}$, giving a combinatorial realization of Schur-Weyl duality directly on words.

\begin{theorem}[Combinatorial Schur-Weyl Duality] Give $[n]^{\otimes k}$ a graph structure by connecting two words $\rho,\rho'$ by an edge if either one is obtained from the other by a single crystal operator, or they are elementarily Knuth equivalent. Then this graph decomposes as a disjoint union, indexed by partitions $\lambda$, in which each component ${\mathcal G}_{\lambda}$ is graph isomorphic to the Cartesian product of the abstract irreducible crystal graph $\mathcal B_{\lambda}$ of highest weight $\lambda$ and the abstract dual equivalence graph $\mathcal S^{\lambda}$ of shape $\lambda$. Thus \[ [n]^{\otimes k}\cong\bigsqcup_{\lambda}\bigl(\mathcal B_{\lambda}\times\mathcal S^{\lambda}\bigr). \] \end{theorem}

\begin{proof} The map $\Phi:[n]^{\otimes k}\longrightarrow \PCW_{n,k}^{\max}$, defined by $\Phi(\rho)=\widehat{\kappa}_{|\rho|}$, is a bijection. By Theorem~\ref{thm:f hat crystal op}, it preserves crystal operators, so two words are joined by a crystal edge if and only if their corresponding parsed maximal content words are joined by the same crystal edge. By Theorem~\ref{theorem biword dual Knuth}, two words are elementarily Knuth equivalent if and only if their corresponding parsed maximal content words are elementarily dual equivalent. Hence $\Phi$ is a graph isomorphism from the graph structure defined above on $[n]^{\otimes k}$ to the graph on $\PCW_{n,k}^{\max}$ whose edges are crystal edges and elementary dual equivalence edges. The latter decomposes as the disjoint union of the graphs ${\mathcal G}_{\lambda}$, and by Theorem~\ref{Thm:S-W Duality}, ${\mathcal G}_{\lambda}\cong\mathcal B_{\lambda}\times\mathcal S^{\lambda}$. Therefore \[ [n]^{\otimes k}\cong\bigsqcup_{\lambda}\left(\mathcal B_{\lambda}\times\mathcal S^{\lambda}\right), \] as claimed. \end{proof}
\begin{remark}
The preceding theorem also explains in what sense the two graph structures can be regarded as crystallized versions of the two commuting actions in classical Schur-Weyl duality. Lie algebra operators commute with the action of $S_k$ by permutation of tensor factors in $V^{\otimes k}$. Crystal operators $e_j,f_j$ on $[n]^{\otimes k}$ provide the familiar crystallization of the $\mathfrak{gl}_n$-action: the linear action on $V^{\otimes k}$ is replaced by partial operations on individual words. In a similar, though less literal, sense, elementary Knuth operators provide a crystallized proxy for the $S_k$-action. The action of $K_i^+$, when not $\NULL$, permutes the $i$-th and $(i+1)$-st letters of a word, while $K_i^-$ permutes the $(i-1)$-st and $i$-th letters. Thus these operators retain, in partial combinatorial form, the permutation of tensor factors underlying the $S_k$-action.

Our results identify this structure with the second factor in the combinatorial Schur-Weyl decomposition. Under the bijection taking a word $\rho$ to its associated parsed maximal content word $\widehat{\kappa}_{|\rho|}$, elementary Knuth equivalence becomes elementary dual equivalence of the underlying permutations. Consequently, the connected graphs generated by these crystallized permutation operations are precisely the dual equivalence graphs $\mathcal S^\lambda$ occurring in
\[
[n]^{\otimes k}\cong\bigsqcup_\lambda
\left(\mathcal B_\lambda\times\mathcal S^\lambda\right).
\]

 Thus plactic equivalence, Knuth equivalence, and dual equivalence appear not merely as parallel equivalence relations, but as different realizations of the $S_k$ direction (depicted as the horizontal graphs in the diagram above) that is compatible with the crystal structure in our combinatorial model of Schur-Weyl duality.
\end{remark}

\medskip
    \renewcommand{\arraystretch}{1} 
\subsection{Combinatorial Schur-Weyl Duality and RSK}

Our combinatorial view of Schur-Weyl duality can also be interpreted naturally in terms of the RSK correspondence. The classical RSK algorithm on biwords gives a map
\[
\rho\longmapsto RSK\left({\kappa\choose\rho}\right)=(P,Q),
\]
where $P$ and $Q$ are tableaux of the same shape $\lambda$. Since throughout this section we assume $\kappa=12\cdots k$, the biword is maximal, and $Q$ is necessarily a SYT. Thus, in the present setting, RSK gives explicit coordinates for a word $\rho$: $P$ specifies a node in an irreducible crystal graph of highest weight $\lambda$, while $Q$ specifies a node in the dual equivalence graph of SYT of shape $\lambda$. In this sense, the pair $(P,Q)$ gives coordinates for $\rho$ in the Cartesian product $\mathcal B_\lambda\times\mathcal S^\lambda$. The RSK correspondence appears in diverse combinatorial settings (see \cite{Alexandersson_SymCatalog} for an overview). As shown in \cite[Chapter 8]{BS}, if
\[
RSK\left({\kappa\choose\rho}\right)=(P,Q)
\quad\text{and}\quad
RSK\left({\kappa\choose\rho'}\right)=(P',Q'),
\]
then $P=P'$ implies that the two biwords are plactically equivalent and hence occupy corresponding positions in isomorphic irreducible crystal graphs of highest weight $\lambda$, while $Q=Q'$ implies that the two biwords lie in the same irreducible crystal.

Using the bijection from biwords to ptableaux, these classical results of RSK on biwords can be interpreted naturally in terms of ptableaux. In~\cite{AW2023}, the RSK algorithm is defined as a map $T\mapsto RSK(T)=(PT,T_{\mmax})$, where $PT$ and $T_{\mmax}$ are perforated tableaux:

\begin{theorem}
The algorithm {\emph{RSK}}, defined in \cite{AW2023}, maps a ptableau $T$ to a pair of ptableaux $RSK(T)=(PT,T_{\mmax})$ such that
\begin{enumerate}
\item $PT$ is an element in the distinguished crystal.
\item $PT$ is plactically equivalent to $T$.
\item $T_{\mmax}$ is the highest weight element in the irreducible crystal graph containing $T$.
\end{enumerate}
\label{RSK algorithm}
\end{theorem}

For the maximal biwords considered here, the SYT $Q$ of the classical RSK algorithm corresponds to $T_{\mmax}$, now viewed as the highest weight element of the irreducible crystal containing $T$. In addition, the SSYT $P$ and the ptableau $PT$, as well as their respective biwords, are dual. Consequently, if $T$ is the ptableau associated to the maximal biword of $\rho$, then
\[
RSK(T)=(PT,T_{\mmax})
\]
again gives coordinates for $\rho$ in $\mathcal B_\lambda\times\mathcal S^\lambda$: $PT$ specifies a position in the distinguished crystal of highest weight $\lambda$, while the standard tableau $T_{\mmax}$ specifies a position in the dual equivalence graph $\mathcal S^\lambda$.

The RSK correspondence therefore gives another way to construct the Cartesian graph $\mathcal B_\lambda\times\mathcal S^\lambda$: after identifying a word with its pair $(P,Q)$, we can use crystal operators in the $P$-coordinate and dual equivalence operators in the $Q$-coordinate to define its two families of edges. What connects this Cartesian graph specifically to Schur-Weyl duality, however, is that both graph structures can be realized directly on the words in $[n]^{\otimes k}$. Crystal operators give a combinatorial proxy for the $\mathfrak{gl}_n$-action, while the results above show that the dual equivalence direction can be realized by elementary Knuth equivalences on words, or equivalently by elementary dual equivalences on their parsed maximal content words. Moreover, these two structures are compatible: applying a crystal operator to the corresponding nodes of a dual equivalence graph produces another dual equivalence graph, although an edge labeled $D_i^+$ may become one labeled $D_i^-$, or conversely. Thus RSK provides explicit coordinates for the Cartesian structure, while the preceding construction identifies the two directions of that structure intrinsically on the words themselves.

\subsection{Combinatorial Schur-Weyl Duality and Quasicrystals}

We now position our results on combinatorial Schur-Weyl duality within earlier work on quasicrystals and our results in Section~\ref{section: Parsed word quasicrystal}. Assaf \cite{Assaf08fpsac-crystalops} explored aspects of Schur-Weyl duality by means of dual equivalence graphs, studying a single irreducible crystal of SSYT of fixed shape $\lambda$.  Such crystals contain some nodes that are SYT of shape $\lambda$, and hence are naturally the nodes of the dual equivalence graph $\mathcal S^\lambda$. Assaf showed that the elementary dual equivalence operators on these standard nodes can be realized by explicit sequences of crystal operators within this single irreducible crystal. Thus the dual equivalence graph can, in this setting, be recovered from relations among distinguished nodes internal to one crystal. It was not clear how this related specifically to Schur-Weyl duality, since the construction was contained within a single crystal, but it did establish a bridge between crystals and dual equivalence.

A related perspective was developed by Maas-Gariépy \cite{FMG23}, who decomposed a crystal of SSYT into connected induced subgraphs called \emph{quasicrystals}. There, it was determined that each quasicrystal contains a unique standard node of the sort studied by Assaf. Maas-Gariépy defined the \emph{crystal skeleton} as the graph induced by treating each quasicrystal as a single node. She conjectured that the dual equivalence graph of shape $\lambda$ is a subgraph of the crystal skeleton. More recently, Brauner, Corteel, Daugherty, and Schilling \cite{BCDS} proved this conjecture as part of a more general analysis showing that crystal skeletons generalize dual equivalence graphs.

The parsed maximal content word model and the view of combinatorial Schur-Weyl duality presented here allow a broader view. Our analysis allows us to model \emph{all} quasicrystals within all irreducible crystals in $[n]^{\otimes k}$, and not only the distinguished crystal of words that are reading words of SSYT. Recall (Theorem~\ref{thm:quasi constant content}) that within a fixed irreducible crystal, the parsed maximal content words belonging to a given quasicrystal are precisely the different valid parsings of a single underlying permutation. We can therefore ``collapse'' each quasicrystal to this unique permutation and regard that permutation as an index for the quasicrystal. In this way, the quasicrystals throughout $[n]^{\otimes k}$, rather than only those within a single crystal of SSYT, can be represented by permutations (and all permutations represent quasicrystals).

The permutation indices associated to quasicrystals retain information about both the \emph{transverse} structure connecting plactically equivalent quasicrystals and the \emph{vertical} relations among quasicrystals within a given irreducible crystal. In the transverse direction, dual equivalent permutations index plactically equivalent quasicrystals, and hence quasicrystals occupying corresponding positions in isomorphic irreducible crystals. Thus the permutation indices of a plactic equivalence class of quasicrystals form a dual equivalence graph. Within a given irreducible crystal, on the other hand, the parsed forms of the permutations indexing its quasicrystals are related through crystal operators. These crystal operators therefore map the transverse dual equivalence graphs to one another, identifying graph proxies for distinct isomorphic $S_k$ representations. After collapsing each quasicrystal to its permutation index, crystal relations consequently retain the organization of quasicrystals within irreducible crystals, while dual equivalence retains the relations among corresponding quasicrystals across isomorphic irreducible crystals.

The following examples illustrate this collapse and its interaction with the two graph structures. In particular, they provide a way to compare the dual equivalence structure obtained from the permutation indices with the quasicrystal structure from which those indices arise. 
\medskip

 \begin{example}
Let $\lambda=(2,1,1)$. There exist three irreducible crystals of words of highest weight $(2,1,1)$ in $[4]^{\otimes 4}$. We begin with these crystals before passing to their quasicrystal indices given by the parsed content word model. The diagram below shows the three isomorphic crystals, together with the explicit elementary Knuth equivalences connecting their highest weight elements. Similar Knuth equivalences connect the plactically equivalent words occupying corresponding positions throughout the three crystals:

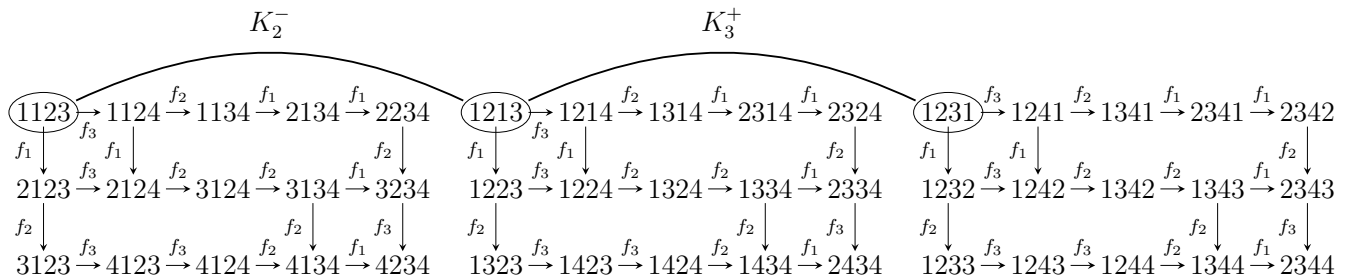
\begin{center}
\scalebox{0.88}{
\begin{tikzpicture}[
    node/.style={inner sep=2pt},
    >=stealth
]

\def\xsep{1.35}   
\def\ysep{1.15}   
\def\gsep{6.8}    

\def\Ashift{0}
\def\Bshift{\gsep}
\def\Cshift{2*\gsep}


\node[node] (A11) at ({\Ashift+1*\xsep},{-1*\ysep}) {$1123$};
\node[node] (A12) at ({\Ashift+2*\xsep},{-1*\ysep}) {$1124$};
\node[node] (A13) at ({\Ashift+3*\xsep},{-1*\ysep}) {$1134$};
\node[node] (A14) at ({\Ashift+4*\xsep},{-1*\ysep}) {$2134$};
\node[node] (A15) at ({\Ashift+5*\xsep},{-1*\ysep}) {$2234$};

\node[node] (A21) at ({\Ashift+1*\xsep},{-2*\ysep}) {$2123$};
\node[node] (A22) at ({\Ashift+2*\xsep},{-2*\ysep}) {$2124$};
\node[node] (A23) at ({\Ashift+3*\xsep},{-2*\ysep}) {$3124$};
\node[node] (A24) at ({\Ashift+4*\xsep},{-2*\ysep}) {$3134$};
\node[node] (A25) at ({\Ashift+5*\xsep},{-2*\ysep}) {$3234$};

\node[node] (A31) at ({\Ashift+1*\xsep},{-3*\ysep}) {$3123$};
\node[node] (A32) at ({\Ashift+2*\xsep},{-3*\ysep}) {$4123$};
\node[node] (A33) at ({\Ashift+3*\xsep},{-3*\ysep}) {$4124$};
\node[node] (A34) at ({\Ashift+4*\xsep},{-3*\ysep}) {$4134$};
\node[node] (A35) at ({\Ashift+5*\xsep},{-3*\ysep}) {$4234$};



\node[node] (B11) at ({\Bshift+1*\xsep},{-1*\ysep}) {$1213$};
\node[node] (B12) at ({\Bshift+2*\xsep},{-1*\ysep}) {$1214$};
\node[node] (B13) at ({\Bshift+3*\xsep},{-1*\ysep}) {$1314$};
\node[node] (B14) at ({\Bshift+4*\xsep},{-1*\ysep}) {$2314$};
\node[node] (B15) at ({\Bshift+5*\xsep},{-1*\ysep}) {$2324$};

\node[node] (B21) at ({\Bshift+1*\xsep},{-2*\ysep}) {$1223$};
\node[node] (B22) at ({\Bshift+2*\xsep},{-2*\ysep}) {$1224$};
\node[node] (B23) at ({\Bshift+3*\xsep},{-2*\ysep}) {$1324$};
\node[node] (B24) at ({\Bshift+4*\xsep},{-2*\ysep}) {$1334$};
\node[node] (B25) at ({\Bshift+5*\xsep},{-2*\ysep}) {$2334$};

\node[node] (B31) at ({\Bshift+1*\xsep},{-3*\ysep}) {$1323$};
\node[node] (B32) at ({\Bshift+2*\xsep},{-3*\ysep}) {$1423$};
\node[node] (B33) at ({\Bshift+3*\xsep},{-3*\ysep}) {$1424$};
\node[node] (B34) at ({\Bshift+4*\xsep},{-3*\ysep}) {$1434$};
\node[node] (B35) at ({\Bshift+5*\xsep},{-3*\ysep}) {$2434$};



\node[node] (C11) at ({\Cshift+1*\xsep},{-1*\ysep}) {$1231$};
\node[node] (C12) at ({\Cshift+2*\xsep},{-1*\ysep}) {$1241$};
\node[node] (C13) at ({\Cshift+3*\xsep},{-1*\ysep}) {$1341$};
\node[node] (C14) at ({\Cshift+4*\xsep},{-1*\ysep}) {$2341$};
\node[node] (C15) at ({\Cshift+5*\xsep},{-1*\ysep}) {$2342$};

\node[node] (C21) at ({\Cshift+1*\xsep},{-2*\ysep}) {$1232$};
\node[node] (C22) at ({\Cshift+2*\xsep},{-2*\ysep}) {$1242$};
\node[node] (C23) at ({\Cshift+3*\xsep},{-2*\ysep}) {$1342$};
\node[node] (C24) at ({\Cshift+4*\xsep},{-2*\ysep}) {$1343$};
\node[node] (C25) at ({\Cshift+5*\xsep},{-2*\ysep}) {$2343$};

\node[node] (C31) at ({\Cshift+1*\xsep},{-3*\ysep}) {$1233$};
\node[node] (C32) at ({\Cshift+2*\xsep},{-3*\ysep}) {$1243$};
\node[node] (C33) at ({\Cshift+3*\xsep},{-3*\ysep}) {$1244$};
\node[node] (C34) at ({\Cshift+4*\xsep},{-3*\ysep}) {$1344$};
\node[node] (C35) at ({\Cshift+5*\xsep},{-3*\ysep}) {$2344$};


\draw[->] (A11) -- node[below] {$\scriptstyle f_3$} (A12);
\draw[->] (A12) -- node[above] {$\scriptstyle f_2$} (A13);
\draw[->] (A13) -- node[above] {$\scriptstyle f_1$} (A14);
\draw[->] (A14) -- node[above] {$\scriptstyle f_1$} (A15);

\draw[->] (A21) -- node[above] {$\scriptstyle f_3$} (A22);
\draw[->] (A22) -- node[above] {$\scriptstyle f_2$} (A23);
\draw[->] (A23) -- node[above] {$\scriptstyle f_2$} (A24);
\draw[->] (A24) -- node[above] {$\scriptstyle f_1$} (A25);

\draw[->] (A31) -- node[above] {$\scriptstyle f_3$} (A32);
\draw[->] (A32) -- node[above] {$\scriptstyle f_3$} (A33);
\draw[->] (A33) -- node[above] {$\scriptstyle f_2$} (A34);
\draw[->] (A34) -- node[above] {$\scriptstyle f_1$} (A35);

\draw[->] (A11) -- node[left] {$\scriptstyle f_1$} (A21);
\draw[->] (A21) -- node[left] {$\scriptstyle f_2$}(A31);
\draw[->] (A15) -- node[left] {$\scriptstyle f_2$}(A25);
\draw[->] (A25) -- node[left] {$\scriptstyle f_3$}(A35);
\draw[->] (A12) -- node[left] {$\scriptstyle f_1$}(A22);
\draw[->] (A24) -- node[left] {$\scriptstyle f_2$}(A34);


\draw[->] (B11) -- node[below] {$\scriptstyle f_3$} (B12);
\draw[->] (B12) -- node[above] {$\scriptstyle f_2$} (B13);
\draw[->] (B13) -- node[above] {$\scriptstyle f_1$} (B14);
\draw[->] (B14) -- node[above] {$\scriptstyle f_1$} (B15);

\draw[->] (B21) -- node[above] {$\scriptstyle f_3$} (B22);
\draw[->] (B22) -- node[above] {$\scriptstyle f_2$} (B23);
\draw[->] (B23) -- node[above] {$\scriptstyle f_2$} (B24);
\draw[->] (B24) -- node[above] {$\scriptstyle f_1$} (B25);

\draw[->] (B31) -- node[above] {$\scriptstyle f_3$} (B32);
\draw[->] (B32) -- node[above] {$\scriptstyle f_3$} (B33);
\draw[->] (B33) -- node[above] {$\scriptstyle f_2$} (B34);
\draw[->] (B34) -- node[above] {$\scriptstyle f_1$} (B35);

\draw[->] (B11) -- node[left] {$\scriptstyle f_1$} (B21);
\draw[->] (B21) -- node[left] {$\scriptstyle f_2$}(B31);
\draw[->] (B15) -- node[left] {$\scriptstyle f_2$}(B25);
\draw[->] (B25) -- node[left] {$\scriptstyle f_3$}(B35);
\draw[->] (B12) -- node[left] {$\scriptstyle f_1$}(B22);
\draw[->] (B24) -- node[left] {$\scriptstyle f_2$}(B34);


\draw[->] (C11) -- node[above] {$\scriptstyle f_3$} (C12);
\draw[->] (C12) -- node[above] {$\scriptstyle f_2$} (C13);
\draw[->] (C13) -- node[above] {$\scriptstyle f_1$} (C14);
\draw[->] (C14) -- node[above] {$\scriptstyle f_1$} (C15);

\draw[->] (C21) -- node[above] {$\scriptstyle f_3$} (C22);
\draw[->] (C22) -- node[above] {$\scriptstyle f_2$} (C23);
\draw[->] (C23) -- node[above] {$\scriptstyle f_2$} (C24);
\draw[->] (C24) -- node[above] {$\scriptstyle f_1$} (C25);

\draw[->] (C31) -- node[above] {$\scriptstyle f_3$} (C32);
\draw[->] (C32) -- node[above] {$\scriptstyle f_3$} (C33);
\draw[->] (C33) -- node[above] {$\scriptstyle f_2$} (C34);
\draw[->] (C34) -- node[above] {$\scriptstyle f_1$} (C35);

\draw[->] (C11) -- node[left] {$\scriptstyle f_1$} (C21);
\draw[->] (C21) -- node[left] {$\scriptstyle f_2$}(C31);
\draw[->] (C15) -- node[left] {$\scriptstyle f_2$}(C25);
\draw[->] (C25) -- node[left] {$\scriptstyle f_3$}(C35);
\draw[->] (C12) -- node[left] {$\scriptstyle f_1$}(C22);
\draw[->] (C24) -- node[left] {$\scriptstyle f_2$}(C34);


\draw (A11) ellipse [x radius=0.52cm, y radius=0.30cm];
\draw (B11) ellipse [x radius=0.52cm, y radius=0.30cm];
\draw (C11) ellipse [x radius=0.52cm, y radius=0.30cm];

\draw[ thick]
    (A11.20)
    to[out=25,in=155]
    node[midway, above=2pt] {$K_{2}^{-}$}
    (B11.160);

\draw[thick]
    (B11.20)
    to[out=25,in=155]
    node[midway, above=2pt] {$K_{3}^{+}$}
    (C11.160);
\end{tikzpicture}
}
\end{center}

 We now pass to the associated crystals of parsed maximal content words. Within each such crystal, the parsed maximal content words belonging to a single quasicrystal have the same underlying permutation; in the diagram below, we enclose each such quasicrystal in a box. Thus, collapsing each box to its underlying permutation gives the permutation index of that quasicrystal. Moreover, the elementary Knuth equivalences between plactically equivalent words in the diagram above correspond to elementary dual equivalences between the underlying permutations indexing the corresponding quasicrystals below.
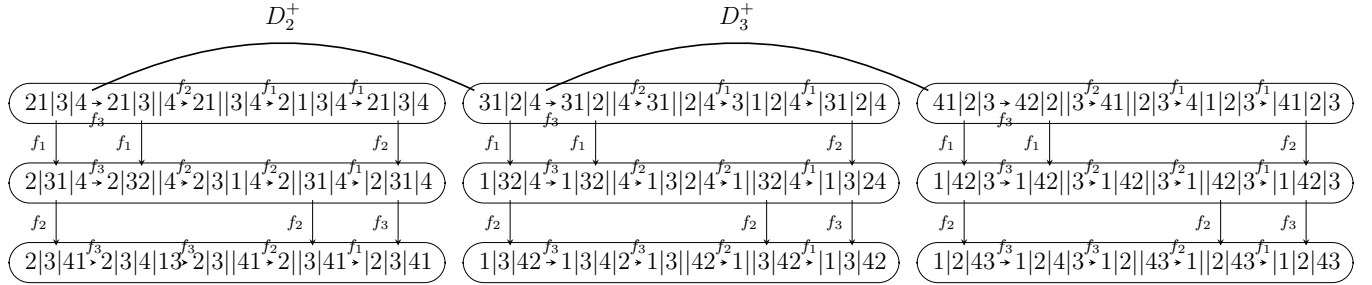
\begin{center}
\scalebox{0.78}{
\begin{tikzpicture}[
    node/.style={inner sep=2pt},
    >=stealth
]

\def\xsep{1.45}   
\def\ysep{1.35}   
\def\gsep{7.7}    

\def\Ashift{0}
\def\Bshift{\gsep}
\def\Cshift{2*\gsep}


\draw[rounded corners=10pt] ({\Ashift+0.45*\xsep},{-1*\ysep+0.34}) rectangle ({\Ashift+5.55*\xsep},{-1*\ysep-0.34});
\draw[rounded corners=10pt] ({\Ashift+0.45*\xsep},{-2*\ysep+0.34}) rectangle ({\Ashift+5.55*\xsep},{-2*\ysep-0.34});
\draw[rounded corners=10pt] ({\Ashift+0.45*\xsep},{-3*\ysep+0.34}) rectangle ({\Ashift+5.55*\xsep},{-3*\ysep-0.34});

\draw[rounded corners=10pt] ({\Bshift+0.45*\xsep},{-1*\ysep+0.34}) rectangle ({\Bshift+5.55*\xsep},{-1*\ysep-0.34});
\draw[rounded corners=10pt] ({\Bshift+0.45*\xsep},{-2*\ysep+0.34}) rectangle ({\Bshift+5.55*\xsep},{-2*\ysep-0.34});
\draw[rounded corners=10pt] ({\Bshift+0.45*\xsep},{-3*\ysep+0.34}) rectangle ({\Bshift+5.55*\xsep},{-3*\ysep-0.34});

\draw[rounded corners=10pt] ({\Cshift+0.45*\xsep},{-1*\ysep+0.34}) rectangle ({\Cshift+5.55*\xsep},{-1*\ysep-0.34});
\draw[rounded corners=10pt] ({\Cshift+0.45*\xsep},{-2*\ysep+0.34}) rectangle ({\Cshift+5.55*\xsep},{-2*\ysep-0.34});
\draw[rounded corners=10pt] ({\Cshift+0.45*\xsep},{-3*\ysep+0.34}) rectangle ({\Cshift+5.55*\xsep},{-3*\ysep-0.34});


\node[node] (A11) at ({\Ashift+1*\xsep},{-1*\ysep}) {$21|3|4$};
\node[node] (A12) at ({\Ashift+2*\xsep},{-1*\ysep}) {$21|3||4$};
\node[node] (A13) at ({\Ashift+3*\xsep},{-1*\ysep}) {$21||3|4$};
\node[node] (A14) at ({\Ashift+4*\xsep},{-1*\ysep}) {$2|1|3|4$};
\node[node] (A15) at ({\Ashift+5*\xsep},{-1*\ysep}) {$21|3|4$};

\node[node] (A21) at ({\Ashift+1*\xsep},{-2*\ysep}) {$2|31|4$};
\node[node] (A22) at ({\Ashift+2*\xsep},{-2*\ysep}) {$2|32||4$};
\node[node] (A23) at ({\Ashift+3*\xsep},{-2*\ysep}) {$2|3|1|4$};
\node[node] (A24) at ({\Ashift+4*\xsep},{-2*\ysep}) {$2||31|4$};
\node[node] (A25) at ({\Ashift+5*\xsep},{-2*\ysep}) {$|2|31|4$};

\node[node] (A31) at ({\Ashift+1*\xsep},{-3*\ysep}) {$2|3|41$};
\node[node] (A32) at ({\Ashift+2*\xsep},{-3*\ysep}) {$2|3|4|13$};
\node[node] (A33) at ({\Ashift+3*\xsep},{-3*\ysep}) {$2|3||41$};
\node[node] (A34) at ({\Ashift+4*\xsep},{-3*\ysep}) {$2||3|41$};
\node[node] (A35) at ({\Ashift+5*\xsep},{-3*\ysep}) {$|2|3|41$};



\node[node] (B11) at ({\Bshift+1*\xsep},{-1*\ysep}) {$31|2|4$};
\node[node] (B12) at ({\Bshift+2*\xsep},{-1*\ysep}) {$31|2||4$};
\node[node] (B13) at ({\Bshift+3*\xsep},{-1*\ysep}) {$31||2|4$};
\node[node] (B14) at ({\Bshift+4*\xsep},{-1*\ysep}) {$3|1|2|4$};
\node[node] (B15) at ({\Bshift+5*\xsep},{-1*\ysep}) {$|31|2|4$};

\node[node] (B21) at ({\Bshift+1*\xsep},{-2*\ysep}) {$1|32|4$};
\node[node] (B22) at ({\Bshift+2*\xsep},{-2*\ysep}) {$1|32||4$};
\node[node] (B23) at ({\Bshift+3*\xsep},{-2*\ysep}) {$1|3|2|4$};
\node[node] (B24) at ({\Bshift+4*\xsep},{-2*\ysep}) {$1||32|4$};
\node[node] (B25) at ({\Bshift+5*\xsep},{-2*\ysep}) {$|1|3|24$};

\node[node] (B31) at ({\Bshift+1*\xsep},{-3*\ysep}) {$1|3|42$};
\node[node] (B32) at ({\Bshift+2*\xsep},{-3*\ysep}) {$1|3|4|2$};
\node[node] (B33) at ({\Bshift+3*\xsep},{-3*\ysep}) {$1|3||42$};
\node[node] (B34) at ({\Bshift+4*\xsep},{-3*\ysep}) {$1||3|42$};
\node[node] (B35) at ({\Bshift+5*\xsep},{-3*\ysep}) {$|1|3|42$};



\node[node] (C11) at ({\Cshift+1*\xsep},{-1*\ysep}) {$41|2|3$};
\node[node] (C12) at ({\Cshift+2*\xsep},{-1*\ysep}) {$42|2||3$};
\node[node] (C13) at ({\Cshift+3*\xsep},{-1*\ysep}) {$41||2|3$};
\node[node] (C14) at ({\Cshift+4*\xsep},{-1*\ysep}) {$4|1|2|3$};
\node[node] (C15) at ({\Cshift+5*\xsep},{-1*\ysep}) {$|41|2|3$};

\node[node] (C21) at ({\Cshift+1*\xsep},{-2*\ysep}) {$1|42|3$};
\node[node] (C22) at ({\Cshift+2*\xsep},{-2*\ysep}) {$1|42||3$};
\node[node] (C23) at ({\Cshift+3*\xsep},{-2*\ysep}) {$1|42||3$};
\node[node] (C24) at ({\Cshift+4*\xsep},{-2*\ysep}) {$1||42|3$};
\node[node] (C25) at ({\Cshift+5*\xsep},{-2*\ysep}) {$|1|42|3$};

\node[node] (C31) at ({\Cshift+1*\xsep},{-3*\ysep}) {$1|2|43$};
\node[node] (C32) at ({\Cshift+2*\xsep},{-3*\ysep}) {$1|2|4|3$};
\node[node] (C33) at ({\Cshift+3*\xsep},{-3*\ysep}) {$1|2||43$};
\node[node] (C34) at ({\Cshift+4*\xsep},{-3*\ysep}) {$1||2|43$};
\node[node] (C35) at ({\Cshift+5*\xsep},{-3*\ysep}) {$|1|2|43$};


\draw[->] (A11) -- node[below] {$\scriptstyle f_3$} (A12);
\draw[->] (A12) -- node[above] {$\scriptstyle f_2$} (A13);
\draw[->] (A13) -- node[above] {$\scriptstyle f_1$} (A14);
\draw[->] (A14) -- node[above] {$\scriptstyle f_1$} (A15);

\draw[->] (A21) -- node[above] {$\scriptstyle f_3$} (A22);
\draw[->] (A22) -- node[above] {$\scriptstyle f_2$} (A23);
\draw[->] (A23) -- node[above] {$\scriptstyle f_2$} (A24);
\draw[->] (A24) -- node[above] {$\scriptstyle f_1$} (A25);

\draw[->] (A31) -- node[above] {$\scriptstyle f_3$} (A32);
\draw[->] (A32) -- node[above] {$\scriptstyle f_3$} (A33);
\draw[->] (A33) -- node[above] {$\scriptstyle f_2$} (A34);
\draw[->] (A34) -- node[above] {$\scriptstyle f_1$} (A35);

\draw[->] (A11) -- node[left] {$\scriptstyle f_1$} (A21);
\draw[->] (A21) -- node[left] {$\scriptstyle f_2$}(A31);
\draw[->] (A15) -- node[left] {$\scriptstyle f_2$}(A25);
\draw[->] (A25) -- node[left] {$\scriptstyle f_3$}(A35);
\draw[->] (A12) -- node[left] {$\scriptstyle f_1$}(A22);
\draw[->] (A24) -- node[left] {$\scriptstyle f_2$}(A34);


\draw[->] (B11) -- node[below] {$\scriptstyle f_3$} (B12);
\draw[->] (B12) -- node[above] {$\scriptstyle f_2$} (B13);
\draw[->] (B13) -- node[above] {$\scriptstyle f_1$} (B14);
\draw[->] (B14) -- node[above] {$\scriptstyle f_1$} (B15);

\draw[->] (B21) -- node[above] {$\scriptstyle f_3$} (B22);
\draw[->] (B22) -- node[above] {$\scriptstyle f_2$} (B23);
\draw[->] (B23) -- node[above] {$\scriptstyle f_2$} (B24);
\draw[->] (B24) -- node[above] {$\scriptstyle f_1$} (B25);

\draw[->] (B31) -- node[above] {$\scriptstyle f_3$} (B32);
\draw[->] (B32) -- node[above] {$\scriptstyle f_3$} (B33);
\draw[->] (B33) -- node[above] {$\scriptstyle f_2$} (B34);
\draw[->] (B34) -- node[above] {$\scriptstyle f_1$} (B35);

\draw[->] (B11) -- node[left] {$\scriptstyle f_1$} (B21);
\draw[->] (B21) -- node[left] {$\scriptstyle f_2$}(B31);
\draw[->] (B15) -- node[left] {$\scriptstyle f_2$}(B25);
\draw[->] (B25) -- node[left] {$\scriptstyle f_3$}(B35);
\draw[->] (B12) -- node[left] {$\scriptstyle f_1$}(B22);
\draw[->] (B24) -- node[left] {$\scriptstyle f_2$}(B34);


\draw[->] (C11) -- node[below] {$\scriptstyle f_3$} (C12);
\draw[->] (C12) -- node[above] {$\scriptstyle f_2$} (C13);
\draw[->] (C13) -- node[above] {$\scriptstyle f_1$} (C14);
\draw[->] (C14) -- node[above] {$\scriptstyle f_1$} (C15);

\draw[->] (C21) -- node[above] {$\scriptstyle f_3$} (C22);
\draw[->] (C22) -- node[above] {$\scriptstyle f_2$} (C23);
\draw[->] (C23) -- node[above] {$\scriptstyle f_2$} (C24);
\draw[->] (C24) -- node[above] {$\scriptstyle f_1$} (C25);

\draw[->] (C31) -- node[above] {$\scriptstyle f_3$} (C32);
\draw[->] (C32) -- node[above] {$\scriptstyle f_3$} (C33);
\draw[->] (C33) -- node[above] {$\scriptstyle f_2$} (C34);
\draw[->] (C34) -- node[above] {$\scriptstyle f_1$} (C35);

\draw[->] (C11) -- node[left] {$\scriptstyle f_1$} (C21);
\draw[->] (C21) -- node[left] {$\scriptstyle f_2$}(C31);
\draw[->] (C15) -- node[left] {$\scriptstyle f_2$}(C25);
\draw[->] (C25) -- node[left] {$\scriptstyle f_3$}(C35);
\draw[->] (C12) -- node[left] {$\scriptstyle f_1$}(C22);
\draw[->] (C24) -- node[left] {$\scriptstyle f_2$}(C34);

\draw[thick]
    (A11.20)
    to[out=25,in=155]
    node[midway, above=2pt] {$D_{2}^{+}$}
    (B11.160);

\draw[thick]
    (B11.20)
    to[out=25,in=155]
    node[midway, above=2pt] {$D_{3}^{+}$}
    (C11.160);
\end{tikzpicture}
}
\end{center}

In particular, the diagram contains three rows and three columns of quasicrystals. Each column is the quasicrystal decomposition of a fixed irreducible crystal of highest weight $\lambda=(2,1,1)$, while each row consists of plactically equivalent quasicrystals whose permutation indices form a dual equivalence graph. Collapsing each quasicrystal to its underlying permutation therefore produces three distinct, isomorphic dual equivalence graphs, one from each row. At the same time, the crystal operators relating quasicrystals within each column descend to relations among the permutation indices of these three dual equivalence graphs. Thus the collapsed diagram retains both the transverse dual equivalence structure and the vertical crystal structure exhibited by the original collection of quasicrystals.

\end{example}

\section{A Biword Proof for $h_{\mu} = \sum_{\lambda} K_{\lambda \mu} s_{\lambda}$\label{homogeneous} }

As an application of our biword approach that uses biwords with content words that are neither maximal nor minimal, we present the following elementary proof of a well-known result:
\begin{theorem} The multiplicity of the Schur function $s_\lambda$ in the expansion of a homogeneous symmetric function $h_\mu$ is $K_{\lambda \mu}$, the number of SSYT of shape $\lambda$ and content $\mu:$
    $$h_{\mu} = \sum_{\lambda} K_{\lambda \mu} s_{\lambda}.$$
\end{theorem}
\begin{proof}
   A term in the product expansion $h_\mu=h_{\mu_1}h_{\mu_2}\cdots$ is obtained by choosing, for each factor $h_{\mu_i}$, a sequence of $\mu_i$ variables whose indices are weakly increasing. Represent such a term by a word $\rho$ in which the variable indices chosen from each factor $h_{\mu_i}$ are written in weakly \emph{decreasing} order. For example, if $\mu=(5,3)$, the term $(x_1^2x_2x_3^2)(x_1^3)$ in the product expansion $h_5h_3$ has associated word $33211|111$, with the parsing bar separating the indices arising from the two factors.

For each such term and its associated word $\rho$, form the biword ${\kappa \choose \rho}$, where $\kappa=1^{\mu_1}2^{\mu_2}\cdots$; that is, place $1$'s over the first factor of $\rho$, $2$'s over the second factor, and so on. Thus, each term in the product expansion of $h_\mu$ determines a unique biword whose content row $\kappa$ has weight $\mu$. Note that this biword is typically neither minimal nor maximal (see Definitions~\ref{defn:minimal_biword} and~\ref{def:maximal-biword}).

Now apply RSK to the associated biword. As in Section~6, RSK associates to the biword a pair $(P,Q)$ of SSYT of the same shape $\lambda$. The content of $P$ is determined by the bottom row $\rho$, while the content of $Q$ is determined by the top row $\kappa$. Since $\wt(\kappa)=\mu$, $Q$ is an SSYT of shape $\lambda$ and content $\mu$.

Recall from Section~4.1 that the Schur function $s_\lambda$ is the character of the irreducible crystal $\mathcal B_\lambda$, and can be expressed as the sum of the monomials $x^{\wt(P)}$ over all SSYT $P$ of shape $\lambda$. Thus, for a fixed $Q$ of shape $\lambda$, the collection of $P$'s occurring in the RSK pairs $(P,Q)$ indexes precisely the terms of $s_\lambda$. The distinct tableaux $Q$ therefore index the distinct copies of $s_\lambda$ arising in this decomposition of $h_\mu$.

By RSK, each tableau $Q$ has the same shape $\lambda$ as $P$, while its content is determined by the content word $\kappa$, which by construction is $\mu$. Thus, the distinct copies of $s_\lambda$ are indexed precisely by the SSYT of shape $\lambda$ and content $\mu$. There are $K_{\lambda\mu}$ such tableaux, and therefore \[h_\mu=\sum_\lambda K_{\lambda\mu}s_\lambda.\]
\end{proof}

\bibliographystyle{alpha}
\bibliography{biblio}

\end{document}